\documentclass{article}

\def\version{October 9, 2012}

\nonstopmode

\usepackage{upgreek}
\usepackage{latexsym,bm,amssymb}

\usepackage{amsthm}
\usepackage{amsmath}
\usepackage{mathrsfs}
\usepackage{times}

\usepackage[colorlinks=false]{hyperref}
\hypersetup{urlcolor=black, citecolor=red}

\newcommand{\notyet}[1]{}

\DeclareSymbolFont{AMSb}{U}{msb}{m}{n}
\DeclareSymbolFontAlphabet{\mathbb}{AMSb}

\newcommand{\qm}{q}

\newcommand{\e}{{\bm e}}
\newcommand{\llongrightarrow}{-\!\!\!\!-\!\!\!\!-\!\!\!\!\longrightarrow\,}
\newcommand{\longlongrightarrow}{\,-\!\!\!-\!\!\!-\!\!\!-\!\!\!-\!\!\!-\!\!\!-\!\!\!-\!\!\!\longrightarrow\,}
\newcommand{\E}{\mathscr{X}}
\newcommand{\bS}{\mathcal{S}}

\newcommand{\xxd}{X}
\newcommand{\yyd}{Y}
\providecommand{\ttd}{T}
\renewcommand{\ttd}{T}

\newcommand{\dist}{\,{\rm dist}\,}

\newcommand{\supp}{\mathop{\rm supp}}
\newcommand{\ch}{\mathop{\rm c.h.}\!}
\newcommand{\supppi}{\mathop{\rm supp}\limits\sb{\mathbb{\,\!\!\!\mod\!\pi}}}

\newcommand{\p}{\partial}
\newcommand{\at}[1]{\vert\sb{\sb{#1}}}
\newcommand{\At}[1]{\Big\vert\sb{\sb{#1}}}

\def\Re{{\rm Re\, }}
\def\Im{{\rm Im\,}}
\providecommand{\C}{\mathbb{C}}
\renewcommand{\C}{\mathbb{C}}
\newcommand{\R}{\mathbb{R}}
\newcommand{\N}{\mathbb{N}}
\newcommand{\Z}{\mathbb{Z}}
\newcommand{\torus}{\mathbb{T}}
\newcommand{\Abs}[1]{\Big\vert#1\Big\vert}
\newcommand{\abs}[1]{\vert #1 \vert}

\newcommand{\norm}[1]{\Vert #1 \Vert}

\newcommand{\sothat}{\;{\rm :}\ }

\newcommand{\dotT}{\stackrel{\circ}{\mathbb{T}}}

\providecommand{\ltor}[1]{
\ifnum #1=1{\it i}\else\ifnum #1=2{\it ii}\else\ifnum #1=3{\it iii}
\else\ifnum #1=4 {\it iv}\fi\fi\fi\fi
}

\DeclareMathSymbol{\varGamma}{\mathord}{letters}{"00}
\DeclareMathSymbol{\varDelta}{\mathord}{letters}{"01}
\DeclareMathSymbol{\varSigma}{\mathord}{letters}{"06}
\DeclareMathSymbol{\varPhi}{\mathord}{letters}{"08}
\DeclareMathSymbol{\varOmega}{\mathord}{letters}{"0A}

\theoremstyle{plain}
\newtheorem{theorem}{Theorem}[section]

\newtheorem{lemma}[theorem]{Lemma}
\newtheorem{corollary}[theorem]{Corollary}
\newtheorem{proposition}[theorem]{Proposition}
\theoremstyle{definition}
\newtheorem{definition}[theorem]{Definition}
\newtheorem{assumption}[theorem]{Assumption}

\theoremstyle{remark}
\newtheorem{remark}[theorem]{Remark}

\makeatletter\@addtoreset{equation}{section}
\makeatother

\begin{document}

\title{Weak attractor
of the Klein-Gordon field in discrete space-time
interacting with a nonlinear oscillator}

\author{
{\sc Andrew Comech}
\\
{\it\small
Texas A\&M University,
College Station, TX 77843, USA}
\\
{\it\small
Institute for Information Transmission Problems, Moscow 101447, Russia}
}

\date{\version}

\maketitle






\begin{abstract}
We consider the $\mathbf{U}(1)$-invariant
nonlinear Klein-Gordon equation
in discrete space and discrete time,
which is the discretization 
of the nonlinear continuous Klein-Gordon equation.
To obtain this equation, we use
the energy-conserving finite-difference scheme of Strauss-Vazquez.
We prove that
each finite energy solution converges as $\ttd\to\pm\infty$
to the finite-dimensional
set of all multifrequency solitary wave solutions
with one, two, and four frequencies.
The components of the solitary manifold
corresponding to the solitary waves
of the first two types are generically two-dimensional,
while the
component corresponding to the last type
is generically four-dimensional.
The attraction
to the set of solitary waves
is caused by the nonlinear energy transfer
from lower harmonics to the continuous spectrum
and subsequent radiation.
For the proof, we develop
the well-posedness for the nonlinear wave equation
in discrete space-time,
apply the technique of quasimeasures,
and also obtain the version of the Titchmarsh convolution theorem
for distributions on the circle.
\end{abstract}

\bigskip

\hfill{\it To the memory of Boris Fedosov and Mark Vishik}

\section{Introduction}

In this paper we study
the long-time asymptotics for dispersive Hamiltonian systems.
The first results in this direction were obtained
by Segal \cite{MR0153967,MR0152908},
Strauss \cite{MR0233062},
and Morawetz and Strauss \cite{MR0303097},
who considered
the nonlinear scattering and local convergence to zero
for finite energy solutions
to nonlinear wave equations.
Apparently, there can be no such convergence to zero
when there are localized standing wave solutions;
in the case of $\mathbf{U}(1)$-invariant systems,
these solutions are solitary waves
of the form
$\phi(x)e^{-i\omega t}$,
with $\omega\in\R$ and $\phi$ decaying at infinity
(one could say, ``nonlinear Schr\"odinger eigenstates'').
In this case, one expects that generically
any finite energy solution
breaks into a superposition of
outgoing solitary waves and radiation;
the statement known as the Soliton Resolution Conjecture
(see \cite{MR2275691,MR2304091}).
The Soliton Resolution Conjecture
implies that any finite energy solution
locally
converges either to zero or to a solitary wave.
Thus, for a $\mathbf{U}(1)$-invariant
dispersive Hamiltonian system,
one expects that the weak attractor is formed
by the set of all solitary waves.
For a translation invariant system,
this implies that
the convergence to solitary waves
is to take place -- locally --
in any inertial reference frame.

Existence
of finite-dimensional attractors
(formed by static stationary states)
is extensively studied for \emph{dissipative systems},
such as
the Ginzburg-Landau,
the Kuramoto-Siva\-shin\-sky,
and the 2D forced Navier-Stokes equations,
where
the diffusive part of the equation damps higher frequencies
and in some cases leads to existence
of a finite-dimensional attractor \cite{MR1156492,MR1441312,MR1868930}.
Existence of attractors for finite difference approximations
of such dissipative systems,
as well as the relation between the attractors of continuous systems
and their approximations,
was considered in \cite{MR1049907,MR1092888,MR1124326}.

We are interested in extending these results
to the Hamiltonian systems,
where the convergence to a certain attracting set
(for both large positive and negative times)
takes place not because of the dissipation,
but instead due to the dispersion,
and thus takes place ``weakly'', in the weighted norms,
with perturbations dispersing
because of the local energy decay.
In
\cite{ubk-arma},
we considered a weak attractor
of the $\mathbf{U}(1)$-invariant nonlinear Klein-Gordon equation
in one dimension,
coupled to a nonlinear oscillator located at $x=0$:
\begin{equation}\label{nlw-0}
\p\sb t^2\psi(x,t)=\Delta\psi(x,t)
-m^2\psi(x,t)
-\delta(x)p(\abs{\psi(x,t)}^2)\psi(x,t),
\qquad
x\in\R,
\end{equation}
where 
$\psi(x,t)\in\C$
and $p(\cdot)$ is a potential with real coefficients,
with positive coefficient at the leading order term.
We proved
in \cite{ubk-arma}
that the attractor
of all finite energy solutions
is formed by the set of all solitary waves,
$\phi\sb\omega(x) e^{-i\omega t}$,
with
$\omega\in\R$ and $\phi\sb\omega\in H\sp 1(\R)$.
The general strategy
of the proof has been to consider
the omega-limit trajectories
and then to prove that each omega-limit
trajectory has a point spectrum,
and thus is a solitary wave.

In this paper,
we extend this result
to the finite difference approximation of the
$n$-dimensional Klein-Gordon equation
interacting with a nonlinear oscillator.
Our intention was to show that in the discrete
case, just as in the continuous one,
the attractor is formed by the set of solitary waves.
This turned out to be true, except that
in the discrete case,
besides usual one-frequency solitary waves,
the set of solitary wave solutions may contain
the two- and four-frequency components.
This is in agreement with our version
of the Titchmarsh convolution theorem for distributions
on the circle,
which we needed to develop to complete the argument.
These multifrequency solitary waves
disappear in the continuous limit.
To our knowledge,
this is the first result
on the weak attraction
for the Hamiltonian model
on discrete space-time.

The discretized models are widely studied
in applied mathematics
and in theoretical physics,
in part
due to atoms in a crystal forming a lattice,
in part due to some of these models
(such as the Ising model)
being exactly solvable.
Moreover, it is the discretized model
that is used in numerical simulations of the continuous 
Klein-Gordon equation.
The ground for considering the
energy-conserving difference schemes
for the nonlinear Klein-Gordon equations
and nonlinear wave equations
was set by Strauss and Vazquez in \cite{MR0503140}.
The importance of having
conserved quantities in the numerical scheme
was illustrated
by noticing that instability occurs
for the finite-difference schemes
which do not conserve the energy \cite{MR1047100}.
Let us mention that our approach is also applicable
to other energy-conserving finite-difference schemes
so long as there are a priori bounds
on the norm of the solution.
Such schemes have been constructed in \cite{MR1360462,MR1852556,SubCMAME}.


Our approach relies on
the well-posedness results and the a priori estimates
for the Strauss-Vazquez finite-difference scheme
which we developed in \cite{xt-discrete-rjmp}.
While the discrete energy for the Strauss-Vazquez scheme
given in \cite{MR0503140}
contained quadratic terms which in general
are not positive-definite,
we have shown \cite{xt-discrete-rjmp}
that the conserved discrete energy is positive-definite
under the condition
\begin{equation}\label{one-over-root-n}
\frac{\tau}{\varepsilon}
\le \frac{1}{\sqrt{n}}
\end{equation}
on the grid ratio,
where
$\varepsilon$ is the space step
with respect to each component of $x\in\R^n$
and $\tau$ is the time step;
the Strauss-Vazquez finite-difference scheme
with the grid ratio $\tau/\varepsilon=1/\sqrt{n}$
also preserves the discrete charge.
The positive-definiteness
of the conserved energy
provides one with the a priori energy estimates
and results in the stability of the finite-difference scheme.
(The relation
\eqref{one-over-root-n}
agrees with the stability criterion
in \cite{virieux-1986-889}.)
While the charge conservation
does not seem to be particularly important on its own,
it could be considered as an indication that
the $\mathbf{U}(1)$-invariance of the continuous equation
is in a certain sense compatible with the chosen discretization procedure.
See the discussion in \cite[Section 1]{MR1360462}.
We reproduce our results on the well-posedness
for the Strauss-Vazquez finite-difference scheme
in Appendix~\ref{sect-wp}.

There is another important feature
of our approach to the finite difference equation,
compared to the approach
which we developed in \cite{ubk-arma,ukr-hp,ukk-jmpa}
for the continuous case.
In the discrete case,
\emph{the spectral gap},
where the frequencies of the solitons are located
and where, as it turns out,
the spectrum of the omega-limit trajectory could be located,
consists of \emph{two} open neighborhoods
of the circle.
This does not allow us to apply the Titchmarsh convolution theorem
in a direct form 
as in \cite{ubk-arma,ukr-hp,ukk-jmpa}.
To circumvent this problem,
we derive a version of the Titchmarsh convolution theorem
for distributions on the circle;
see Appendix~\ref{sect-titchmarsh-circle}.
This version of the Titchmarsh convolution theorem
does not allow one to reduce the spectrum
of omega-limit trajectories to a single point;
we end up with the spectrum
consisting of one, two, and four frequencies.
Indeed such omega-limit trajectories exist;
we explicitly construct
solitary waves with one, two, and four frequencies.


\medskip

Here is the plan of the paper.
In Section~\ref{sect-disc},
we describe the model and state the main results.
In Section~\ref{sect-omega-limit},
we introduce the omega-limit trajectories
and describe the proof of the main result:
the convergence of any finite energy solution
to the set of solitary waves.
The main idea is that
such a convergence
is equivalent to showing that each omega-limit trajectory
itself is a solitary wave.
In Section~\ref{sect-dispersive},
we separate the dispersive part of the solution,
and consider the regularity
of the remaining part in Section \ref{sect-regularity}.
In Section~\ref{sect-spectral-relation}
we obtain the spectral relation
satisfied by the omega-limit trajectory.
For this, we use the technique
of quasimeasures, which we borrow from \cite{ubk-arma}.
In Section~\ref{sect-nsa}
we apply
to the spectral relation
our version of the Titchmarsh convolution theorem
on the circle,
proving that the spectrum
of any omega-limit trajectory
consists of finitely many frequencies.
This completes
the proof that each omega-limit trajectory is a (multifrequency) solitary wave.
We give an explicit construction of multifrequency solitary waves
in Section~\ref{sect-solitary-waves}.
Appendix~\ref{sect-wp}
gives
the well-posedness for the finite difference scheme approximation.
The versions of the Titchmarsh convolution theorem
for distributions on the circle are stated and proved
in Appendix~\ref{sect-titchmarsh-circle}.

\bigskip

\noindent
ACKNOWLEDGMENTS.
The author is grateful
to Alexander Komech for the suggestion
to consider
the discrete analog of the Klein-Gordon equation,
to Evgeny Gorin for interesting discussions,
to Juliette Chabassier and Patrick Joly
for the references and the preprint
of their paper \cite{SubCMAME}.
Special thanks to the anonymous referee for
pointing out the misprints.

\section{Definitions and main results}

\label{sect-disc}

In
\cite{ubk-arma},
we considered the weak attractor
of the $\mathbf{U}(1)$-invariant nonlinear Klein-Gordon equation
in one dimension,
coupled to a nonlinear oscillator located at $x=0$:
\begin{equation}\label{nlw}
\p\sb t^2\psi(x,t)=\Delta\psi(x,t)
-m^2\psi(x,t)
-\delta(x)W'(\abs{\psi(x,t)}^2)\psi(x,t),
\qquad
x\in\R,
\end{equation}
where 
$\psi(x,t)\in\C$
and $W(\cdot)$ is a real-valued
polynomial
which represents the potential energy of the oscillator:
\[
W(\abs{\psi}^2)
=C\sb 0\abs{\psi}^2
+C\sb 1\abs{\psi}^4
+\dots+C\sb p\abs{\psi}^{2(p+1)},
\qquad
p\ge 1,\quad C\sb p>0.
\]
Equation \eqref{nlw}
is a Hamiltonian system
with the Hamiltonian
\begin{equation}\label{def-energy-kg}
\mathscr{E}(\psi,\dot\psi)
=\frac{1}{2}\int\sb{\R}\big(
\abs{\dot\psi}^2
+\abs{\psi'}^2
+m^2\abs{\psi}^2
\big)\,dx
+\frac 1 2 W(\abs{\psi(0,t)}^2).
\end{equation}
In this paper, we will consider the discrete version of
equation \eqref{nlw}.
We pick $\varepsilon>0$ and $\tau>0$
and substitute the continuous variables
$(x,t)\in\R^n\times\R$
by
\[
x=\varepsilon \xxd,
\qquad
t=\tau \ttd,
\qquad
\mbox{where}
\quad
(\xxd,\,\ttd)\in\Z^n\times\Z.
\]

\begin{remark}
Note that
we may couple a nonlinear oscillator
to the Klein-Gordon field on the space-time lattice
in any dimension $n\ge 1$.
In the continuous case \cite{ubk-arma},
one can only consider the dimension $n=1$,
when the Sobolev estimates
(which we have due to the energy conservation)
ensure that the solution is continuous
as a function of $x$,
so that the nonlinear term
in \eqref{nlw}
is well-defined.
\end{remark}

From now on, we assume that
$(\xxd,\ttd)\in\Z^n\times\Z$
is a point on the space-time integer lattice.
Let $\psi\in l(\Z^n\times\Z,\C)$
be a complex-valued function
defined on this lattice.
We will indicate dependence
on the lattice points $\Z^n\times\Z$
by superscripts for the temporal dependence
and by subscripts for the spatial dependence,
so that
$\psi\sb{\xxd}\sp{\ttd}$
is the value of $\psi\in l(\Z^n\times\Z,\C)$
at the point $\xxd\in\Z^n$ at the moment $\ttd\in\Z$.

The Strauss-Vazquez finite-difference scheme
\cite{MR0503140}
applied to equation \eqref{nlw} takes the form
\begin{equation}\label{dkg-c}
\frac{1}{\tau^2}D\sb{\ttd}^2\psi\sb{\xxd}\sp{\ttd}
=
\frac{1}{\varepsilon^2}\bm{D}\sb{\xxd}^2\psi\sb{\xxd}\sp{\ttd}
-
m^2
\frac{\psi\sb{\xxd}\sp{\ttd+1}+\psi\sb{\xxd}\sp{\ttd-1}}{2}
+
\delta\sb{\xxd,0}
f\sp{\ttd},
\qquad
\xxd\in\Z^n,\quad\ttd\in\Z,
\end{equation}
where the nonlinear term is given by
\begin{equation}\label{def-f}
f\sp{\ttd}:=
\left\{
\begin{array}{l}
-\frac{W(\abs{\psi\sb{0}\sp{\ttd+1}}^2)-W(\abs{\psi\sb{0}\sp{\ttd-1}}^2)}
{\abs{\psi\sb{0}\sp{\ttd+1}}^2-\abs{\psi\sb{0}\sp{\ttd-1}}^2}
\cdot \frac{\psi\sb{0}\sp{\ttd+1}+\psi\sb{0}\sp{\ttd-1}}{2}
\qquad
\mbox{if}
\quad
\abs{\psi\sb{0}\sp{\ttd+1}}\ne\abs{\psi\sb{0}\sp{\ttd-1}};
\\
-W'(\abs{\psi\sb{0}\sp{\ttd+1}}^2)
\cdot \frac{\psi\sb{0}\sp{\ttd+1}+\psi\sb{0}\sp{\ttd-1}}{2}
\qquad
\mbox{if}
\quad
\abs{\psi\sb{0}\sp{\ttd+1}}=\abs{\psi\sb{0}\sp{\ttd-1}}.
\end{array}
\right.
\end{equation}
In \eqref{dkg-c},
we used the notations
\begin{equation}\label{def-ddot-delta}
D\sb{\ttd}^2\psi\sb{\xxd}\sp{\ttd}
=\psi\sb{\xxd}\sp{\ttd+1}-2\psi\sb{\xxd}\sp{\ttd}+\psi\sb{\xxd}\sp{\ttd-1},
\qquad
\bm{D}\sb{\xxd}^2\psi\sb{\xxd}\sp{\ttd}
=
\sum\sb{j=1}\sp{n}
\big(\psi\sb{\xxd+\e\sb j}\sp{\ttd}-2\psi\sb{\xxd}\sp{\ttd}+\psi\sb{\xxd-\e\sb j}\sp{\ttd}\big),
\end{equation}
with
\begin{equation}\label{def-ej}
\e\sb 1=(1,0,0,0,\dots)\in\Z^n,
\qquad
\e\sb 2=(0,1,0,0,\dots)\in\Z^n,
\qquad
\mbox{etc.}
\end{equation}

\begin{remark}
By the little B\'ezout theorem,
\[
\frac{W(\abs{\psi\sb{0}\sp{\ttd+1}}^2)-W(\abs{\psi\sb{0}\sp{\ttd-1}}^2)}
{\abs{\psi\sb{0}\sp{\ttd+1}}^2-\abs{\psi\sb{0}\sp{\ttd-1}}^2}
\]
is a polynomial of
$\abs{\psi\sb{0}\sp{\ttd\pm 1}}^2$.
This polynomial
coincides with the second line in \eqref{def-f}
when
$\abs{\psi\sb{0}\sp{\ttd-1}}
=\abs{\psi\sb{0}\sp{\ttd+1}}$.
\end{remark}

\begin{assumption}\label{ass-alpha-n}
$
\displaystyle
\frac{\tau}{\varepsilon}
=\frac{1}{\sqrt{n}}.
$
\end{assumption}

\begin{assumption}\label{ass-wp}
$
W(\lambda)=\sum\sb{q=0}\sp{p}C\sb{q}\lambda^{q+1},
$
where
$
\ p\in\N$,
$
\ C\sb{q}\in\R
$
for
$0\le q\le p$,
and
$
\ C\sb{p}>0.
$
\end{assumption}

We introduce the phase space
\begin{equation}
\E=l^2(\Z^n)\times l^2(\Z^n),
\qquad
\norm{(u,v)}\sb{\E}^2
=\norm{u}\sb{l^2}^2+\norm{v}\sb{l^2}^2,
\end{equation}
where
\[
\norm{u}\sb{l^2}^2
=\sum\sb{\xxd\in\Z^n}\abs{u\sb{\xxd}}^2,
\qquad
u\in l^2(\Z^n).
\]

We will denote by $l(\Z,\E)$
the space of functions of $\ttd\in\Z$ with values in $\E$.

\begin{definition}[Discrete energy]
The energy of the function 
$\psi\in l(\Z,\E)$
at the moment $\ttd\in\Z$
is
\begin{eqnarray}\label{def-energy-t}
&
E\sp{\ttd}
=
\sum\sb{\xxd\in\Z^n}
\varepsilon^n\Big[
\sum\sb{j=1}\sp{n}
\frac{\abs{\psi\sb{\xxd}\sp{\ttd+1}-\psi\sb{\xxd-\e\sb j}\sp{\ttd}}^2
+\abs{\psi\sb{\xxd}\sp{\ttd+1}-\psi\sb{\xxd+\e\sb j}\sp{\ttd}}^2
}{4n\tau^2}
+
m^2\frac{\abs{\psi\sb{\xxd}\sp{\ttd+1}}^2+\abs{\psi\sb{\xxd}\sp{\ttd}}^2}{4}
\Big]
\nonumber
\\
&
+
\frac{W(\abs{\psi\sb{0}\sp{\ttd+1}}^2)+W(\abs{\psi\sb{0}\sp{\ttd}}^2)}{4}.
\end{eqnarray}
\end{definition}

\begin{remark}
In the case $n=1$
the continuous limit
of the energy $E$ in \eqref{def-energy-t}
coincides with the classical
energy functional of the Klein-Gordon equation
interacting with an oscillator
described by the potential $W$; see \eqref{def-energy-kg}.
\end{remark}

We consider the Cauchy problem
\begin{equation}\label{dkg-cp}
\left\{
\begin{array}{l}
D\sb{\ttd}^2\psi
-
\frac{1}{n}\bm{D}\sb{\xxd}^2\psi
+
\tau^2 m^2
\frac{\psi\sb{0}\sp{\ttd+1}+\psi\sb{0}\sp{\ttd-1}}{2}
=\tau^2\delta\sb{\xxd,0}f\sp{\ttd},
\qquad
\xxd\in\Z^n,
\quad
\ttd\ge 1,
\\
\\
(\psi\sp{\ttd},\psi\sp{\ttd+1})\at{\ttd=0}=(u\sp{0},u\sp{1})\in\E=l^2(\Z^n)\times l^2(\Z^n),
\end{array}
\right.
\end{equation}
where $f\sp{\ttd}$ is defined by \eqref{def-f}.

\begin{theorem}[Well-posedness]
\label{theorem-gwp-dkg}
Let $n\in\N$.
\begin{enumerate}
\item
There is $\tau\sb 0>0$
such that
for any $\tau\in(0,\tau\sb 0)$
and for all $(u\sp{0},u\sp{1})\in\E$
the Cauchy problem \eqref{dkg-cp}
has a unique solution
$\psi\in l\sp\infty(\Z,l^2(\Z^n))$.

\item
The value of the energy functional 
is conserved:
$
E\sp{\ttd}=E\sp{0},
$
$
\ttd\in\Z.
$
\item
$\psi\sb{\xxd}\sp{\ttd}$
satisfies the \emph{a priori} estimate
\begin{equation}\label{l2-bound}
\varepsilon^n\norm{\psi\sp{\ttd}}\sb{l^2}^2
\le\frac{4}{m^2}
\big[E^0-\inf\sb{\lambda\ge 0}W(\lambda)\big],
\qquad
\ttd\in\Z.
\end{equation}

\item
For each $\ttd\in\Z$,
the map
$U(\ttd):(u\sp{0},u\sp{1})\mapsto(\psi\sp{\ttd},\psi\sp{\ttd+1})$
is continuous
in $\E$.
\item
For each $\ttd\in\Z$,
the map
$U(\ttd):(u\sp{0},u\sp{1})\mapsto(\psi\sp{\ttd},\psi\sp{\ttd+1})$
is weakly continuous in $\E$.
\end{enumerate}
\end{theorem}

The main part of this theorem
(all the statements but the last one)
is proved in \cite{xt-discrete-rjmp};
we reproduce this proof in Appendix~\ref{sect-wp}
(the existence, uniqueness,
and continuous dependence on continuous data
are proved in Theorems~\ref{theorem-e}, \ref{theorem-u},
and~\ref{theorem-pol},
and the a priori estimates are proved in Theorem~\ref{theorem-a-priori}).
Let us mention that the bound \eqref{l2-bound}
follows from the energy conservation
since the first term in \eqref{def-energy-t} is nonnegative,
while $\inf\sb{\lambda\ge 0}W(\lambda)>-\infty$
due to Assumption~\ref{ass-wp}.

Let us sketch the proof of the weak continuity of $U(T)$,
which is the last statement of the theorem.
Let $\Psi\sb j\in\E$,
$j\in\N$,
be a sequence in $\E$
weakly convergent to some $\Psi\in\E$.
By the Banach-Steinhaus theorem,
$\Psi\sb j$, $j\in\N$, are uniformly bounded in $\E$;
so are $U(\ttd)(\Psi\sb j)$.
Now one can see that the weak convergence
of
$U(\ttd)(\Psi\sb j)$
to $U(\ttd)(\Psi)$
in $\E$
follows from the continuity of $U(\ttd)$ in $\E$,
from the finite speed of propagation
(the value
$\psi\sb{\xxd}\sp{\ttd}$
only depends on the initial data
$(u\sb{\yyd}\sp{0},u\sb{\yyd}\sp{1})$
in the ball
$\abs{\yyd}\le \abs{\xxd}+\abs{\ttd}$),
and from the convergence
$\Psi\sb j\to \Psi$
in the topology of $l^2(B)$
for any bounded set $B\subset\Z^n$.

We will use the standard notation
\[
\torus:=\R\mod 2\pi.
\]

\begin{definition}[Solitary waves]
\label{def-solitary-waves}
\begin{enumerate}
\item
The solitary waves of equation \eqref{dkg-c}
are solutions of the form
\begin{equation}\label{solitary-waves}
\psi\sp{\ttd}
=\sum\sb{k=1}\sp{N}
\phi\sb{k} e^{-i\omega\sb{k} \ttd},
\qquad
\ttd\in\Z,
\qquad
{\rm where}\quad
\phi\sb{k}\in l^2(\Z^n),
\quad
\omega\sb{k}\in\torus.
\end{equation}
\item
The solitary manifold is the set
\begin{equation}\label{solitary-manifold}
\bS
=
\left\{
\big(\psi\sp{\ttd},\psi\sp{\ttd+1})\at{\ttd=0}\right\}\subset
\E=l^2(\Z^n)\times l^2(\Z^n),
\end{equation}
where $\psi\sp{\ttd}$
are solitary wave solutions to \eqref{dkg-c}
of the form
\eqref{solitary-waves}.
\end{enumerate}
\end{definition}
The set $\bS$ is nonempty since
$\phi e^{-i\omega\ttd}$
with
$\phi\equiv 0$ is formally a solitary wave
corresponding to any $\omega\in\torus$. 

Define
\begin{equation}\label{def-omega-m}
\omega\sb m:=\arccos\Big(1+\frac{\tau^2 m^2}{2}\Big)^{-1}.
\end{equation}

\begin{assumption}\label{ass-m-is-small}
$\omega\sb m<\frac{\pi}{4(p+1)}$,
where $p\in\N$ is defined in 
Assumption~\ref{ass-wp}.
\end{assumption}

This assumption is needed so that
the Titchmarsh convolution theorem
for distributions on the circle
(see Theorem~\ref{theorem-titchmarsh-circle} below)
will be applicable for the analysis of
omega-limit trajectories.
One can see from \eqref{def-omega-m}
that
for any fixed $m>0$
Assumption~\ref{ass-m-is-small} is satisfied
as long as the time step $\tau>0$ is sufficiently small.

\medskip

Our main result is that the weak attractor
of all finite energy solutions to \eqref{dkg-c} 
coincides with the solitary manifold $\bS$.

\begin{theorem}[Solitary manifold as the weak attractor]
\label{main-theorem-dkg}
Assume that $\tau\in(0,\tau\sb 0)$,
with $\tau\sb 0>0$ as in Theorem~\ref{theorem-gwp-dkg}.
Let Assumptions~\ref{ass-wp} and~\ref{ass-m-is-small} be satisfied.
Then:
\begin{enumerate}
\item
For any initial data
$(u\sp{0},u\sp{1})\in\E$
the solution to the Cauchy problem
\eqref{dkg-cp}
weakly converges to $\bS$
as $\ttd\to\pm\infty$.
\item
The frequencies of the solitary waves \eqref{solitary-waves}
satisfy
$\omega\in\varOmega\sb 0\cup\varOmega\sb\pi$
($\omega\sb k\in\overline{\varOmega\sb 0}\cup\overline{\varOmega\sb\pi}$
if $n\ge 5$),
where the spectral gaps
$\varOmega\sb 0$ and $\varOmega\sb\pi$
are defined by
\begin{equation}\label{def-spectral-gaps}
\varOmega\sb 0=(-\omega\sb m,\omega\sb m)\subset\torus,
\qquad
\varOmega\sb\pi=(\pi-\omega\sb m,\pi+\omega\sb m)\subset\torus,
\end{equation}
with $\omega\sb m=\arccos\big(1+\frac{\tau^2 m^2}{2}\big)^{-1}$.
\item
The set of all solitary wave solutions
consists of solutions of the following three types:
\begin{enumerate}
\item
One-frequency solitary waves of the form
\begin{equation}
\psi\sb\xxd\sp\ttd
=
\phi\sb\xxd
e^{-i\omega\ttd},
\qquad
\xxd\in\Z^n,
\quad
\ttd\in\Z,
\end{equation}
with
$\phi\in l^2(\Z^n)$.
The corresponding component of the solitary manifold
is generically two-dimensional.
\item
Two-frequency solitary waves of the form
\begin{equation}
\psi\sb\xxd\sp\ttd
=
\big(1+(-1)^{\ttd+\Lambda\cdot\xxd}\sigma\big)\phi\sb\xxd
e^{-i\omega\ttd},
\qquad
\xxd\in\Z^n,
\quad
\ttd\in\Z,
\end{equation}
with
$\sigma\in\{\pm 1\}$,
$\Lambda=(1,\dots,1)\in\Z^n$,
and
$\phi\in l^2(\Z^n)$.
The corresponding component of the solitary manifold
is generically two-dimensional.
\item
Four-frequency solitary waves of the form
\begin{equation}
\psi\sb\xxd\sp\ttd
=
\big(1+(-1)^{\ttd+\Lambda\cdot\xxd}\big)\phi\sb\xxd
e^{-i\omega\ttd}
+
\big(1-(-1)^{\ttd+\Lambda\cdot\xxd}\big)\theta\sb\xxd
e^{-i\omega'\ttd},
\qquad
\xxd\in\Z^n,
\quad
\ttd\in\Z,
\end{equation}
with
$\Lambda=(1,\dots,1)\in\Z^n$
and
$\phi,\,\theta\in l^2(\Z^n)$.
The corresponding component of the solitary manifold
is generically four-dimensional.
\end{enumerate}
\end{enumerate}
\end{theorem}

Definition \ref{def-solitary-waves}
and Theorem~\ref{main-theorem-dkg}
show that the set $\bS$
satisfies the following two properties:
\begin{enumerate}
\item
It is invariant under the evolution
described by equation \eqref{dkg-c}.
\item
It is the smallest set
to which all finite energy solutions converge weakly.
\end{enumerate}

It follows that $\bS$ is the weak attractor
of all finite energy solutions to \eqref{dkg-c}.


\begin{remark}
The convergence
of any finite energy solution
to $\bS$,
stated in Theorem~\ref{main-theorem-dkg},
also holds in certain normed spaces.
For example, let $s>0$;
for $u\in l^2(\Z^n)$,
denote
$\norm{u}\sb{l^2\sb{-s}}
=\big(\sum\sb{\xxd\in\Z^n}\abs{u\sb\xxd}^2(1+\xxd^2)^{-s}\big)^{1/2}$,
and
for $(u,v)\in\E$,
denote
$\norm{(u,v)}\sb{\E\sb{-s}}
=(\norm{u}\sb{l^2\sb{-s}}+\norm{v}\sb{l^2\sb{-s}})^{1/2}$.
Then,
for any finite energy solution, one has
\begin{equation}\label{attraction}
\dist\sb{\E\sb{-s}}
\big(
(\psi\sp{\ttd},\psi\sp{\ttd+1}),\bS
\big)
\mathop{\longrightarrow}\limits\sb{\ttd\to\pm\infty}0,
\end{equation}
where
$\dist\sb{\E\sb{-s}}(\Psi,\bS)=\inf\sb{\Phi\in\bS}
\dist\sb{\E\sb{-s}}(\Psi,\Phi)$.
This follows from the fact that
the weak convergence in $\E$
implies the strong convergence in $\E\sb{-s}$,
for any $s>0$.
\end{remark}

\section{Omega-limit trajectories}
\label{sect-omega-limit}

Here we explain our approach to the proof of Theorem 
\ref{main-theorem-dkg}.
Since the equation is time-reversible,
it suffices to prove the theorem
for $\ttd\to+\infty$.
The following notion of  omega-limit trajectory
plays the crucial role in our approach.

\begin{definition}
[Omega-limit points and omega-limit trajectories]
\label{def-omega-pt}
\ \begin{enumerate}
\item
$(z\sp 0, z\sp 1)\in\E$ is an omega-limit point
of $\psi\in l(\Z,l^2(\Z^n))$
if there is a sequence $\ttd\sb j\to+\infty$
such that
$(\psi\sp{\ttd\sb j},\psi\sp{\ttd\sb j+1})$
converges to
$(z\sp 0, z\sp 1)$
weakly in $\E$.
We denote the set of all omega-limit points
of $\psi\in l(\Z,l^2(\Z^n))$ by $\upomega(\psi)$.
%
\item
An omega-limit trajectory of the solution $\psi\sp{\ttd}\sb{\xxd}$
to \eqref{dkg-cp}
is a solution $\zeta\sb{\xxd}\sp{\ttd}$
to the discrete nonlinear Klein-Gordon equation
\eqref{dkg-c},
\begin{equation}\label{kgz}
D\sb{\ttd}^2\zeta\sp{\ttd}\sb{\xxd}
-\frac{1}{n}\bm{D}\sb{\xxd}^2\zeta\sp{\ttd}\sb{\xxd}
+\tau^2 m^2\frac{\zeta\sb{\xxd}\sp{\ttd+1}+\zeta\sb{\xxd}\sp{\ttd-1}}{2}
=\tau^2\delta\sb{\xxd,0}g\sp{\ttd},\qquad
\xxd\in \Z^n,\quad \ttd\in\Z,
\end{equation}
where
(Cf. \eqref{def-f})
\begin{equation}\label{def-g}
g\sp{\ttd}
=
\left\{
\begin{array}{l}
-
\frac{W(\abs{\zeta\sb{0}\sp{\ttd+1}}^2)-W(\abs{\zeta\sb{0}\sp{\ttd-1}}^2)}
{\abs{\zeta\sb{0}\sp{\ttd+1}}^2-\abs{\zeta\sb{0}\sp{\ttd-1}}^2}
\frac{\zeta\sb{0}\sp{\ttd+1}+\zeta\sb{0}\sp{\ttd-1}}{2}
\quad
\mbox{if}
\quad
\abs{\zeta\sb{0}\sp{\ttd+1}}^2\ne \abs{\zeta\sb{0}\sp{\ttd-1}}^2,
\\
-
W'(\abs{\zeta\sb{0}\sp{\ttd+1}}^2)
\frac{\zeta\sb{0}\sp{\ttd+1}+\zeta\sb{0}\sp{\ttd-1}}{2}
\quad
\mbox{if}
\quad
\abs{\zeta\sb{0}\sp{\ttd+1}}^2=\abs{\zeta\sb{0}\sp{\ttd-1}}^2,
\end{array}
\right.
\qquad \ttd\in\Z,
\end{equation}
with the initial data
at an omega-limit point
of $\psi\sp{\ttd}\sb{\xxd}$:
\[
(\zeta\sp\ttd,\zeta\sp{\ttd+1})\at{\ttd=0}
=(z\sp 0, z\sp 1)\in\upomega(\psi).
\]
\end{enumerate}
\end{definition}

\begin{lemma}
If
$(z\sp 0, z\sp 1)=\mathop{\mbox{\rm w-lim}}\limits\sb{\ttd\sb j\to\infty}
(\psi\sp{\ttd\sb j},\psi\sp{\ttd\sb j+1})$
(weakly in $\E$)
is an omega-limit point of $\psi\in l(\Z,l^2(\Z^n))$
and if
$\zeta\in l(\Z,l^2(\Z^n))$
is the omega-limit trajectory
with
\[
(\zeta\sp{\ttd},\zeta\sp{\ttd+1})\at{\ttd=0}=(z\sp 0, z\sp 1),
\]
then $\psi\sp{\ttd\sb j+\ttd}\to\zeta\sp{\ttd}$,
weakly in $l^2(\Z^n)$,
and in particular there is the convergence
\begin{equation}\label{psij-to-zeta}
\zeta\sp{\ttd}\sb{\xxd}
=\lim\sb{\ttd\sb j\to\infty}\psi\sp{\ttd\sb j+\ttd}\sb{\xxd},
\qquad \xxd\in\Z^n,
\quad
\ttd\in\Z.
\end{equation}
\end{lemma}

\begin{proof}
This immediately follows from the weak continuity
of $U(\ttd)$ stated in Theorem~\ref{theorem-gwp-dkg}.
\end{proof}


We will deduce Theorem~\ref{main-theorem-dkg}
from the following proposition.

\begin{proposition}\label{proposition-sw}
Under the conditions of  Theorem~\ref{main-theorem-dkg},
any omega-limit trajectory $\zeta\sp{\ttd}\sb{\xxd}$
of some finite energy solution
is a
solitary wave
of one of the following types:
\begin{enumerate}
\item
$\zeta\sp\ttd\sb\xxd=\phi e^{-i\omega\ttd}$,
with
$\phi\in l^2(\Z^n)$
and
$\omega\in\varOmega\sb 0\cup\varOmega\sb\pi$
($\omega\in\overline{\varOmega\sb 0}\cup\overline{\varOmega\sb\pi}$
if $n\ge 5$);
\item
$\zeta\sb\xxd\sp\ttd=
(1+(-1)\sp{\ttd+\Lambda\cdot\xxd}\sigma)
\phi\sb\xxd e^{-i\omega\ttd}$,
with
$\sigma\in\{\pm 1\}$,
$\phi\in l^2(\Z^n)$,
and
$\omega\in\varOmega\sb 0$
($\omega\in\overline{\varOmega\sb 0}$
if $n\ge 5$);
\item
$\zeta\sp\ttd\sb\xxd
=
(1+(-1)\sp{\ttd+\Lambda\cdot\xxd}))\phi\sb\xxd e^{-i\omega\ttd}
+
(1-(-1)\sp{\ttd+\Lambda\cdot\xxd}))\theta\sb\xxd e^{-i\omega'\ttd}$,
with
$\phi,\,\theta\in l^2(\Z^n)$
and
$\omega,\,\omega'\in\varOmega\sb 0$
($\omega,\,\omega'\in\overline{\varOmega\sb 0}$
if $n\ge 5$).
\end{enumerate}
\end{proposition}

\medskip

Proposition~\ref{proposition-sw}
can be used to complete the proof of Theorem~\ref{main-theorem-dkg},
as follows.

\begin{proof}[Proof of Theorem~\ref{main-theorem-dkg}]
Let $\ttd\sb j\in\N$, $j\in\N$, be a sequence
such that $\ttd\sb j\to+\infty$.
By the Banach-Alaoglu theorem,
a priori bounds on $\norm{(\psi\sp{\ttd},\psi\sp{\ttd+1})}\sb{\E}$
stated in Theorem~\ref{theorem-gwp-dkg}
allow us to choose a subsequence
$\{\ttd\sb{j\sb r}\sothat r\in\N\}$
such that
\begin{equation}\label{gte}
(\psi\sp{\ttd\sb{j\sb r}},\psi\sp{\ttd\sb{j\sb r}+1})
\mathop{\llongrightarrow}\limits\sb{r\to\infty}
(z\sp 0, z\sp 1)\in\E,
\qquad
\mbox{weakly in}\ \E.
\end{equation}
Let $\zeta\in l(\Z,\E)$
be the corresponding \emph{omega-limit trajectory},
that is,
the solution to the Cauchy problem \eqref{dkg-cp}
with the initial data
$(\zeta\sp{\ttd},\zeta\sp{\ttd+1})\at{\ttd=0}=(z\sp 0, z\sp 1)$.
By Proposition~\ref{proposition-sw},
$\zeta\sp{\ttd}$ is a solitary wave.
It follows that
$(z\sp 0, z\sp 1)=(\zeta\sp{\ttd},\zeta\sp{\ttd+1})\at{\ttd=0}\in\bS$.
Thus, the first two statements of Theorem~\ref{main-theorem-dkg}
follow from Proposition~\ref{proposition-sw}.

\medskip

Let us prove the
last statement of Theorem~\ref{main-theorem-dkg},
namely, that
the set of all solitary waves only consists
of one-, two-, and four-frequency solitary waves.
It will follow from Proposition~\ref{proposition-sw}
if we can show that each solitary wave solution
is itself (its own) omega-limit trajectory,
has to be of the form specified by
Proposition~\ref{proposition-sw}.

\begin{lemma}\label{lemma-35}
Let $\psi\sp{\ttd}=\sum\sb{k=1}\sp{N}\phi\sb k e^{-i\omega\sb k\ttd}$,
with $\phi\sb k\in l^2(\Z^n)$, $\omega\sb k\in\torus$.
Then
\[
(\psi\sp\ttd,\psi\sp{\ttd+1})\at{\ttd=0}\in\upomega(\psi),
\]
so that
$\psi\sp\ttd$
is the omega-limit trajectory of itself.
\end{lemma}

\begin{proof}
Pick any sequence
$\ttd\sb j\in\N$,
$\ttd\sb j\to\infty$
such that
$\omega\sb 1\ttd\sb j\to 0\in\torus$
as $j\to\infty$.
Then either $\{\omega\sb 2\ttd\sb j\sothat j\in\N\}$
is dense in $\torus$,
or it is a subset of
$\{\frac{k\pi}{q}\in\torus\sothat k\in\Z\sb{2q}\}$, for some $q\in\N$.
In the former case,
we take a subsequence $\ttd\sb j'$ of $\ttd\sb j$
such that
$\omega\sb 2\ttd\sb j'\to 0$ in $\torus$;
In the latter case,
we consider a new sequence, $\ttd\sb j'=q\ttd\sb j$,
so that $\omega\sb 2\ttd\sb j'=0$
(and we still have $\omega\sb 1\ttd\sb j'
=q\omega\sb 1\ttd\sb j\to 0\in\torus$).


Repeating this process, we
end up with a sequence such that
$\omega\sb k \ttd\sb j\to 0\in\torus$
as $j\to\infty$
for all $1\le k\le N$.
It follows that
$(\psi\sp{\ttd\sb j},\psi\sp{\ttd\sb j+1})
\mathop{\longrightarrow}\limits\sp{\E}\sb{j\to\infty}
(\psi\sp 0,\psi\sp 1)$.
By Definition~\ref{def-omega-pt},
$(\psi\sp 0,\psi\sp 1)$
is the omega-limit point of $\psi\sp\ttd$;
that is, $(\psi\sp 0,\psi\sp 1)\in\upomega(\psi)$.
Hence,
$\psi\sp\ttd$ itself is an omega-limit trajectory
of a finite energy solution,
and has to be of one of the three types mentioned
in Proposition~\ref{proposition-sw}.
\end{proof}

The dimension of the components of the solitary manifold
corresponding to one-, two-, and four-frequency solitary waves 
are computed in
Lemma~\ref{lemma-one-frequency}, Lemma~\ref{lemma-two-frequencies},
and Lemma~\ref{lemma-four-frequencies} below.

This finishes the proof of Theorem~\ref{main-theorem-dkg}.
\end{proof}

It remains to prove Proposition~\ref{proposition-sw},
which is the contents of in the remaining part of the paper.
We will prove it
analyzing the spectrum of omega-limit trajectories.
Everywhere below, we suppose that the conditions
of Proposition~\ref{proposition-sw}
(that is, conditions of Theorem~\ref{main-theorem-dkg}) hold.

\section{Separation of the dispersive component}

\label{sect-dispersive}

We rewrite \eqref{dkg-cp}
as a linear nonhomogeneous equation
\begin{equation}\label{leq}
\begin{array}{l}
(\bm{A}\psi)\sb{\xxd}\sp{\ttd}:=  D\sb{\ttd}^2\psi\sb{\xxd}\sp{\ttd}
-\frac{1}{n}\bm{D}\sb{\xxd}^2\psi\sb{\xxd}\sp{\ttd}
+\tau^2 m^2\frac{\psi\sb{\xxd}\sp{\ttd+1}+\psi\sb{\xxd}\sp{\ttd-1}}{2}
=\tau^2\delta\sb{\xxd,0}f\sp\ttd,
\quad
(\xxd,\ttd)\in\Z^n\times\Z,
\\ \\
(\psi\sp{\ttd},\psi\sp{\ttd+1})\at{\ttd=0}=(u\sp 0,u\sp 1)\in\E
=(l^2(\Z^n)\times l^2(\Z^n)),
\end{array}
\end{equation}
where $f\sp{\ttd}$ is given by \eqref{def-f}.

Let $a(\xi,\omega)$ be the symbol of the operator
$\bm{A}$ in the left-hand side of \eqref{leq}:
\begin{equation}\label{sym}
a(\xi,\omega):=
(2+\tau^2 m^2)\cos\omega-\frac{2}{n}\sum\sb{j=1}\sp{n}\cos\xi\sb j,
\qquad
\xi\in\torus^n,
\quad
\omega\in\torus.
\end{equation}
For a fixed value of $\omega\in\torus$,
the dispersion relation
\begin{equation}\label{Ar0}
a(\xi,\omega)=0,
\qquad
\xi\in\torus^n,\quad\omega\in\torus,
\end{equation}
admits a solution $\xi\in\torus^n$ if and only if 
$\abs{\cos\omega}\le \big(1+\frac{\tau^2 m^2}{2}\big)^{-1}$,
or, equivalently,
when $\omega$ belongs to the continuous spectrum
$\varOmega_c$
of the linear discrete equation \eqref{leq},
which is given by
\begin{equation}\label{cs}
\varOmega_c
=
\torus\setminus(\varOmega\sb 0\cup\varOmega\sb\pi);
\end{equation}
the spectral gaps
$\varOmega\sb 0$ and $\varOmega\sb\pi$
have been defined in \eqref{def-spectral-gaps}.

Note that due to the factor $\frac{1}{n}$
in \eqref{leq}
the continuous spectrum
does not depend on the dimension $n$.


\begin{lemma}\label{lemma-finite-norm}
If $n\le 4$,
the expression
$\frac{1}{a(\xi,\omega)}$
is of finite $L^2$-norm in $\xi\in\torus^n$
if and only if
$\omega\in\varOmega\sb 0\cup\varOmega\sb\pi$.
If $n\ge 5$,
$\frac{1}{a(\xi,\omega)}$
is of finite $L^2$-norm in $\xi\in\torus^n$
if and only if
$\omega\in\overline{\varOmega\sb 0}\cup\overline{\varOmega\sb\pi}$.
\end{lemma}

\begin{proof}
When $\omega\in\varOmega_c\setminus\p\varOmega_c$,
$a(\xi,\omega)$ vanishes on the (singular) hypersurface in $\torus^n$,
then
$\frac{1}{a(\xi,\omega)}\not\in L^2(\torus^n)$.
If $\omega\in\varOmega\sb 0\cup\varOmega\sb\pi$,
then
$a(\xi,\omega)$
does not vanish for any $\xi\in\torus^n$,
hence
$\frac{1}{a(\xi,\omega)}$
is of finite $L^2$-norm in $\xi\in\torus^n$.

For $\omega\sb b\in\p\varOmega_c=\{\pm\omega\sb m,\pi\pm\omega\sb m\}$,
$a(\xi,\omega\sb b)$ vanishes at the points
$\xi=(0,\dots,0)\in\torus^n$ and $\xi=(\pi,\dots,\pi)\in\torus^n$.
The consideration is the same in the neighborhoods of
both of these points:
near $\xi=(0,\dots,0)\in\torus^n$,
one has
$\frac{1}{a(\xi,\omega\sb b)}=\frac{1}{\frac 1 n\xi^2+o(\xi^2)}$;
therefore,
for $n\le 4$,
one has
$\frac{1}{a(\xi,\omega\sb b)}\not\in L^2(\torus^n)$;
for $n\ge 5$,
$\frac{1}{a(\xi,\omega\sb b)}\in L^2(\torus^n)$.
\end{proof}

For $\ttd\ge 0$,
we decompose
the solution to \eqref{leq}
into
$\psi\sb{\xxd}\sp{\ttd}=\chi\sb{\xxd}\sp{\ttd}+\varphi\sb{\xxd}\sp{\ttd}$,
where $\chi$, $\varphi$
satisfy the equations
\begin{equation}\label{eq1}
(\bm{A}\chi)\sp\ttd\sb\xxd=0
\qquad\qquad
\xxd\in\Z^n,
\qquad
\ttd\ge 1;
\qquad
(\chi\sp{\ttd},\chi\sp{\ttd+1})\at{\ttd=0}
=(\psi\sp{0},\psi\sp{1}),
\end{equation}
\begin{equation}\label{eq2}
(\bm{A}\varphi)\sp\ttd\sb\xxd
=
\tau^2\delta\sb{\xxd,0}f\sp{\ttd},
\quad \xxd\in\Z^n,
\qquad
\ttd\ge 1;
\qquad
(\varphi\sp{\ttd},\varphi\sp{\ttd+1})\at{\ttd=0}=(0,0),
\end{equation}
with $f\sp{\ttd}$ from \eqref{def-f}.
Note that both
$\chi$ and $\varphi$
are only defined for $\ttd\ge 0$.
Due to the energy conservation
(Theorem~\ref{theorem-gwp-dkg}),
which certainly also takes place for the linear equation,
the component $\chi$ is bounded in time:
\begin{equation}\label{chib}
\sup\sb{\ttd\in\Z\sb{+}}\norm{\chi\sp{\ttd}}\sb{l^2(\Z^n)}<\infty.
\end{equation}
Moreover, $\chi$ is purely dispersive in the following sense.

\begin{lemma}\label{lemma-pd}
For any bounded $B\subset\Z^n$,
\begin{equation}\label{chi0}
\norm{\chi\sp{\ttd}}\sb{l^2(B)}\to 0,
\qquad
{\abs{\ttd}\to\infty}.
\end{equation}
\end{lemma}

\begin{proof}
The solution $\chi\sp{\ttd}$
to \eqref{eq1}
is given by
\begin{equation}\label{solchi}
\chi\sb{\xxd}\sp{\ttd}
=
\int\sb{\torus^n}
\Big(
e^{i(\xi\cdot \xxd+\omega(\xi)\ttd)}\mathcal{P}\sb{+}(\xi)
+e^{i(\xi\cdot \xxd-\omega(\xi)\ttd)}\mathcal{P}\sb{-}(\xi)
\Big)
\,\frac{d\xi}{(2\pi)^n},
\end{equation}
where
$\omega(\xi)$ is the unique solution to the dispersion relation
\eqref{Ar0}
which satisfies
$\omega(\xi)\in[\omega\sb m,\pi-\omega\sb m]\subset(0,\pi)$
(Cf. \eqref{cs}),
while functions $\mathcal{P}\sb{\pm}(\xi)$ 
are determined from the initial data $(u\sp 0,u\sp 1)$
in \eqref{dkg-cp}
by
\[
\hat u\sp 0(\xi)=\mathcal{P}\sb{+}(\xi)+\mathcal{P}\sb{-}(\xi),
\qquad
\hat u\sp 1(\xi)=e^{i\omega(\xi)}\mathcal{P}\sb{+}(\xi)+
e^{-i\omega(\xi)}\mathcal{P}\sb{-}(\xi).
\]
Since
$
\det
\left[
\begin{array}{cc}
1&1\\e^{i\omega(\xi)}&e^{-i\omega(\xi)}
\end{array}
\right]
=-2i\sin\omega(\xi)
$,
with
$
\inf\sb{\xi\in\torus^n}\abs{\sin\omega(\xi)}=\sin\omega\sb m>0,
$
one can see that
$\mathcal{P}\sb{\pm}\in L^2(\torus^n)$,
and moreover there is $C_{\tau,m}<\infty$
independent on $u\sp 0,\,u\sp 1$
such that
\begin{equation}\label{pl2}
\norm{\mathcal{P}\sb{\pm}}\sb{L^2(\torus^n)}
\le
C_{\tau,m}
\big(\norm{u\sp{0}}\sb{l^2(\Z^n)}
+\norm{u\sp{1}}\sb{l^2(\Z^n)}\big).
\end{equation}
Further,
\[
\abs{\nabla\sb\xi\big(\xi\cdot \xxd\pm \omega(\xi)\ttd\big)}
=
\abs{\xxd\pm \nabla\omega(\xi)\ttd},
\]
where $\nabla\omega(\xi)$
can be determined by differentiating the dispersion relation
\eqref{Ar0}:
\[
\nabla\omega(\xi)
=\frac{1}{\big(1+\frac{\tau^2 m^2}{2}\big)
\sin\omega(\xi)}(\sin\xi\sb 1,\,\dots,\,\sin\xi\sb n),
\qquad
\xi\in\torus^n,
\]
where $\sin\omega(\xi)\ge\sin\omega\sb m>0$.
Thus, $\nabla\omega(\xi)$
vanishes only at the discrete set of points
$\xi\in\{0;\pi\}^n\subset\torus^n$.
Hence,
for any $\delta>0$ we can choose
a $\delta$-neighborhood $U\sb\delta$ of   
the set
$\{0;\pi\}^n\subset\torus^n$ such that for any
bounded $B\subset\Z^n$
there is $\ttd\sb{\delta,B}>0$ and $c\sb{\delta,B}>0$
such that
\begin{equation}\label{derpf}
\abs{\nabla\sb{\xi}(\xi\cdot \xxd\pm\omega(\xi)\ttd)}
\ge c\sb{\delta,B}\abs{\ttd},
\qquad
\xi\not\in U\sb\delta,
\qquad
\xxd\in B,
\quad
\abs{\ttd}\ge \ttd\sb{\delta,B}.
\end{equation}

Let us fix a bounded set $B\subset\Z^n$.

Pick an arbitrary $\epsilon>0$.
We choose $\delta>0$ sufficiently small
and split the initial data $\mathcal{P}\sb{\pm}$ into
$\mathcal{P}\sb{\pm}(\xi)=\mathcal{R}\sb{\pm}(\xi)+\mathcal{S}\sb{\pm}(\xi)$,
so that
$\norm{\mathcal{R}\sb{\pm}}\sb{L^2(\torus^n)}<\epsilon/3$
while
$\mathcal{S}\sb{\pm}(\xi)$ are smooth and supported outside
a $\delta$-neighborhood
$U_\delta$ of $\{0;\pi\}^n\subset\torus^n$.
Substituting the splitting
$\mathcal{P}\sb{\pm}=\mathcal{R}\sb{\pm}+\mathcal{S}\sb{\pm}$
into \eqref{solchi}, we have
$\chi\sb{\xxd}\sp{\ttd}=\rho\sb{\xxd}\sp{\ttd}+\sigma\sb{\xxd}\sp{\ttd}$,
where
\begin{equation}
\rho\sb{\xxd}\sp{\ttd}
=
\int\sb{\torus^n}
\Big(
e^{i(\xi\cdot \xxd+\omega(\xi)\ttd)}\mathcal{R}\sb{+}(\xi)
+e^{i(\xi\cdot \xxd-\omega(\xi)\ttd)}\mathcal{R}\sb{-}(\xi)
\Big)
\,\frac{d\xi}{(2\pi)^n},
\end{equation}
\begin{equation}\label{solsigma}
\sigma\sb{\xxd}\sp{\ttd}
=
\int\sb{\torus^n}
\Big(
e^{i(\xi\cdot \xxd+\omega(\xi)\ttd)}\mathcal{S}\sb{+}(\xi)
+e^{i(\xi\cdot \xxd-\omega(\xi)\ttd)}\mathcal{S}\sb{-}(\xi)
\Big)
\,\frac{d\xi}{(2\pi)^n}.
\end{equation}
Due to our choice of $\mathcal{R}\sb\pm$, one has
\[
\norm{\rho\sp{\ttd}}\sb{l^2(\Z^n)}
\le
\norm{\mathcal{R}\sb{-}}\sb{L^2(\torus^n)}
+\norm{\mathcal{R}\sb{+}}\sb{L^2(\torus^n)}
\le 2\epsilon/3,
\qquad
\ttd\in\Z.
\]
Using \eqref{derpf}
to integrate by parts in \eqref{solsigma},
one proves that
$\lim\sb{\abs{\ttd}\to\infty}\norm{\sigma\sp{\ttd}}\sb{l^2(B)}=0$,
so that
$\norm{\sigma\sp{\ttd}}\sb{l^2(B)}<\epsilon/3$
for sufficiently large $\abs{\ttd}$.
It follows that
\[
\norm{\chi\sp{\ttd}}\sb{l^2(B)}
\le
\norm{\rho\sp{\ttd}}\sb{l^2}
+
\norm{\sigma\sp{\ttd}}\sb{l^2(B)}<\epsilon
\]
for sufficiently large $\abs{\ttd}$.
Since $\epsilon>0$ is arbitrary,
we conclude that
\[
\lim\sb{\abs{\ttd}\to\infty}\norm{\chi\sp{\ttd}}\sb{l^2(B)}=0.
\]
\end{proof}

\section{Regularity on the continuous spectrum}
\label{sect-regularity}

Now we consider the equation on $\varphi\sp\ttd$;
see \eqref{eq2}.
Let us recall that
$f\sp{\ttd}$ is defined by \eqref{def-f} for $\ttd\in\Z$,
but is only considered in \eqref{eq2}
for $\ttd\ge 1$.
The function $\varphi\sb{\xxd}\sp{\ttd}$ is defined by \eqref{eq2}
for $\ttd\ge 0$
(with $\varphi\sb{\xxd}\sp 0=\varphi\sb{\xxd}\sp 1=0$).
We extend $f\sp{\ttd}$ and $\varphi\sp{\ttd}$
by zeros for $\ttd\le 0$,
so that
\begin{equation}\label{f-phi-extended}
f\sp{\ttd}=0
\quad\mbox{for}\quad \ttd\le 0,
\qquad
\varphi\sp{\ttd}=0
\quad\mbox{for}\quad \ttd\le 1.
\end{equation}
Then equation \eqref{eq2}
is satisfied for all $\ttd\in\Z$:
\begin{equation}\label{eq22}
(\bm{A}\varphi)\sp\ttd\sb\xxd
:=
D\sb{\ttd}^2\varphi\sb{\xxd}\sp{\ttd}
-\frac{1}{n}\bm{D}\sb{\xxd}^2\varphi\sb{\xxd}\sp{\ttd}
+\tau^2 m^2\frac{\varphi\sb{\xxd}\sp{\ttd+1}+\varphi\sb{\xxd}\sp{\ttd-1}}{2}
=\tau^2\delta\sb{\xxd,0}f\sp{\ttd},
\quad
\xxd\in\Z^n,
\quad
\ttd\in\Z.
\end{equation}
We introduce the Fourier transforms
\begin{eqnarray}\label{def-phi-tilde}
{\tilde\varphi}\sb{\xxd}(\omega)
&:=&
\sum\sb{\ttd\in\N}e^{i\omega \ttd}\varphi\sb{\xxd}\sp{\ttd},
\qquad\qquad\qquad
\quad
\xxd\in\Z^n,
\qquad
\omega\in\torus;
\\
\label{def-phi-tilde-hat}
\hat{\tilde\varphi}(\xi,\omega)
&:=&
\sum\sb{\xxd\in\Z^n, \ttd\in\N}
e^{-i\xi\cdot \xxd+i\omega \ttd}\varphi\sb{\xxd}\sp{\ttd},
\qquad
\xi\in\torus^n,
\qquad\omega\in\torus;
\\
\label{def-f-tilde}
\tilde f(\omega)
&:=&
\sum\sb{\ttd\in\N}e^{i\omega \ttd}f\sp{\ttd},
\qquad\qquad\qquad
\omega\in\torus.
\end{eqnarray}
Since in the above relations
the summation is over $\ttd\in\N$,
we can extend
\eqref{def-phi-tilde}--\eqref{def-f-tilde}
to the upper half-plane
as analytic functions of $\omega\in\C\sp{+}$.
Then equation \eqref{eq2} yields
\begin{equation}\label{leqF}
a(\xi,\omega)\hat{\tilde\varphi}(\xi,\omega)
=\tau^2\tilde f(\omega),
\qquad
\xi\in\torus^n,
\qquad
\omega\in\C\sp{+}\!\!\!\!\mod 2\pi,
\end{equation}
where $a(\xi,\omega)$ is defined by extending
\eqref{sym} to $\omega\in\C$:
\begin{equation}\label{sym-c}
a(\xi,\omega):=
(2+\tau^2 m^2)\cos\omega-\frac{2}{n}\sum\sb{j=1}\sp{n}\cos\xi\sb j,
\qquad
\xi\in\torus^n,
\quad
\omega\in\C.
\end{equation}

\begin{lemma}\label{lemma-nonzero}
For 
$\xi\in\torus^n$, $\omega\in\C\setminus\varOmega_c$,
one has
$a(\xi,\omega)\ne 0$.
\end{lemma}

\begin{proof}
Recall that
$\varOmega_c$ was defined by \eqref{cs}
so that $a(\xi,\omega)\ne 0$
for $\xi\in\torus^n$, $\omega\in\torus\setminus\varOmega_c$.
For $\omega\in\C\setminus\R$,
one has
\[
a(\xi,\omega)
=
(2+\tau^2 m^2)
\big(\cos\Re\omega\,\cosh\Im\omega
-i\sin\Re\omega\,\sinh\Im\omega\big)
-\frac{2}{n}\sum\sb{j=1}\sp{n}\cos\xi_j.
\]
Now it suffices to notice that
if $\Im\omega\ne 0$ and $\Re\omega\notin\{0;\pi\}$,
then $\Im a(\xi,\omega)\ne 0$,
while for $\Re\omega\in\{0;\pi\}$
one has $\abs{\Re a(\xi,\omega)}\ge\tau^2 m^2$.
\end{proof}

By Lemma~\ref{lemma-nonzero},
for $\omega$ away from $\varOmega_c$,
equation \eqref{leqF} yields
\begin{equation}\label{varphi-f-a}
\hat{\tilde\varphi}(\xi,\omega)=\frac
{\tau^2\tilde f(\omega)}{a(\xi,\omega)},
\qquad
\xi\in\torus^n,
\qquad
\omega\in(\C\sp{+}\!\!\!\!\mod 2\pi)\setminus\varOmega_c.
\end{equation}
For $\omega\in(\C\sp{+}\!\!\!\!\mod 2\pi)\setminus\varOmega_c$,
since
$\inf\sb{\xi\in\torus^n}\abs{a(\xi,\omega)}
=\min\sb{\xi\in\torus^n}\abs{a(\xi,\omega)}>0$,
the operator of multiplication by
$1/a(\xi,\omega)$
is a bounded linear operator from $L^2(\torus^n)$ to $L^2(\torus^n)$.
In the coordinate representation,
\eqref{varphi-f-a} can be written as
\begin{equation}\label{varphi-f-a-xi}
\tilde\varphi\sb\xxd(\omega)
=
\big(\bm{R}(\omega)\big[\tau^2\delta\sb{\yyd,0}\tilde f(\omega)\big]\big)\sb\xxd,
\qquad
\omega\in(\C\sp{+}\!\!\!\!\mod 2\pi)\setminus\varOmega_c,
\qquad
\xxd,\,\yyd\in\Z^n,
\end{equation}
where
$\bm{R}(\omega)$
is a bounded linear operator
\begin{equation}\label{def-resolvent}
\bm{R}(\omega)
=\mathscr{F}^{-1}
\circ
\frac{1}{a(\xi,\omega)}
\circ
\mathscr{F}
:\;l^2(\Z^n)\to l^2(\Z^n),
\quad
\omega\in(\C\sp{+}\!\!\!\!\mod 2\pi)\setminus\varOmega_c,
\end{equation}
with the Fourier transform
$\mathscr{F}:\,l^2(\Z^n)\to L^2(\torus^n)$
and its inverse
given by
\[
(\mathscr{F} u)(\xi)
=\sum\sb{\xxd\in\Z^n}e^{-i\xi\cdot\xxd}u\sb\xxd,
\qquad
(\mathscr{F}^{-1} v)\sb\xxd
=
\int\sb{\torus^n}e^{i\xi\cdot\xxd}v(\xi)\,
\frac{d\xi}{(2\pi)^n}.
\]
The expression in the right-hand side of \eqref{varphi-f-a-xi}
can be written as
\begin{equation}\label{leqFsc}
\big(\bm{R}(\omega)\big[
\tau^2\delta\sb{\yyd,0}\tilde f(\omega)
\big]\big)\sb{\xxd}
=
\sum\sb{\yyd\in\Z^n}\mathcal{G}\sb{\xxd-\yyd}(\omega)
\tau^2\delta\sb{\yyd,0}\tilde f(\omega)
=
\mathcal{G}\sb{\xxd}(\omega)
\tau^2\tilde f(\omega),
\end{equation}
where
$
\xxd\in\Z^n,
$
$
\omega\in(\C\sp{+}\!\!\!\!\mod 2\pi)\setminus\varOmega_c,
$
and $\mathcal{G}\sb{\xxd}(\omega)$
stands for the fundamental solution,
which is the inverse Fourier transform
of $1/a(\xi,\omega)$:
\begin{equation}\label{def-varsigma-c}
\mathcal{G}\sb{\xxd}(\omega)
=
\mathscr{F}^{-1}
\big[
\frac{1}{a(\xi,\omega)}
\big](\xxd)
=
\int\sb{\torus^n}
\frac{e^{i\xi\cdot \xxd}}
{
(2+\tau^2 m^2)\cos\omega
-\frac{2}{n}\sum\sb{j=1}\sp{n}\cos\xi\sb j
}
\,\frac{d\xi}{(2\pi)^n},
\end{equation}
where
$\omega\in(\C\sp{+}\!\!\!\!\mod 2\pi)\setminus\varOmega_c$.

Let us study properties of $\mathcal{G}\sb{\xxd}(\omega)$.
We start by introducing the set of the singular points
in the continuous spectrum $\varOmega_c$:
\begin{equation}\label{def-singular-points}
\Sigma
:=\left\{\pm\arccos
\frac{1-\frac{2l}{n}}{1+\frac{\tau^2 m^2}{2}}
\sothat\ l=0,\,1,\,\dots,\,n\right\}
\subset\varOmega_c
\subset\torus.
\end{equation}
These frequencies correspond to the
critical points of the symbol $a(\xi,\omega)$,
that is, for $\omega\in\Sigma$,
there is $\xi\in\torus^n$
such that
both $a(\xi,\omega)=0$ and $\nabla\sb\xi a(\xi,\omega)=0$.
(The relation $\nabla\sb\xi a(\xi,\omega)=0$ implies that
$\cos\xi\sb j=\pm 1$ for all $1\le j\le n$;
the value of $l$
in \eqref{def-singular-points}
is the number of cosines
in the denominator of \eqref{def-varsigma-c}
which are equal to $-1$.)
Note that in the one-dimensional case
one has $\Sigma=\p\varOmega_c=\{\pm\omega\sb m;\,\pi\pm\omega\sb m\}$,
with $\omega\sb m=\arccos\big(1+\frac{\tau^2m^2}{2}\big)^{-1}$.

\begin{lemma}\label{lemma-lap}
\begin{enumerate}
\item
For each fixed value of $\xxd\in\Z^n$, the function $\mathcal{G}\sb{\xxd}(\omega)$
is analytic in $\omega\in(\C\sp{+}\!\!\!\!\mod 2\pi)\setminus\varOmega_c$
and admits the trace
\begin{equation}\label{lap}
\mathcal{G}\sb{\xxd}(\omega):=
\mathcal{G}\sb{\xxd}(\omega+i 0),
\qquad
\omega\in\torus\setminus\Sigma,
\end{equation}
which is a smooth function of $\omega\in\torus\setminus\Sigma$.
\item
For
$\xxd\in\Z^n$,
$\omega\in\varOmega_0\cup\varOmega_\pi$
(for
$\omega\in\overline{\varOmega_0}\cup\overline{\varOmega_\pi}$
if $n\ge 5$),
there are relations
\begin{equation}\label{ggg}
\mathcal{G}\sb\xxd(-\omega)=\mathcal{G}\sb\xxd(\omega),
\qquad
\mathcal{G}\sb\xxd(\omega+\pi)
=-(-1)^{\Lambda\cdot\xxd}\mathcal{G}\sb\xxd(\omega),
\end{equation}
where $\Lambda=(1,\dots,1)\in\Z^n$.
\item
\begin{enumerate}
\item
If $n\le 4$,
$\mathcal{G}\sb 0(\omega)$
is a monotonically increasing function
on the interval $0\le\omega<\omega_m$,
where it is strictly positive,
and on the interval
$\pi-\omega_m<\omega\le\pi$
where it is strictly negative.
\item
If $n\ge 5$,
$\mathcal{G}\sb 0(\omega)$
is a monotonically increasing function
on the interval $0\le\omega\le\omega_m$,
where it is strictly positive,
and on the interval
$\pi-\omega_m\le\omega\le\pi$,
where it is strictly negative.
\end{enumerate}
\item
For every $\xxd\in\Z^n$, the boundary trace $\mathcal{G}\sb{\xxd}(\omega+i0)$
is a multiplier in
$\mathscr{D}'(\torus\setminus\Sigma)$,
the space of distributions on $\torus\setminus\Sigma$.
\item
For any
closed set
$I\subset\torus\setminus(\overline{\varOmega\sb 0}\cup\overline{\varOmega\sb\pi})$,
there is
$c\sb I>0$
such that
\begin{equation}\label{nv-large}
\norm{\mathcal{G}(\omega+i\epsilon)}\sb{l^2(\Z^n)}^2
\ge
\frac{c\sb I}{\epsilon},
\qquad
\mbox{for all}\quad\epsilon\in(0,1),
\quad
\omega\in I.
\end{equation}
\end{enumerate}
\end{lemma}

\begin{proof}
{\it 1.} The analyticity follows directly from 
\eqref{def-varsigma-c} and the definition \eqref{cs}
of $\varOmega_c$.
The continuity of the traces for $\omega\in\torus\setminus\Sigma$ 
follows by the Sokhotsky-Plemelj formula
after taking the symbol $a(\xi,\omega)$ in
\eqref{def-varsigma-c} as a new coordinate function 
in a neighborhood 
of the hypersurface $\Gamma_\omega:=\{\xi\in\torus^n: a(\xi,\omega)=0\}$;
see e.g. \cite{MR0211027,MR1914696}.
The smoothness follows similarly;
see e.g. \cite[Proposition 2.2]{MR2604564}.

\smallskip

\noindent
{\it 2.}
For
$\xxd\in\Z^n$,
$\omega\in\varOmega_0\cup\varOmega_\pi$,
we have:
\begin{eqnarray}
\mathcal{G}\sb\xxd(\omega+\pi)
&=&
\int\sb{\torus^n}
\frac{e^{i(\xi+\pi\Lambda)\cdot\xxd}}{a(\xi+\pi\Lambda,\omega+\pi)}
\frac{d\xi}{(2\pi)^n}
\nonumber
\\
&=&
-(-1)^{\Lambda\cdot\xxd}\int\sb{\torus^n}
\frac{e^{i\xi\cdot\xxd}}{a(\xi,\omega)}
\frac{d\xi}{(2\pi)^n}
=-(-1)^{\Lambda\cdot\xxd}\mathcal{G}\sb\xxd(\omega).
\nonumber
\end{eqnarray}

\smallskip

\noindent
{\it 3.}
The monotonicity
of $\mathcal{G}\sb 0(\omega)$
for $\omega\in\torus\setminus\varOmega_c$
immediately follows from the definition \eqref{def-varsigma-c}.
For $n\ge 5$,
one notices that
$\mathcal{G}\sb 0(\omega)$
remains finite
at $\omega=\pm\omega\sb m$
and $\omega=\pi\pm\omega\sb m$



\smallskip

\noindent
{\it 4.}
To prove that $\mathcal{G}\sb\xxd(\omega+i0)$
is a multiplier
in the space of distributions,
it suffices to notice that
for each $\xxd\in\Z^n$
the trace
$\mathcal{G}\sb{\xxd}(\omega+i 0)$
is a smooth function of
$\omega\in\torus\setminus\Sigma$.

\smallskip

\noindent
{\it 5.}
Using
\eqref{def-varsigma-c}
and the Plancherel theorem,
we compute:
\begin{eqnarray}\label{denomi}
&&
\norm{\mathcal{G}(\omega+i\epsilon)}\sb{l^2(\Z^n)}^2
=
\sum\sb{\xxd\in\Z^n}
\abs{\mathcal{G}\sb{\xxd}(\omega+i\epsilon)}^2
=
\int\sb{\torus^n}
\frac{1}{\abs{a(\xi,\omega+i\epsilon)}^2}
\frac{d\xi}{(2\pi)^n}
\\
&&
\hskip -6pt
\ge
\int\sb{\torus^n}
\frac{1}{\Abs{
(2+\tau^2 m^2)(\cos\omega\cosh\epsilon-i\sin\omega\sinh\epsilon)
-\frac{2}{n}\sum\limits\sb{j=1}\sp{n}\cos\xi\sb j
}^2}
\frac{d\xi}{(2\pi)^n}
\nonumber
\\
&&
\hskip -6pt
\ge
\!\!\!\!\!\!
\int\limits\sb{U\sb{\epsilon}(\Gamma_\omega)}
\frac{1
}
{
\Abs{
(2+\tau^2 m^2)\cos\omega\cosh\epsilon
-\frac{2}{n}\!\!\sum\limits\sb{j=1}\sp{n}\cos\xi\sb j}^2
+\abs{(2+\tau^2 m^2)\sin\omega\sinh\epsilon
}^2}
\frac{d\xi}{(2\pi)^n}
,
\nonumber
\end{eqnarray}
where
$U_\epsilon(\Gamma_\omega)\subset\torus^n$
is the $\epsilon$-neighborhood of
the 
level set 
$\Gamma\sb{\omega}\subset\torus^n$
which is defined by the dispersion relation
\[
\Gamma\sb{\omega}
=\Big\{\xi\in\torus^n\sothat
\ 
a(\xi,\omega)=0
\Big\}.
\]
Note that since
$\omega\in I\subset
\torus\setminus(\overline{\varOmega\sb 0}\cup\overline{\varOmega\sb\pi})$,
one has $\big(1+\frac{\tau^2 m^2}{2}\big)\abs{\cos\omega}<1$,
therefore $\Gamma_\omega$ is a nonempty submanifold of $\torus^n$
of codimension one,
piecewise smooth away from the discrete set $\{0;\pi\}^n\subset\torus^n$
(on this set,
one could have
simultaneously $a(\xi,\omega)=0$, $\nabla\sb\xi a(\xi,\omega)=0$;
in fact, $\Gamma\sb{\omega}$
does not contain these points for $\omega\in\torus\setminus \Sigma$).
It follows that
$\abs{U_\epsilon(\Gamma_\omega)}=O(\epsilon)$.
Moreover,
since the hypersurface $\Gamma_\omega$
has strictly positive area
for $\omega\in I$,
there is $v\sb I>0$
(dependent on $I$ but not on $\omega$)
so that
\[
\abs{U_\epsilon(\Gamma_\omega)}
\ge v\sb I\epsilon,
\qquad
\forall\omega\in I.
\]
One can see that
for all
$\xi\in U_\epsilon(\Gamma_\omega)$
and all $\epsilon\in(0,1)$
the denominator
in the integral in the right-hand side of \eqref{denomi}
is bounded from above by
$k\sb{I}\epsilon^2$,
for some $k\sb{I}>0$
which depends on $I$ but not on $\omega\in I$.
Thus,
there is $c\sb I>0$ such that
\[
\norm{\mathcal{G}(\omega+i\epsilon)}\sb{l^2(\Z^n)}^2
\ge\frac{c\sb I}{\epsilon},
\qquad
\epsilon\in(0,1).
\]
\end{proof}


\begin{lemma}\label{lemma-ac}
$\tilde f\at{\torus\setminus(\overline{\varOmega\sb 0}\cup\overline{\varOmega\sb\pi})}
\in L^2\big(\torus\setminus(\overline{\varOmega\sb 0}\cup\overline{\varOmega\sb\pi})\big)$.
\end{lemma}

\begin{proof}
We will prove this lemma
considering
$\tilde f(\omega)$ for $\omega\in\C\sp{+}$
and then taking the limit $\Im\omega\to 0$.
Since $f\sp{\ttd}=0$ for $\ttd\le 0$
(Cf. \eqref{f-phi-extended})
and
$f\sp{\ttd}$ is bounded due to \eqref{l2-bound},
the Fourier transform
\eqref{def-f-tilde}
extended to $\omega\in\C\sp{+}$,
\begin{equation}\label{distf+}
\tilde f(\omega)=\sum\sb{\ttd\in\Z}e^{i\omega \ttd}f\sp{\ttd}
=\sum\sb{\ttd\in\N}e^{i\omega \ttd}f\sp{\ttd},
\qquad
\omega\in\C\sp{+}\!\!\!\!\mod 2\pi,
\end{equation}
defines an analytic function of $\omega\in\C\sp{+}$,
which satisfies
$\abs{\tilde f(\omega)}\le\frac{C}{1-e^{-\abs{\Im\omega}}}$,
$\omega\in\C\sp{+}$,
with some $C<\infty$.
Since for each $\ttd\in\N$
there is a convergence
\[
f\sp{\ttd} e^{-\epsilon \ttd}\to f\sp{\ttd},
\qquad
\epsilon\to 0+,
\]
$\tilde f$ satisfies
\begin{equation}\label{trr}
\tilde f(\omega+ i\epsilon)\to\tilde f(\omega),
\qquad
\omega\in\C\sp{+}\!\!\!\!\mod 2\pi,
\quad
\epsilon\to 0+,
\end{equation}
with the convergence in the sense of distributions.
We need to show
that for any closed subset
$I\subset\torus\setminus(\overline{\varOmega\sb 0}\cup\overline{\varOmega\sb\pi})
$
there is 
$C\sb I<\infty$
such that
\begin{equation}\label{iif}
\int\sb{I}\abs{\tilde f(\omega+i\epsilon)}^2\,d\omega<C\sb I,
\qquad
\epsilon\in(0,1).
\end{equation}
If this were the case,
then, taking into account the convergence \eqref{trr},
we would conclude that
\[
\int\sb{I}\abs{\tilde f(\omega)}^2 d\omega
\le C\sb I
\]
by the Banach-Alaoglu theorem on weak compactness,
finishing the proof.

It remains to prove \eqref{iif}.
Similarly to \eqref{distf+}, we define
\begin{equation}\label{tilde-varphi}
\tilde\varphi\sb{\xxd}(\omega)
=\sum\sb{\ttd\in\N}e^{i\omega \ttd}\varphi\sb{\xxd}\sp{\ttd},
\qquad
\xxd\in\Z^n,
\quad
\omega\in\C\sp{+}\!\!\!\!\mod 2\pi,
\end{equation}
which is an analytic function valued in $l^2(\Z^n)$.
Its limit as $\Im\omega\to 0+$
exists as an element of $\mathscr{D}'(\torus,l^2(\Z^n))$.
Due to equation \eqref{eq22},
the complex Fourier transforms
of $\varphi\sp{\ttd}\sb{\xxd}$ and of $f\sp{\ttd}$
are related by
\begin{equation}\label{bvafc}
\widetilde{\varphi}\sb{\xxd}(\omega)=\tau^2
\mathcal{G}\sb{\xxd}(\omega)\tilde f(\omega),
\qquad
\omega\in\C\sp{+}\!\!\!\!\mod 2\pi.
\end{equation}
Using \eqref{bvafc}
and the Plancherel theorem,
we see that
\begin{equation}\label{le-ce}
\tau^4
\int\sb{\torus}
\norm{\mathcal{G}(\omega+i\epsilon)}\sb{l^2(\Z^n)}^2
\abs{\tilde f(\omega+i\epsilon)}^2
\,\frac{d\omega}{2\pi}
=
\sum\sb{\ttd\in\N}\norm{\varphi\sp{\ttd}}\sb{l^2(\Z^n)}^2
e^{-2\epsilon \ttd}
\le\frac{\sup\limits\sb{\ttd\in\Z\sb{+}}\norm{\varphi\sp{\ttd}}\sb{l^2(\Z^n)}^2}
{1-e^{-2\epsilon}},
\end{equation}
where
\begin{equation}\label{phib}
\sup\sb{\ttd\in\Z\sb{+}}\norm{\varphi\sp{\ttd}}\sb{l^2(\Z^n)}<\infty
\end{equation}
due to the a priori estimates
on $\psi$ and $\chi$
(see \eqref{l2-bound} and \eqref{chib}).



Combining the bound
\eqref{le-ce}
with the bound
on
$\norm{\mathcal{G}(\omega+i\epsilon)}\sb{l^2(\Z^n)}$
obtained in Lemma~\ref{lemma-lap} (Cf. \eqref{nv-large}),
we conclude
that there is $C\sb I<\infty$ such that
\eqref{iif} holds.
This completes the proof of the lemma.
\end{proof}

\section{Spectral relation}
\label{sect-spectral-relation}

Let $\zeta\in l(\Z,l^2(\Z^n))$
be the omega-limit trajectory of the solution
$\psi\in l(\Z,l^2(\Z^n))$
to the Cauchy problem \eqref{dkg-cp},
in the sense of Definition~\ref{def-omega-pt}.
That is, we assume that
$\zeta$ is a solution to \eqref{kgz}
and that there is a sequence
$\ttd\sb j\in\N$, $j\in\N$,
such that
$(\psi\sp{\ttd\sb j},\psi\sp{\ttd\sb j+1})$
converge weakly to
$(\zeta\sp{\ttd},\zeta\sp{\ttd+1})\at{\ttd=0}$.

Let us express $\tilde\zeta\sb{\xxd}(\omega)$
in terms of $\tilde g(\omega)$;
this representation
will allow us to express
$\zeta\sb{\xxd}\sp{\ttd}$ in equation \eqref{kgz}
via $\zeta\sb 0\sp{\ttd}$.
By the definition of omega-limit trajectory
$\tilde\zeta\sb{\xxd}(\omega)$
(see Definition~\ref{def-omega-pt}),
its Fourier transform in $\ttd$
satisfies the stationary Helmholtz-type equation
\begin{equation}\label{eqn-tilde-zeta}
\Big(
(2+\tau^2 m^2)\cos\omega-2
-\frac{1}{n}\bm{D}\sb{\xxd}^2
\Big)
\tilde\zeta\sb{\xxd}(\omega)
=\tau^2\delta\sb{\xxd,0}\tilde g(\omega),
\qquad
\xxd\in\Z^n,
\quad
\omega\in\torus.
\end{equation}

\begin{lemma}\label{lemma-b-sg}
The solution to the stationary problem
\eqref{eqn-tilde-zeta}
satisfies
\begin{equation}\label{bsg}
\tilde\zeta\sb{\xxd}(\omega)=\tau^2\mathcal{G}\sb{\xxd}(\omega)\tilde g(\omega),
\qquad
\omega\in\torus\setminus\Sigma,
\qquad
\xxd\in\Z^n,
\end{equation}
where
$\Sigma$
is the set of singular points
defined in \eqref{def-singular-points}.
\end{lemma}

\begin{proof}
By \eqref{eqn-tilde-zeta},
the Fourier transform of $\zeta$ in $\xxd$ and $\ttd$,
\[
\hat{\tilde\zeta}(\xi,\omega)
=\sum\sb{\xxd\in\Z^n,\,\ttd\in\Z}
\zeta\sb{\xxd}\sp{\ttd}
e^{-i\xi\cdot \xxd}
e^{i\omega \ttd},
\qquad
\xi\in\torus^n,
\quad
\omega\in\torus,
\]
satisfies
\begin{equation}\label{eqn-tilde-zeta-xo}
\Big(
(2+\tau^2 m^2)\cos\omega
-\frac{2}{n}\sum\sb{j=1}\sp{n}\cos\xi\sb j
\Big)
\hat{\tilde\zeta}(\xi,\omega)
=a(\xi,\omega)
\hat{\tilde\zeta}(\xi,\omega)
=\tau^2\tilde g(\omega),
\end{equation}
where
$\xi\in\torus^n$
and
$\omega\in\torus$.
By \eqref{def-f}, \eqref{def-g},
and \eqref{psij-to-zeta},
\begin{equation}\label{f-to-g}
f_j\sp{\ttd}
:=
f\sp{\ttd+\ttd\sb j}\mathop{\longrightarrow}\limits\sb{j\to\infty} g\sp{\ttd},
\qquad
\forall \ttd\in\Z,
\end{equation}
and moreover the functions $\{f_j\}$
are uniformly bounded
in $l^\infty(\Z)$
due to \eqref{l2-bound}.
This is enough to conclude that
$f\sb j$ converge to $g$
in the sense of tempered distributions on
$\Z$ (dual to $\mathscr{S}(\Z)$,
which is
the vector space of sequences decaying faster than
any power of $n$),
hence,
due to
continuity of the Fourier transform
in the space of tempered distributions
and due to
$\mathscr{S}'(\torus)=\mathscr{D}'(\torus)$,
\begin{equation}\label{fj-to-g}
\tilde f(\omega)e^{-i\omega \ttd\sb j}
\mathop{\longlongrightarrow}\limits\sb{j\to\infty}
\sp{\mathscr{D}'(\torus)} 
\tilde g(\omega).
\end{equation}
Due to \eqref{psij-to-zeta} and Lemma~\ref{lemma-pd},
\begin{equation}\label{phij-to-zeta}
\zeta\sp{\ttd}\sb{\xxd}
=\lim\sb{\ttd\sb j\to\infty}\psi\sp{\ttd\sb j+\ttd}\sb{\xxd}
=\lim\sb{\ttd\sb j\to\infty}\varphi\sp{\ttd\sb j+\ttd}\sb{\xxd},
\qquad \xxd\in\Z^n,
\quad
\ttd\in\Z.
\end{equation}
Moreover, by Theorem~\ref{theorem-gwp-dkg},
the solutions
$\psi\sb\xxd\sp\ttd$,
$\chi\sb\xxd\sp\ttd$
(Cf.~\eqref{eq1}),
and hence
their difference
$\varphi\sb\xxd\sp\ttd=\psi\sb\xxd\sp\ttd-\chi\sb\xxd\sp\ttd$
are bounded uniformly in $\ttd$ and $\xxd$.
Therefore, similarly to how we arrived at \eqref{fj-to-g},
\begin{equation}\label{tilde-psij-to-zeta}
\tilde\varphi\sb{\xxd}(\omega)e^{-i\omega \ttd\sb j}
\mathop{\longlongrightarrow}\limits\sb{j\to\infty}
\sp{\mathscr{D}'(\torus,l^2(\Z^n))} 
\tilde\zeta\sb{\xxd}(\omega).
\end{equation}
Now the proof follows from
taking the limit $\Im\omega\to 0+$
in \eqref{bvafc}
and using
\eqref{fj-to-g}
and
\eqref{tilde-psij-to-zeta},
and also taking into account that
for each 
$\xxd\in\Z^n$,
$\mathcal{G}\sb{\xxd}(\omega)$
is a multiplier in $\mathscr{D}'(\torus\setminus\Sigma)$
by Lemma~\ref{lemma-lap}.
\end{proof}

\bigskip

Let us show that
the spectra of $\zeta\sb 0\sp{\ttd}$ and $g\sp{\ttd}$
are located inside
the closures of
the spectral gaps.

\begin{lemma}\label{lemma-supp-g}
$\supp\tilde g
\subset\overline{\varOmega\sb 0}\cup\overline{\varOmega\sb\pi}
$.
\end{lemma}

\begin{proof}
By Lemma~\ref{lemma-ac},
for any $\varrho\in C\sp\infty\sb 0(\torus)$ 
with
$\supp\varrho\in\torus\setminus(\overline{\varOmega\sb 0}\cup\overline{\varOmega\sb\pi})$,
one has
$\varrho(\omega)\tilde f(\omega)\in L^1(\torus)$.
Hence, by the Riemann-Lebesgue lemma,
\[
\varrho(\omega)\tilde f(\omega) e^{-i\omega \ttd\sb j}
\mathop{\longrightarrow}\limits\sb{j\to\infty}0
\]
in the sense of distributions on $\torus$.
Due to \eqref{fj-to-g},
we conclude that $\tilde g=0$
on
$\torus\setminus
(\overline{\varOmega\sb 0}\cup\overline{\varOmega\sb\pi})$.
\end{proof}

\begin{lemma}\label{lemma-nii}
A point $\omega\in\varOmega_c$
can not be
an isolated point of the support of $\tilde\zeta\sb 0$.
\end{lemma}

\begin{proof}
Assume that $\omega\sb 1\in\varOmega_c\setminus\p\varOmega_c$
is an isolated point of the support of $\tilde\zeta\sb 0$.
By Lemma~\ref{lemma-supp-g}
we know that $\omega\sb 1\notin\supp\tilde g$.
Then
there is $\varrho\subset C\sp\infty\sb 0(\torus)$
such that
$\varrho(\omega\sb 1)=1$ and
$\supp\varrho\cap\supp\tilde g=\emptyset$.
Due to the spectral representation \eqref{bsg},
we see that
for any $\xxd\in\Z^n$ one has
$\supp\varrho\cap\supp\tilde\zeta\sb{\xxd}\subset\{\omega\sb 1\}$.
Therefore,
$\varrho(\omega)\tilde\zeta\sb{\xxd}(\omega)=\delta(\omega-\omega\sb 1)M\sb{\xxd}$,
where $M\in l^2(\Z^n)$ and $M\not\equiv 0$.
(The terms of the form
$N\sb X\delta^{(k)}(\omega-\omega\sb 1)$,
$N\in l\sp 2(\Z^n)$,
with $k\ge 1$
do not appear
due to the a priori bounds on $\psi\sb\xxd\sp\ttd$;
see Theorem~\ref{theorem-gwp-dkg}.)
The relation \eqref{eqn-tilde-zeta} implies that
$M$ satisfies the equation
\begin{equation}\label{eqM}
-2(1-\cos\omega\sb 1)
M\sb{\xxd}
-\frac{1}{n}\bm{D}\sb{\xxd}^2 M\sb{\xxd}+\tau^2 m^2 M\sb{\xxd}\cos\omega\sb 1
=0,
\end{equation}
hence its Fourier transform satisfies
\begin{equation}\label{eqM-hat}
\Big[
(2+\tau^2 m^2)\cos\omega\sb 1
-\frac{2}{n}\sum\sb{j=1}\sp{n}\cos\xi_j
\Big]
\hat M(\xi)
=0.
\end{equation}
It follows that $\hat M$ is supported
on the hypersurface
$\Gamma\sb{\omega\sb 1}=\{\xi\in\torus^n\sothat a(\xi,\omega\sb 1)=0\}
\subset\torus^{n}$.
(This hypersurface has singular points
if $\omega\sb 1\in\Sigma$;
see \eqref{def-singular-points}.)
Since there is no nonzero
$\hat M\in l^2(\torus^n)$
supported on such a hypersurface,
we arrive at a contradiction;
hence,
$\omega\sb 1\in\varOmega_c\setminus\p\varOmega_c$
can not be an isolated point
of the support of $\tilde\zeta\sb 0$.
\end{proof}

\begin{lemma}\label{lemma-supp-zeta}
$\supp\tilde\zeta\sb 0
\subset\overline{\varOmega\sb 0}\cup\overline{\varOmega\sb\pi}
$.
\end{lemma}

\begin{proof}
This inclusion follows from Lemma~\ref{lemma-supp-g},
the spectral representation \eqref{bsg},
and Lemma~\ref{lemma-nii}.
\end{proof}


We will use the construction of quasimeasures \cite{ubk-arma}.
Denote by $\check\qm$ the inverse Fourier transform
of $\qm\in\mathscr{D}'(\torus)$:
\[
\check\qm(\ttd)=\mathscr{F}^{-1}[\qm(\omega)](\ttd)
=\int\sb{\torus}e^{-i\omega \ttd}\qm(\omega)\,\frac{d\omega}{2\pi},
\qquad
\ttd\in\Z.
\]




\begin{definition}[Quasimeasures]
\label{def-qm}
\begin{enumerate}
\item
The space $l\sp\infty\sb{F}(\Z)$
is the vector space $l\sp\infty(\Z)$
endowed with the following convergence:
\[
f\sb\epsilon(\ttd)
\mathop{\llongrightarrow}\limits\sp{l\sp\infty\sb{F}}
\sb{\epsilon\to 0+}
f(\ttd)
\]
if
\[
\sup\limits\sb{\epsilon>0}\norm{f\sb\epsilon}\sb{l\sp\infty(\Z)}<\infty
\qquad
\mbox{and}
\qquad\forall \ttd\sb{1}\in\N,
\ \lim\limits\sb{\epsilon\to 0+}
\sup\limits\sb{\ttd\in\Z,\,\abs{\ttd}\le\ttd\sb{1}}
\abs{f\sb\epsilon(\ttd)-f(\ttd)}\to 0.
\]
\item
The space of \emph{quasimeasures}
is the vector space of distributions
with bounded Fourier transform,
\[
\mathscr{Q}(\torus)=\{\qm\in\mathscr{D}'(\torus)\sothat
\check \qm\in l\sp\infty(\Z)\},
\]
endowed
with the following convergence:
\[
\qm\sb\epsilon(\omega)\stackrel{\mathscr{Q}}
{\mathop{\llongrightarrow}
\limits\sb{\epsilon\to 0+}}
\qm(\omega)
\quad{\rm if
}\quad
\check\qm\sb\epsilon(\ttd)\stackrel{l\sp\infty\sb{F}}
{\mathop{\llongrightarrow}\limits\sb{\epsilon\to 0+}}
\check\qm(\ttd).
\]
\end{enumerate}
\end{definition}

For example,
any function from $L^1(\torus)$ is a quasimeasure,
and so is any finite Borel measure on $\torus$.


Let
$M\in C(\torus)$,
and let
$\bm{M}:\,C(\torus)\to C(\torus)$
be
the operator of multiplication
by $M$:
\[
\mbox{for\ }u\in C(\torus),
\quad
\bm{M}:\,u\mapsto M u\in C(\torus).
\]

\begin{lemma}[Multipliers in the space of quasimeasures]
\label{lemma-qm-multipliers}
\quad
\begin{enumerate}
\item
\label{lq1-i-0}
If
$M\in C(\torus)$
is such that
$\check M\in l^1(\Z)$,
then
\[
\bm{M}:\,C(\torus)\to C(\torus),
\qquad
\bm{M}:\,u\mapsto M u
\]
extends
to a bounded linear operator
$
\bm{M}:\,\mathscr{Q}(\torus)\to\mathscr{Q}(\torus).
$
\item
\label{lq1-i-0i}
Let
$\check M\sb\epsilon\in l^1(\Z)$
be bounded uniformly for $\epsilon>0$.
If
\begin{equation}\label{q-m}
\qm\sb\epsilon
\mathop{\llongrightarrow}\limits\sp{\mathscr{Q}(\torus)}\sb{\epsilon\to 0+}
\qm\,,
\qquad
\check M\sb\epsilon
\mathop{\llongrightarrow}\limits\sp{l^1(\Z)}\sb{\epsilon\to 0+}
\check M,
\end{equation}
then
$
M\sb\epsilon\qm\sb\epsilon
\mathop{\llongrightarrow}\limits\sp{\mathscr{Q}(\torus)}\sb{\epsilon\to 0+}
M\qm
$.
\end{enumerate}
\end{lemma}
\begin{proof}
We define
$M(\omega)\qm(\omega)
:=\mathscr{F}[\big(\check M\ast\check\qm\big)(\ttd)](\omega)$
that agrees with the case
$\qm\in C(\torus)$.
The statement
({\it 1}) follows from ({\it 2})
with
$M\sb\epsilon=M$
and $\qm\sb\epsilon\in C(\torus)$.
To prove  ({\it 2}), by Definition~\ref{def-qm},
we need to show that
\begin{equation}\label{lemma-second-statement-0}
\mathscr{F}^{-1}
[M\sb\epsilon(\omega)\qm\sb\epsilon(\omega)](\ttd)
=\big(\check M\sb\epsilon\ast\check\qm\sb\epsilon\big)(\ttd)
\mathop{\llongrightarrow}\limits\sp{l\sp\infty\sb{F}}\sb{\epsilon\to 0+}
\big(\check M\ast\check\qm\big)(\ttd).
\end{equation}
Define the functions
\[
f\sb\epsilon(\ttd):=\mathscr{F}^{-1}
[M\sb\epsilon(\omega)\qm\sb\epsilon(\omega)](\ttd)
=\big(\check M\sb\epsilon\ast\check\qm\sb\epsilon\big)(\ttd),
\phantom{\int}
\]
\[
f(\ttd):=\mathscr{F}^{-1}
[M(\omega)\qm(\omega)](\ttd)=
\big(\check M\ast\check\qm\big)(\ttd).
\]
By Definition~\ref{def-qm},
to prove the convergence \eqref{lemma-second-statement-0},
we need to show that
\begin{equation}\label{wnts1}
\sup\sb{\epsilon>0}\norm{f\sb\epsilon}\sb{l\sp\infty(\Z)}<\infty
\end{equation}
and that for any $\ttd\sb 1\in\N$
one has
\begin{equation}\label{wnts2}
\lim\sb{\epsilon\to 0+}
\sup\sb{\abs{\ttd}<\ttd\sb 1}\abs{f\sb\epsilon(\ttd)-f(\ttd)}
=0.
\end{equation}
To prove \eqref{wnts1}, we write:
\[
\norm{f\sb\epsilon}\sb{l\sp\infty(\Z)}
\le
\sum\sb{\ttd'\in\Z}
\abs{\check M\sb\epsilon(\ttd')\check\qm\sb\epsilon(\ttd-\ttd')}
\le
\norm{\check M\sb\epsilon}\sb{l^1(\Z)}
\norm{\check\qm\sb\epsilon}\sb{l\sp\infty(\Z)},
\]
which is bounded uniformly in $\epsilon>0$.

It remains to prove \eqref{wnts2}.
We need to show that,
given $\ttd\sb 1\in\N$,
for any $\delta>0$
there is $\epsilon\sb\delta>0$
such that for any $\epsilon\in(0,\epsilon\sb\delta)$
one has
$\sup\sb{\ttd}\abs{f\sb\epsilon(\ttd)-f(\ttd)}<\delta$.
We have:
\begin{equation}\label{two-terms-0}
f\sb\epsilon(\ttd)-f(\ttd)=
\big(\check M\sb\epsilon\ast\check\qm\sb\epsilon\big)(\ttd)
-\big(\check M\ast\check\qm\big)(\ttd)
=
\big(\!(\check M\sb\epsilon-\check M)\ast\check\qm\sb\epsilon\big)(\ttd)
+\big(\check M\ast(\check\qm\sb\epsilon-\check\qm)\!\big)(\ttd).
\end{equation}
The first term in the right-hand side
of \eqref{two-terms-0}
converges to zero uniformly in $\ttd\in\Z$
since
$\check M\sb\epsilon-\check M\to 0$ in $l^1(\Z)$ while
$\check\qm\sb\epsilon\in l\sp\infty(\Z)$ are bounded uniformly
for $\epsilon>0$.
If $\epsilon\sb\delta>0$ is small enough
and $\epsilon\in(0,\epsilon\sb\delta)$,
then the first term in the right-hand side
of \eqref{two-terms-0} is smaller than $\delta/3$.
We break the second term
in the right-hand side of \eqref{two-terms-0}
into 
\begin{equation}\label{lim-second-term-0}
\sum\sb{\abs{\ttd'}>\ttd\sb\delta}
\check M(\ttd')
\big(\check\qm\sb\epsilon(\ttd-\ttd')-\check\qm(\ttd-\ttd')\big)
+
\sum\sb{\abs{\ttd'}\le\ttd\sb\delta}
\check M(\ttd')
\big(\check\qm\sb\epsilon(\ttd-\ttd')-\check\qm(\ttd-\ttd')\big),
\end{equation}
where $\ttd\sb\delta\in\N$ is chosen as follows:
Since $\check M\in l^1(\Z)$,
while $\check\qm\sb\epsilon-\check\qm$ is bounded in $l\sp\infty(\Z)$
uniformly in $\epsilon>0$,
there exists $\ttd\sb\delta\in\N$ so that
\begin{equation}\label{lim-second-term-a0}
\sum\sb{\ttd'\in\Z,\,\abs{\ttd'}>\ttd\sb\delta}
\abs{\check M(\ttd')}
\abs{\check\qm\sb\epsilon(\ttd-\ttd')-\check\qm(\ttd-\ttd')}
<\delta/3.
\end{equation}
On the other hand,
since $\qm\sb\epsilon\to\qm$ in $\mathscr{Q}(\torus)$,
one has
\begin{equation}\label{lim-second-term-b0}
\lim\sb{\epsilon\to 0}
\sup\sb{\abs{\ttd}\le \ttd\sb 1}
\sup\sb{\abs{\ttd'}\le \ttd\sb\delta}
\abs{\check\qm\sb\epsilon(\ttd-\ttd')-\check\qm(\ttd-\ttd')}
\le
\lim\sb{\epsilon\to 0}
\sup\sb{\abs{\ttd}\le \ttd\sb 1+\ttd\sb\delta}
\abs{\check\qm\sb\epsilon(\ttd)-\check\qm(\ttd)}
=0,
\end{equation}
so that, choosing $\epsilon\sb\delta>0$ smaller than necessary,
we make sure that the second term in \eqref{lim-second-term-0}
is also smaller than $\delta/3$ for $\epsilon\in(0,\epsilon\sb\delta)$,
and therefore \eqref{two-terms-0} is bounded by $\delta$.

Thus, as $\epsilon\to 0+$,
\eqref{two-terms-0}
converges to zero uniformly in
$\abs{\ttd}\le \ttd\sb{1}$,
proving \eqref{wnts2}.
The convergence \eqref{lemma-second-statement-0} follows.
\end{proof}


Denote $\Sigma'=\Sigma\setminus\p\varOmega_c$,
where $\p\varOmega_c=\{\pm\omega\sb m,\pi\pm\omega\sb m\}$,
with
$\omega\sb m=\arccos\big(1+\frac{\tau^2m^2}{2}\big)^{-1}$.

\begin{lemma}\label{lemma-r-qm-multiplier}
For $n\ge 1$,
the function
\begin{equation}\label{def-r}
r(\omega)=\frac{1}{\mathcal{G}\sb{0}(\omega+i0)},
\qquad
\omega\in\torus,
\end{equation}
is continuous and real-valued for
$\omega\in\overline{\varOmega\sb 0}\cup\overline{\varOmega\sb\pi}$
and satisfies
\begin{equation}\label{as-cos}
r(\omega)=-r(\pi+\omega),
\qquad
r(\omega)=r(-\omega),
\qquad
\omega\in\overline{\varOmega\sb 0}\cup\overline{\varOmega\sb\pi};
\qquad
r\at{\varOmega\sb 0}>0,
\qquad
r\at{\varOmega\sb\pi}<0.
\end{equation}
It is a
multiplier in the space of quasimeasures
$\mathscr{Q}(\torus\setminus\Sigma')$,
and moreover
for any $\rho\in C\sp\infty(\torus)$
with support
away from $\Sigma'$
one has
\begin{equation}\label{csm}
\mathscr{F}^{-1}
\Big[
\frac{\rho(\omega)}{\mathcal{G}\sb{0}(\omega+i\epsilon)}
\Big](\ttd)
\mathop{\llongrightarrow}\limits\sp{l^1(\Z)}\sb{\epsilon\to 0+}
\mathscr{F}^{-1}
\Big[
\rho(\omega)r(\omega)
\Big](\ttd).
\end{equation}
\end{lemma}

\begin{remark}
Due to Lemma~\ref{lemma-qm-multipliers},
one concludes that
Lemma~\ref{lemma-r-qm-multiplier} implies that,
as $\omega\to 0+$,
the ratio
$\frac{1}{\mathcal{G}\sb 0(\omega+i\epsilon)}$
converges to $r(\omega)$
in the space of multipliers
which act on quasimeasures
with support in
$\torus\setminus\Sigma'$.
\end{remark}

\begin{proof}
The relations \eqref{as-cos}
follow from
\eqref{def-varsigma-c}.

Since $\mathcal{G}\sb 0(\omega)$
is a smooth function
of $\omega\in(\C\sp{+}\!\!\!\!\mod 2\pi)\setminus\Sigma'$,
it is enough to check the convergence
\eqref{csm}
for the Fourier transform
\begin{equation}\label{fto}
\mathscr{F}\sp{-1}
\left[
\frac{\rho(\omega)}{\mathcal{G}_0(\omega+i\epsilon)}
\right](\ttd),
\qquad
\ttd\in\Z,
\quad
\epsilon\in(0,1),
\end{equation}
with
$\rho\in C\sp\infty\sb 0(\torus)$
supported in a small open neighborhood of
$\Sigma\setminus\Sigma'=\p\varOmega_c
=\{\pm\omega\sb m,\pi\pm\omega\sb m\}\subset\torus$,
such that
$\Sigma'\cap\supp\rho=\emptyset$.
We have:
\[
\mathcal{G}_0(\omega+i\epsilon)
=\frac{1}{2}
\int\sb{\torus^n}\frac{1}{\big(1+\frac{\tau^2 m^2}{2}\big)\cos(\omega+i\epsilon)
-\frac{1}{n}\sum\sb{j=1}\sp{n}
\cos\xi\sb j}\frac{d\xi}{(2\pi)^n}.
\]
In the case $n\ge 3$,
the convergence
of \eqref{fto} in $l^1$ as $\epsilon\to 0$
is straightforward
since
at the points $\omega\in\p\varOmega_c$
the function $\mathcal{G}_0(\omega+i\epsilon)$
has a nonzero limit
as $\epsilon\to 0$.

We leave the case $n=2$ to the reader
and consider the case $n=1$.
It suffices to consider the case when $\rho$ is supported
in a small neighborhood of $\omega=\omega\sb m$
(all other cases $\omega\in\p\varOmega\sb c$ are handled similarly).
Let $\epsilon\in(0,1)$.
One evaluates $\mathcal{G}_0(\omega+i\epsilon)$
explicitly, getting
\begin{eqnarray}
\mathcal{G}_0(\omega+i\epsilon)
&=&
\frac 1 2 
\int\sb{\torus}\frac{1}{\big(1+\frac{\tau^2 m^2}{2}\big)\cos(\omega+i\epsilon)
-\cos\xi}\frac{d\xi}{2\pi}
\nonumber
\\
&=&
\frac 1 2
\frac{1}{\sqrt{\big(1+\frac{\tau^2 m^2}{2}\big)^2\cos^2(\omega+i\epsilon)-1}},
\nonumber
\end{eqnarray}
hence
\begin{eqnarray}
\frac{\rho(\omega)}{\mathcal{G}\sb 0(\omega+i\epsilon)}
&=&(2+\tau^2 m^2)\rho(\omega)\sqrt{\cos^2(\omega+i\epsilon)-\cos^2\omega\sb m}
\nonumber
\\
&=&
(2+\tau^2 m^2)\rho(\omega)\sqrt{\sin^2\omega\sb m-\sin^2(\omega+i\epsilon)},
\nonumber
\end{eqnarray}
which can be written as
\[
\rho(\omega)f(\omega,\epsilon)\sqrt{\omega-\omega\sb m+i\epsilon},
\]
with $f(\omega,\epsilon)$ begin smooth
on the support of $\rho$,
with the two derivatives bounded
uniformly for $\epsilon\in(0,1)$. 

It suffices to show that
\[
F(\ttd,\epsilon)=\int\sb{\torus} e^{-i\omega\ttd}\rho(\omega)f(\omega,\epsilon)
\sqrt{\omega-\omega\sb m+i\epsilon}
\,d\omega,
\qquad
\ttd\in\Z,
\quad
\epsilon\in(0,1),
\]
decays as $\abs{\ttd}^{-3/2}$,
uniformly in $\epsilon\in(0,1)$.

We pick
$\alpha\in C\sp\infty(\R)$,
$\alpha\at{\abs{s}\ge 2}\equiv 1$,
$\alpha\at{\abs{s}\le 1}\equiv 0$,
and define
$\beta(s)=\alpha(s)-\alpha(s/2)$,
so that $\beta\in C\sp\infty([1,4])$.
Then there is the dyadic decomposition
\[
1=\alpha(s)
+\sum\sb{k=1}\sp{\infty}
\beta(2^{k}s),
\qquad
s\in\R.
\]
For $\ttd\in\Z$,
$\epsilon\in(0,1)$,
we define
\begin{equation}\label{def-f0}
F\sb 0(\ttd,\epsilon)
=\int\sb{\torus}
\alpha(\abs{\omega-\omega\sb m+i\epsilon})
e^{-i\omega\ttd}\rho(\omega)f(\omega,\epsilon)
\sqrt{\omega-\omega\sb m+i\epsilon}
\,d\omega,
\end{equation}
\begin{equation}\label{def-fk}
F\sb k(\ttd,\epsilon)
=\int\sb{\torus}
\beta(2^{k}\abs{\omega-\omega\sb m+i\epsilon})
e^{-i\omega\ttd}\rho(\omega)f(\omega,\epsilon)
\sqrt{\omega-\omega\sb m+i\epsilon}
\,d\omega,
\qquad
k\in\N.
\end{equation}
Since the expression under the integral
defining $F\sb 0$ is smooth in $\omega$
and in $\epsilon$,
there is $C_1<\infty$
independent on $\epsilon\in(0,1)$
such that
$\abs{F\sb 0(\ttd,\epsilon)}\le C_1\abs{\ttd}^{-3/2}$,
$\ttd\in\Z$.
To estimate $\abs{F\sb k(\ttd,\epsilon)}$ with $k\in\N$,
we first notice that
\begin{equation}\label{est-1}
\abs{F\sb k(\ttd,\epsilon)}
\le C_2 2^{-3k/2},
\end{equation}
with $C_2<\infty$ bounded uniformly for $\epsilon\in(0,1)$:
the factor $2^{-k}$
comes from the size of the support in $\omega$
and $2^{-k/2}$
from the magnitude of the square root.
We can also integrate
in \eqref{def-fk}
by parts two times in $\omega$
(with the aid of the operator
$L=i\ttd^{-1}\p\sb\omega$
which is the identity when applied to the exponential),
getting
\begin{equation}\label{est-2}
\abs{F\sb k(\ttd,\epsilon)}
\le C_3 2^{-3k/2}
\Big(\frac{2^{k}}{\abs{\ttd}}\Big)^2,
\end{equation}
where
$2^k/\abs{\ttd}$
is the contribution
from each integration by parts,
when $\p\sb\omega$
could fall onto either on $\beta$ or on the square root
(producing a factor of $2^k$)
or onto $\rho(\omega)f(\omega,\epsilon)$,
and $C_3<\infty$ does not depend on
$\epsilon\in(0,1)$ and $\ttd\in\Z$.
Thus,
taking into account
\eqref{est-1} and \eqref{est-2},
we get the estimate
\begin{eqnarray}
&&
\abs{F(\ttd,\epsilon)}
\le
\abs{F\sb 0(\ttd,\epsilon)}
+
\sum\sb{k\in\N}
\abs{F\sb k(\ttd,\epsilon)}
\nonumber
\\
&&
\le
C_1\abs{\ttd}^{-\frac 3 2}
+
\sum\sb{k\in\N:\;2^k>\abs{\ttd}}
C_2 2^{-\frac{3k}{2}}
+
\sum\sb{k\in\N:\;2^k\le\abs{\ttd}}
C_3 2^{\frac{k}{2}}\abs{\ttd}^{-2}
\le C\abs{\ttd}^{-\frac 3 2},
\nonumber
\end{eqnarray}
valid for all $\ttd\in\Z$;
above, $C<\infty$
does not depend on $\epsilon\in(0,1)$.
It follows that
$\big\vert\mathscr{F}\sp{-1}[
\frac{\rho(\omega)}{\mathcal{G}(\omega+i\epsilon)}](\ttd)\big\vert
\le C\abs{\ttd}^{-3/2}$,
for any $\epsilon\in(0,1)$, $\ttd\in\Z$.
By the dominated convergence theorem (albeit with $\ttd\in\Z$),
the convergence
\eqref{csm}
follows.
\end{proof}

Now we can extend
a version of Lemma~\ref{lemma-b-sg}
with $\xxd=0$
to the supports of $\tilde g$ and $\tilde\zeta\sb 0$,
which
by Lemmas~\ref{lemma-supp-g}
and ~\ref{lemma-supp-zeta}
could include endpoints of the continuous
spectrum, $\p\varOmega_c$.

\begin{lemma}\label{lemma-b-sg-zero}
The solution to the stationary problem
\eqref{eqn-tilde-zeta}
satisfies
\begin{equation}\label{bsg-zero}
r(\omega)\tilde\zeta\sb{0}(\omega)=\tau^2\tilde g(\omega),
\qquad
\omega\in\torus.
\end{equation}
\end{lemma}

\begin{proof}
By \eqref{bvafc},
\begin{equation}\label{bvafc-r}
\frac{1}{\mathcal{G}\sb{0}(\omega)}
\widetilde{\varphi}_0(\omega)=\tau^2
\tilde f(\omega),
\qquad
\omega\in\C\sp{+}\!\!\!\!\mod 2\pi.
\end{equation}
Now, similarly to Lemma~\ref{lemma-b-sg},
the proof follows from
taking the limit $\Im\omega\to 0+$
in \eqref{bvafc-r};
let us provide details.
Due to the convergence \eqref{fj-to-g},
the right hand-side
of \eqref{bvafc-r}
converges to $\tau^2\tilde g(\omega)$, $\omega\in\R$;
let us now consider the left-hand side of \eqref{bvafc-r}.
The convergence \eqref{tilde-psij-to-zeta}
also takes place in the space
of $l^2(\Z^n)$-valued quasimeasures, $\mathscr{Q}(\Z,l^2(\Z^n))$.
Besides, there is a convergence
$\frac{1}{\mathcal{G}\sb 0(\omega+i\epsilon)}\to r(\omega)$
as $\epsilon\to 0+$
in the space of multipliers in $\mathscr{Q}(\torus\setminus\Sigma')$
stated in Lemma~\ref{lemma-r-qm-multiplier}.
Noticing that,
by Lemma~\ref{lemma-supp-zeta},
$\supp\tilde\zeta\sb 0
\subset\overline{\varOmega\sb 0}\cup\overline{\varOmega\sb\pi}
\subset\torus\setminus\Sigma'$,
Lemma~\ref{lemma-qm-multipliers}
on multipliers in the space of quasimeasures
allows us to conclude that
the left-hand side of \eqref{bvafc-r}
converges to $\frac{1}{\mathcal{G}\sb 0(\omega)}\tilde\zeta\sb 0(\omega)$,
$\omega\in\torus$,
and the statement of the lemma follows.
\end{proof}

We define the ``sharp'' operation
$\sharp$
on $\mathscr{D}'(\torus)$
by
\begin{equation}\label{def-sharp}
f\sp\sharp(\omega):=\bar{f}(-\omega),
\qquad
\omega\in\torus.
\end{equation}

\begin{remark}
For $F\in l(\Z)$, one has
$(\tilde F)\sp\sharp=\tilde{\bar F}$.
\end{remark}


Now we will use the spectral representation
stated in Lemma~\ref{lemma-b-sg-zero}
to obtain the following fundamental relation
satisfied by $\zeta\sb 0$,
which we call \emph{the spectral relation}.
In the next section
we will apply our version of the Titchmarsh convolution theorem
to this relation, proving that
the $\omega$-support of $\tilde\zeta$ consists of
one, two, or four frequencies.

\begin{lemma}\label{lemma-zz-is-gg}
\begin{equation}\label{zz-is-gg}
\big(r\tilde\zeta\sb 0\big)
\ast\big(\tilde\zeta\sb{0}\sp\sharp i\sin\omega\big)
+
\big(r\tilde\zeta\sb 0\sp\sharp\big)
\ast\big(\tilde\zeta\sb{0} i\sin\omega\big)
=
-
\tau^2
\widetilde{W(\abs{\zeta\sb{0}\sp{\ttd}}^2)}i\sin\omega,
\qquad
\omega\in\torus.
\end{equation}
\end{lemma}

\begin{proof}
Denote by
$b\sp\ttd$
the coefficient appearing in \eqref{def-g}:
\begin{equation}\label{def-b}
b\sp{\ttd}
=
\left\{\begin{array}{l}
\frac{W(\abs{\zeta\sb 0\sp{\ttd+1}}^2)-W(\abs{\zeta\sb 0\sp{\ttd-1}}^2)}
{\abs{\zeta\sb 0\sp{\ttd+1}}^2-\abs{\zeta\sb 0\sp{\ttd-1}}^2}
\quad
\mbox{if}
\quad
\abs{\zeta\sb 0\sp{\ttd+1}}^2\ne\abs{\zeta\sb 0\sp{\ttd-1}}^2,
\\
W'(\abs{\zeta\sb 0\sp{\ttd+1}}^2)
\quad
\mbox{if}
\quad
\abs{\zeta\sb 0\sp{\ttd+1}}^2=\abs{\zeta\sb 0\sp{\ttd-1}}^2.
\end{array}
\right.
\end{equation}
Then \eqref{def-g} takes the form
$
g\sp{\ttd}
=
-b\sp{\ttd}\frac{\zeta\sb{0}\sp{\ttd+1}+\zeta\sb{0}\sp{\ttd-1}}{2},
$
or, on the Fourier transform side,
\begin{equation}\label{g-tilde}
\tilde g(\omega)
=
-\tilde b\ast\big(\tilde\zeta\sb{0}(\omega)\cos\omega\big),
\qquad
\omega\in\torus.
\end{equation}
Here and below, we are using the identities
\begin{equation}\label{alpha-alpha}
\frac{\widetilde{f\sp{T+1}}(\omega)+\widetilde{f\sp{T-1}}(\omega)}{2}
=\tilde f(\omega)\cos\omega,
\qquad
\frac{\widetilde{f\sp{T+1}}(\omega)-\widetilde{f\sp{T-1}}(\omega)}{2}
=-i\tilde f(\omega)\sin\omega,
\end{equation}
valid for any $\tilde f\in \mathscr{D}'(\torus)$,
which follow from the definition of the Fourier transform
\eqref{def-f-tilde}.
Substituting
\eqref{g-tilde}
into \eqref{bsg-zero},
we obtain the relation
\begin{equation}\label{two-rel-a}
r(\omega)\tilde\zeta\sb{0}(\omega)
=-\tau^2\tilde b\ast\big(\tilde\zeta\sb{0}\cos\omega\big),
\qquad
\omega\in\torus.
\end{equation}
We note that
$r\sp\sharp(\omega):=\overline{r(-\omega)}=\overline{r(\omega)}=r(\omega)$
since $r(\omega)$ is even and real-valued
on the support of $\tilde\zeta\sb 0(\omega)$
by Lemma~\ref{lemma-r-qm-multiplier},
and similarly
$\tilde b\sp\sharp(\omega)=\overline{\tilde b(-\omega)}
=\widetilde{\overline{b}}(\omega)
=\tilde b(\omega)$
since $b\sp{\ttd}$ is real-valued by
\eqref{def-b}.
Hence, by
Lemma~\ref{lemma-b-sg-zero}
and \eqref{g-tilde},
we have:
\begin{equation}\label{two-rel-b}
r(\omega)\tilde\zeta\sb{0}\sp\sharp(\omega)
=-\tau^2\tilde b\ast\big(\tilde\zeta\sb{0}\sp\sharp\cos\omega\big),
\qquad
\omega\in\torus.
\end{equation}
We take the convolution of \eqref{two-rel-a}
with $\tilde\zeta\sb{0}\sp\sharp(\omega)i\sin\omega$,
the convolution of  \eqref{two-rel-b}
with $\tilde\zeta\sb{0}(\omega)i\sin\omega$,
and add them up:
\begin{eqnarray}\label{psr}
&&
\big(r\tilde\zeta\sb{0}\big)
\ast
\big(\tilde\zeta\sb{0}\sp\sharp i\sin\omega\big)
+
\big(r\tilde\zeta\sb{0}\sp\sharp\big)
\ast
\big(\tilde\zeta\sb{0} i\sin\omega\big)
\nonumber
\\
&&
\qquad
=
-\tau^2\tilde b
\ast
\big[
\big(\tilde\zeta\sb{0}\cos\omega\big)
\ast
\big(\tilde\zeta\sb{0}\sp\sharp i\sin\omega\big)
+
\big(\tilde\zeta\sb{0}\sp\sharp\cos\omega\big)
\ast
\big(\tilde\zeta\sb{0}i\sin\omega\big)
\big].
\end{eqnarray}
Using the identity
$\sin(\omega-\sigma)\cos\sigma
+\cos(\omega-\sigma)\sin\sigma=\sin\omega$,
we rewrite the expression in brackets
as $(\tilde\zeta\sb 0\sp\sharp\ast\tilde\zeta\sb 0)i\sin\omega$,
which,
due to the identities \eqref{alpha-alpha},
is the Fourier transform of
$-\frac 1 2\big(\abs{\zeta\sb 0\sp{\ttd+1}}^2-\abs{\zeta\sb 0\sp{\ttd-1}}^2\big)$.
This leads to
\[
\big(r\tilde\zeta\sb 0\big)
\ast\big(\tilde\zeta\sb 0\sp\sharp i\sin\omega\big)
+
\big(r\tilde\zeta\sb 0\sp\sharp\big)
\ast\big(\tilde\zeta\sb 0 i\sin\omega\big)
=\frac{\tau^2}{2}\tilde b
\ast
\big(
\widetilde{\abs{\zeta\sb 0\sp{\ttd+1}}^2}
-\widetilde{\abs{\zeta\sb 0\sp{\ttd-1}}^2}
\big).
\]
As follows from \eqref{def-b}, the right-hand side
of the above relation
is the Fourier transform
of 
$\frac{\tau^2}{2}
(W(\abs{\zeta\sb{0}\sp{\ttd+1}}^2)-W(\abs{\zeta\sb{0}\sp{\ttd-1}}^2))$.
Taking into account the identity
\[
\frac 1 2
\big(
\widetilde{W(\abs{\zeta\sb{0}\sp{\ttd+1}}^2)}
-\widetilde{W(\abs{\zeta\sb{0}\sp{\ttd-1}}^2)}
\big)
=
-\widetilde{W(\abs{\zeta\sb{0}\sp{\ttd}}^2)}i\sin\omega,
\]
we rewrite \eqref{psr} in the desired form
\[
\big(r\tilde\zeta\sb 0\big)
\ast\big(\tilde\zeta\sb{0}\sp\sharp i\sin\omega\big)
+
\big(r\tilde\zeta\sb 0\sp\sharp\big)
\ast\big(\tilde\zeta\sb{0}i\sin\omega\big)
=
-
\tau^2
\widetilde{W(\abs{\zeta\sb{0}\sp{\ttd}}^2)}i\sin\omega.
\]
\end{proof}

\section{Nonlinear spectral analysis of omega-limit trajectories}
\label{sect-nsa}

We are going to prove our main result,
reducing the spectrum of $\zeta\sb 0\sp{\ttd}$
to at most four points.
Denote
\begin{equation}\label{def-t-dot}
\dotT=\torus\setminus\big\{\pm\frac{\pi}{2}\big\},
\qquad
\mathscr{D}'(\dotT)
=
\{f\in\mathscr{D}'(\torus)\sothat
\supp f\subset\dotT\},
\end{equation}

\begin{definition}\label{def-cht-0}
For $f\in\mathscr{D}'(\dotT)$, define
\[
\supppi f
=\Big\{
p\in\big(-\frac{\pi}{2},\frac{\pi}{2}\big)\sothat
\mbox{ either }
p\in\supp f
\mbox{ or }
p+\pi\in\supp f
\Big\}.
\]
\end{definition}

\begin{remark}
For $f\in\mathscr{D}'(\dotT)$,
the convex hull
of $\supppi f$,
which is denoted by
$\ch\supppi f$,
is the smallest closed interval $I\subset(-\frac{\pi}{2},\frac{\pi}{2})$
such that
$\supp f\subset I\cup(\pi+I)$.
\end{remark}

For $y\in\torus$,
let $S\sb{y}$
be the shift operator acting on $\mathscr{D}'(\torus)$:
for $f\in\mathscr{D}'(\torus)$,
\begin{equation}\label{def-shift}
S\sb{y}f(\omega)=f(\omega+y),
\qquad
\omega\in\torus.
\end{equation}
Note that
$S\sb{\pi}:\mathscr{D}'(\dotT)\to \mathscr{D}'(\dotT)$.

\begin{definition}\label{def-l-r-0}
For $\kappa>0$, define
the following subsets of $\mathscr{D}'(\dotT)$:
\[
\bm{L}\sp{\pm}\sb\kappa
=\left\{
f\in\mathscr{D}'(\dotT)
\sothat
\ f\not\equiv 0,
\ f=\pm S\sb{\pi}f
\quad
\mbox{on}
\quad
\Big(-\frac{\pi}{2},\ \inf\supppi f+\kappa\Big)
\right\},
\]
\[
\bm{R}\sp{\pm}\sb\kappa
=\left\{
f\in\mathscr{D}'(\dotT)
\sothat
\ f\not\equiv 0,
\ f=\pm S\sb{\pi}f
\quad
\mbox{on}
\quad
\Big(\sup\supppi f-\kappa,\ \frac{\pi}{2}\Big)
\right\}.
\]
\end{definition}

We use the following
version of the Titchmarsh convolution theorem
on the circle.

\begin{theorem}[Titchmarsh theorem for distributions
on the circle]
\label{theorem-titchmarsh-circle}

Let $f,\,g\in\mathscr{D}'(\dotT)$
satisfy
$
\supppi f+\supppi g\subset(-\frac{\pi}{2},\frac{\pi}{2}).
$
Then
$f\ast g\in\mathscr{D}'(\dotT)$,
and
for each $\kappa>0$
the following statements
{\it ($A$)} and {\it ($B$)}
are equivalent:
\[
\begin{array}{l}
\mbox{
{\it ($A$)}
\ \ $\inf\supppi f\ast g
\ge\inf\supppi f+\inf\supppi g+\kappa$;}
\\ \\
\mbox{
{\it ($B$)}
\ \ Either\ 
\ $f\in\bm{L}\sp{-}\sb\kappa$,
$g\in\bm{L}\sp{+}\sb\kappa$,
 \ or\ 
\ $f\in\bm{L}\sp{+}\sb\kappa$,
$g\in\bm{L}\sp{-}\sb\kappa$.
}
\end{array}
\]

\smallskip

Similarly,
the following statements
{\it ($A'$)} and {\it ($B'$)}
are equivalent:
\[
\begin{array}{l}
\mbox{
{\it ($A'$)}
\ \ $\sup\supppi f\ast g
\le\sup\supppi f+\sup\supppi g-\kappa$;
}
\\ \\
\mbox{
{\it ($B'$)}
\ \ Either
\ $f\in\bm{R}\sp{-}\sb\kappa$,
$g\in\bm{R}\sp{+}\sb\kappa$,
\ or
\ $f\in\bm{R}\sp{+}\sb\kappa$,
$g\in\bm{R}\sp{-}\sb\kappa$.
}
\end{array}
\]
\end{theorem}

\begin{corollary}\label{corollary-titchmarsh-powers}
Let $f\in\mathscr{D}'(\torus)$
be such that
$\supp f\subset
(-\frac{\pi}{2N},\frac{\pi}{2N})
\cup
(\pi-\frac{\pi}{2N},\pi+\frac{\pi}{2N})$,
where $N\in\N$.
Then
\[
\ch\supppi
\underbrace{f\ast\dots\ast f}\sb{N}
=N\ch\supppi f.
\]
\end{corollary}

These results in a slightly different formulation
are proved
in Appendix~\ref{sect-titchmarsh-circle}
(see Theorems~\ref{theorem-titchmarsh-circle-1},
~\ref{theorem-titchmarsh-powers}).


\begin{lemma}\label{lemma-zz}
Under conditions of Proposition~\ref{proposition-sw},
\begin{equation}\label{east}
\ch\supppi
\Big(\widetilde{W(\abs{\zeta\sb 0\sp{\ttd}}^2)}\sin\omega\Big)
\subset
\ch\supppi\Big(\widetilde{\abs{\zeta\sb{0}\sp{\ttd}}^{2}}
\Big).
\end{equation}
\end{lemma}

\begin{proof}
According to Lemma~\ref{lemma-zz-is-gg},
it suffices to prove that
\begin{equation}\label{stp}
\ch\supppi
\Big(
\big(r\tilde\zeta\sb 0\big)
\ast\big(\tilde\zeta\sb{0}\sp\sharp(\omega)i\sin\omega\big)
+
\big(r\tilde\zeta\sb 0\sp\sharp\big)
\ast\big(\tilde\zeta\sb{0}(\omega)i\sin\omega\big)
\Big)
\subset\ch\supppi\tilde\zeta\sb{0}\sp\sharp\ast\tilde\zeta\sb{0}.
\end{equation}
Set
\begin{equation}\label{def-hj}
h\sb 1(\omega)
=
r(\omega)\tilde\zeta\sb 0(\omega)
+
\tilde\zeta\sb{0}(\omega)i\sin\omega,
\qquad
h\sb 2(\omega)
=
r(\omega)\tilde\zeta\sb 0(\omega),
\qquad
h\sb 3(\omega)
=
\tilde\zeta\sb{0}(\omega)i\sin\omega.
\end{equation}
Note that
$\overline{i\sin(-\omega)}=i\sin\omega$
and $\overline{r(-\omega)}=r(\omega)$
for $\omega\in\overline{\varOmega\sb 0}\cup\overline{\varOmega\sb\pi}$
by \eqref{as-cos},
hence
\[
h\sb 1\sp\sharp(\omega)
=
r(\omega)\tilde\zeta\sb{0}\sp\sharp(\omega)
+
\tilde\zeta\sb{0}\sp\sharp(\omega)i\sin\omega,
\qquad
h\sb 2\sp\sharp(\omega)
=
r(\omega)\tilde\zeta\sb{0}\sp\sharp(\omega),
\qquad
h\sb 3\sp\sharp(\omega)
=
\tilde\zeta\sb{0}\sp\sharp(\omega)i\sin\omega.
\]
There is the identity
\[
(r\tilde\zeta\sb 0)
\ast(\tilde\zeta\sb{0}\sp\sharp i\sin\omega)
+
(r\tilde\zeta\sb 0\sp\sharp)
\ast(\tilde\zeta\sb{0} i\sin\omega)
=h\sb 1\sp\sharp\ast h\sb 1
-h\sb 2\sp\sharp\ast h\sb 2
-h\sb 3\sp\sharp\ast h\sb 3.
\]
Therefore,
to prove the inclusion \eqref{stp},
it suffices to prove that
\begin{equation}\label{efp}
\ch\supppi(h\sb j\sp\sharp\ast h\sb j)
\subset
\ch\supppi(\tilde\zeta\sb 0\sp\sharp\ast\tilde\zeta\sb 0),
\qquad
j=1,\,2,\,3.
\end{equation}
Set
$[a,b]=\ch\supppi\tilde\zeta\sb 0
\subset\overline{\varOmega\sb 0}\subset(-\frac{\pi}{2},\frac{\pi}{2})$,
with $a\le b$.

Obviously,
$\supp\tilde\zeta\sb 0\sp\sharp\ast\tilde\zeta\sb 0$
is symmetric (with respect to $\omega=0$),
and
$\supppi\tilde\zeta\sb 0\sp\sharp\ast\tilde\zeta\sb 0
\subset[-b+a,b-a]$.
We consider the following two cases.

\noindent
{\it Case 1:}
$\ch\supppi\tilde\zeta\sb 0\sp\sharp\ast\tilde\zeta\sb 0=[-b+a,b-a]$.
In this case,
the inclusion \eqref{stp} is immediate.
Indeed,
since $\ch\supppi\tilde\zeta\sb{0}=[a,b]$, one has:
\[
\supp h\sb j\subset[a,b]\cup[\pi+a,\pi+b],
\qquad
\supp h\sb j\sp\sharp\subset[-b,-a]\cup[\pi-b,\pi-a];
\qquad
1\le j\le 3.
\]
Therefore,
$\ \supp(h\sb j\sp\sharp\ast h\sb j)\subset[-b+a,b-a]\cup[\pi-b+a,\pi+b-a]$.
Since these two subsets of $\torus$ do not intersect,
one concludes that
$\ \ch\supppi(h\sb j\sp\sharp\ast h\sb j)\subset[-b+a,b-a]$.

\noindent
{\it Case 2:}
for some $\kappa>0$
\[
\ch\supppi\tilde\zeta\sb 0\sp\sharp\ast\tilde\zeta\sb 0
=[-b+a+\kappa,b-a-\kappa].
\]
By Theorem~\ref{theorem-titchmarsh-circle},
we have
$\tilde\zeta\sb 0\in\bm{L}\sp{\sigma}\sb\kappa$
and
$\tilde\zeta\sb 0\sp\sharp\in\bm{L}\sp{-\sigma}\sb\kappa$,
with some $\sigma\in\{\pm\}$,
and also
$\tilde\zeta\sb 0\in\bm{R}\sp{\rho}\sb\kappa$
and
$\tilde\zeta\sb 0\sp\sharp\in\bm{R}\sp{-\rho}\sb\kappa$,
with some $\rho\in\{\pm\}$.
However,
by \eqref{def-sharp},
$f\in\bm{R}\sp\pm\sb\kappa$
implies that
$f\sp\sharp\in\bm{L}\sp\pm\sb\kappa$;
therefore, $\sigma=-\rho$,
so that
\begin{equation}\label{pmpm}
\begin{array}{l}
\mbox{
either
\quad $\tilde\zeta\sb 0\in\bm{L}\sp{+}\sb\kappa\cap\bm{R}\sp{-}\sb\kappa$,
\quad
$\tilde\zeta\sb 0\sp\sharp\in\bm{L}\sp{-}\sb\kappa\cap\bm{R}\sp{+}\sb\kappa$,}
\\ \\
\mbox{
or
\quad $\tilde\zeta\sb 0\in\bm{L}\sp{-}\sb\kappa\cap\bm{R}\sp{+}\sb\kappa$,
\quad
$\tilde\zeta\sb 0\sp\sharp\in\bm{L}\sp{+}\sb\kappa\cap\bm{R}\sp{-}\sb\kappa$.
}
\end{array}
\end{equation}
In \eqref{def-hj}, the multipliers
$r(\omega)$ and $\sin\omega$
are $\pi$-antiperiodic
by \eqref{as-cos}.
Hence, the functions $h\sb j(\omega)$,
$j=1,\,2,\,3$,
defined in \eqref{def-hj},
satisfy the inclusion
\begin{equation}\label{hj-in-lr}
\mbox{either}
\quad
h\sb j\in\bm{L}\sp{-}\sb\kappa\cap\bm{R}\sp{+}\sb\kappa,
\quad
1\le j\le 3,
\qquad
\mbox{or}
\quad
h\sb j\in\bm{L}\sp{+}\sb\kappa\cap\bm{R}\sp{-}\sb\kappa,
\quad
1\le j\le 3,
\end{equation}
with the $\pm$ signs opposite to the signs in \eqref{pmpm}.
Also, since the multipliers
in \eqref{def-hj}
satisfy
$r\sp\sharp(\omega)=r(\omega)$, $(i\sin\omega)\sp\sharp=i\sin\omega$,
the functions $h\sb j\sp\sharp$,
$j=1,\,2,\,3$,
satisfy
\begin{equation}\label{hjs-in-lr}
\mbox{either}
\quad
h\sb j\sp\sharp\in\bm{L}\sp{+}\sb\kappa\cap\bm{R}\sp{-}\sb\kappa,
\quad 1\le j\le 3,
\qquad
\mbox{or}
\quad
h\sb j\sp\sharp\in\bm{L}\sp{-}\sb\kappa\cap\bm{R}\sp{+}\sb\kappa,
\quad 1\le j\le 3.
\end{equation}
Due to inclusions
\eqref{hj-in-lr}
and
\eqref{hjs-in-lr},
Theorem~\ref{theorem-titchmarsh-circle}
gives
\[
\ch\supppi(h\sb j\sp\sharp\ast h\sb j)\subset[-b+a+\kappa,b-a-\kappa].
\]
This yields \eqref{efp}, finishing the proof.
\end{proof}

\begin{lemma}\label{lemma-0pi}
$\supp\widetilde{\abs{\zeta\sb{0}\sp\ttd}^2}\subset\{0;\pi\}$.
\end{lemma}

\begin{proof}
By Lemma~\ref{lemma-zz},
\begin{equation}\label{pq}
\ch\supppi\Big(\widetilde{\abs{\zeta\sb{0}\sp{\ttd}}^{2(p+1)}}\sin\omega\Big)
\subset
\bigcup\sb{1\le q\le p}
\ch\supppi\Big(\widetilde{\abs{\zeta\sb{0}\sp{\ttd}}^{2q}}
\Big),
\end{equation}
since $C\sb p>0$ (see Assumption~\ref{ass-wp}).
At the same time,
due to Lemma~\ref{lemma-supp-zeta},
one has
$\ch\supppi\tilde\zeta\sb{0}\subset[-\omega\sb m,\omega\sb m]$,
therefore
$\ch\supppi\widetilde{\abs{\zeta\sb{0}}^2}\subset[-2\omega\sb m,2\omega\sb m]$.
Since
$\omega\sb m<\frac{\pi}{4(p+1)}$ by Assumption~\ref{ass-m-is-small},
we see that Corollary~\ref{corollary-titchmarsh-powers} is applicable
to
\[
\underbrace{\widetilde{\abs{\zeta\sb{0}\sp\ttd}^2}
\ast\dots
\ast
\widetilde{\abs{\zeta\sb{0}\sp\ttd}^2}
}\sb{q}
=\widetilde{\abs{\zeta\sb 0\sp{\ttd}}^{2q}}\,,
\]
for each $q$ between $1$ and $p+1$;
thus, \eqref{pq} shows that
$\ch\supppi\widetilde{\abs{\zeta\sb{0}\sp\ttd}^2}\subset\{0\}$,
which is equivalent to
$\supp\widetilde{\abs{\zeta\sb{0}\sp\ttd}^2}
\subset\{0;\pi\}$.
\end{proof}

\begin{lemma}\label{lemma-ff}
One of the following possibilities
takes place:
\begin{enumerate}
\item
$\tilde\zeta\sb 0(\omega)=A\delta\sb a(\omega)$,
with $a\in\overline{\varOmega\sb 0}\cup\overline{\varOmega\sb\pi}$
and $A\in\C$;
\medskip
\item
$\tilde\zeta\sb 0(\omega)
=A\delta\sb a(\omega)+A'\delta\sb{\pi+a}(\omega)$,
with $a\in\overline{\varOmega\sb 0}\cup\overline{\varOmega\sb\pi}$
and $A,\,A'\in\C$;
\medskip
\item
$\tilde\zeta\sb 0(\omega)
=A(\delta\sb a(\omega)+\delta\sb{\pi+a}(\omega))
+B(\delta\sb b(\omega)-\delta\sb{\pi+b}(\omega))$,
with $a,\,b\in\overline{\varOmega\sb 0}\cup\overline{\varOmega\sb\pi}$
and $A,\,B\in\C$.
\end{enumerate}
\end{lemma}

\begin{remark}
The above lemma is similar to
Theorem~\ref{theorem-titchmarsh-weird}
in Appendix~\ref{sect-titchmarsh-circle}.
\end{remark}

\begin{proof}
By Lemma~\ref{lemma-supp-zeta},
$\supppi\tilde\zeta\sb{0}\subset\overline{\varOmega\sb 0}$;
we denote $[a,b]:=\ch\supppi\tilde\zeta\sb{0}\subset\overline{\varOmega\sb 0}$.

If $a=b$, so that
$\ch\supppi\tilde\zeta\sb{0}=\{a\}\subset\overline{\varOmega\sb 0}$,
thus $\supp\tilde\zeta\sb{0}\subset\{a;\pi+a\}$;
this is equivalent to the first or second possibilities
stated in the lemma.

Now let us consider the case when $a<b$.
By Lemma~\ref{lemma-0pi},
$\ch\supppi\tilde\zeta\sb{0}\sp\sharp\ast\tilde\zeta\sb{0}
\subset\{0\}$.
Theorem~\ref{theorem-titchmarsh-circle}
is applicable
to $\tilde\zeta\sb 0\sp\sharp\ast\tilde\zeta\sb 0$
since $\supppi\tilde\zeta\sb 0\subset\overline{\varOmega\sb 0}
=[-\omega\sb m,\omega\sb m]$,
with
$\omega\sb m<\frac{\pi}{4(p+1)}$, $p\ge 1$
(Cf. Assumption~\ref{ass-m-is-small}).
The inclusion
$\supppi\tilde \zeta\sb 0\sp\sharp\ast\tilde\zeta\sb 0
\subset\{0\}$
implies that
one can take $\kappa=b-a$ in the statement {\it ($A'$)}
in Theorem~\ref{theorem-titchmarsh-circle},
therefore
either
\begin{equation}\label{either1}
\mbox{
$\tilde\zeta\sb 0\sp\sharp\in\bm{R}\sp{+}\sb{b-a}$
\quad
and
\quad
$\tilde\zeta\sb 0\in\bm{R}\sp{-}\sb{b-a}$,
\qquad
hence
\quad
$\tilde\zeta\sb 0\in\bm{L}\sp{+}\sb{b-a}\cap\bm{R}\sp{-}\sb{b-a}$,
}
\end{equation}
or
\begin{equation}\label{either2}
\mbox{
$\tilde\zeta\sb 0\sp\sharp\in\bm{R}\sp{-}\sb{b-a}$
\quad
and
\quad
$\tilde\zeta\sb 0\in\bm{R}\sp{+}\sb{b-a}$,
\qquad
hence
\quad
$\tilde\zeta\sb 0\in\bm{L}\sp{-}\sb{b-a}\cap\bm{R}\sp{+}\sb{b-a}$.
}
\end{equation}
The cases \eqref{either1} and  \eqref{either2}
are considered in a similar manner.
If \eqref{either1} is satisfied,
then
\begin{equation}\label{case1}
\tilde\zeta\sb 0=-S\sb{\pi}\tilde\zeta\sb 0
\quad\mbox{on}\quad
(a,\pi/2)
\quad\mbox{and}\quad
\tilde\zeta\sb 0=S\sb{\pi}\tilde\zeta\sb 0
\quad\mbox{on}\quad
(-\pi/2,b),
\end{equation}
hence $\tilde\zeta\sb 0\at{(a,b)}=0$,
implying that
\begin{equation}\label{supp-tz}
\supp\tilde\zeta\sb 0\subset\{a;b;\pi+a;\pi+b\},
\end{equation}
and moreover
\begin{equation}\label{tz}
\tilde\zeta\sb 0
=A(\delta\sb a+\delta\sb{\pi+a})
+B(\delta\sb b-\delta\sb{\pi+b}),
\qquad
A,\,B\in\C\setminus 0.
\end{equation}
Note that,
due to the boundedness of $\zeta\sp\ttd\sb 0$
(which follows from
applying Theorem~\ref{theorem-gwp-dkg}
to $\psi$ and then
to its omega-limit trajectory $\zeta$),
its Fourier transform can not contain the
derivatives of $\delta$-functions.
We are thus in the framework
of the third possibility stated in the lemma.

If, instead, \eqref{either2} is satisfied,
we are led to the conclusion
\begin{equation}\label{supp-tz-b}
\supp\tilde\zeta\sb 0\subset\{a;b;\pi+a;\pi+b\},
\end{equation}
and moreover
\begin{equation}\label{tz-b}
\tilde\zeta\sb 0
=A(\delta\sb a-\delta\sb{\pi+a})
+B(\delta\sb b+\delta\sb{\pi+b}),
\qquad
A,\,B\in\C\setminus 0.
\end{equation}
This again puts us in the framework
of the third possibility stated in the lemma.
\end{proof}

\begin{lemma}\label{lemma-mf}
$\zeta\sp\ttd$
is a multifrequency solitary wave
with one, two, or four frequencies.

If $n\le 4$, then
one has
$\supp\tilde\zeta\subset
\Z^n\times(\varOmega\sb 0\cup\varOmega\sb\pi)$.
\end{lemma}

\begin{proof}
By Lemma~\ref{lemma-b-sg-zero},
\begin{equation}
\tilde g(\omega)
=\frac{1}{\tau^2}r(\omega)\tilde\zeta\sb 0(\omega),
\qquad
\omega\in\torus.
\end{equation}
Therefore,
due to Lemma~\ref{lemma-ff},
$\supp\tilde g$
consists of one, two, or four points
inside $\overline{\varOmega\sb 0}\cup\overline{\varOmega\sb\pi}$.
By Lemma~\ref{lemma-b-sg},
\begin{equation}\label{otf-0}
\tilde\zeta\sb\xxd(\omega)
=
\frac{\mathcal{G}\sb\xxd(\omega)}{\mathcal{G}\sb 0(\omega)}
\tilde\zeta\sb 0(\omega),
\qquad
\xxd\in\Z^n,
\quad
\omega\in\varOmega\sb 0\cup\varOmega\sb\pi,
\end{equation}
where we took into account that,
by Lemma~\ref{lemma-lap},
$\mathcal{G}\sb 0(\omega)\ne 0$
for $\omega\in\varOmega\sb 0\cup\varOmega\sb\pi$.
This implies that
\begin{equation}\label{zeta-four}
\tilde\zeta(\omega)
=\sum\sb{j=1}\sp{4}
A\sb j\delta(\omega-\omega\sb j),
\qquad
\omega\sb j\in\overline{\varOmega\sb 0}\cup\overline{\varOmega\sb\pi},
\quad
A\sb j\in l^2(\Z^n),
\quad
1\le j\le 4.
\end{equation}
This implies that
$\zeta\sp\ttd$ is a multifrequency solitary wave
of the form \eqref{solitary-waves}.

By Lemma~\ref{lemma-b-sg},
for each $j$ such that $\omega\sb j\notin\p\varOmega_c$,
in \eqref{zeta-four}
one has
$A\sb j=\tau^2\mathcal{G}(\omega\sb j)$.
Moreover, substituting 
\eqref{zeta-four}
into \eqref{eqn-tilde-zeta},
we see that
even if $\omega\sb j\in\p\varOmega_c$,
then one needs to have
$A\sb j=c\sb j\mathcal{G}(\omega\sb j)$,
with some $c\sb j\in\C$.
At the same time, by Lemma~\ref{lemma-finite-norm},
if $n\le 4$, then
for $\omega\in\torus$
one has
$\mathcal{G}(\omega)\in l^2(\Z^n)$
if and only if
$\omega\in\varOmega\sb 0\cup\varOmega\sb\pi$.
Thus, in the case $n\le 4$,
the requirement that
$\zeta\sp\ttd\in l^2(\Z^n)$
leads to the inclusion
$\omega\sb j\in\varOmega\sb 0\cup\varOmega\sb\pi$,
$1\le j\le 4$.

Note that in the case $n\ge 5$,
when one could have
$\omega\sb j\in\p\varOmega_c$
for some $1\le j\le 4$,
the derivatives of $\delta(\omega-\omega\sb j)$
do not appear
in \eqref{zeta-four}
due to the
uniform $l^2(\Z^n)$-bounds on $\zeta\sp\ttd$.
\end{proof}

By Lemma~\ref{lemma-mf},
we know that the set of all omega-limit
trajectories consists of multifrequency waves.
In Section~\ref{sect-solitary-waves},
we will check that the two-
and four-frequency solitary waves
have the form specified in Proposition~\ref{proposition-sw};
see Lemma~\ref{lemma-two-frequencies}
and Lemma~\ref{lemma-four-frequencies} below.
This will finish the proof of Proposition~\ref{proposition-sw}.

\medskip

Let us complete this section
with a simple derivation of the form
of four-frequency solitary waves
in dimensions $n\le 4$.

\begin{lemma}\label{lemma-four-frequencies-0}
Let $n\le 4$.
Each four-frequency solitary wave
can be represented in the form
\begin{equation}\label{ffsw}
\psi\sb\xxd\sp\ttd
=
\big(1+(-1)^{\ttd+\Lambda\cdot\xxd}\big)\phi\sb\xxd
e^{-i\omega\ttd}
+
\big(1-(-1)^{\ttd+\Lambda\cdot\xxd}\big)\theta\sb\xxd
e^{-i\omega'\ttd},
\qquad
\xxd\in\Z^n,
\quad
\ttd\in\Z,
\end{equation}
with
$\phi,\,\theta\in l^2(\Z^n)$.
\end{lemma}

\begin{proof}
Recall that, by Lemma~\ref{lemma-35},
each multifrequency solitary wave
is its own omega-limit trajectory;
therefore, it is enough to prove
that each four-frequency omega-limit trajectory
has the form \eqref{ffsw}.

Since $\supp\tilde\zeta\sb\xxd\subset\varOmega_0\cup\varOmega_\pi$
by Lemma~\ref{lemma-mf},
the relation \eqref{otf-0}
could be extended to $\omega\in\torus$.
To have a four-frequency omega-limit trajectory, $\tilde\zeta\sb 0$
is to be given by \eqref{tz}.
Then \eqref{otf-0} yields
\begin{eqnarray}
\tilde\zeta\sb\xxd
&=&
\frac{\mathcal{G}\sb\xxd(\omega)}{\mathcal{G}\sb 0(\omega)}
\big(A(\delta\sb a+\delta\sb{\pi+a})
+B(\delta\sb b-\delta\sb{\pi+b})\big)
\nonumber
\\
&=&
A\frac{\mathcal{G}\sb\xxd(a)}{\mathcal{G}\sb 0(a)}
(\delta\sb a+(-1)^{\Lambda\cdot\xxd}\delta\sb{\pi+a})
+B\frac{\mathcal{G}\sb\xxd(b)}{\mathcal{G}\sb 0(b)}
(\delta\sb b-(-1)^{\Lambda\cdot\xxd}\delta\sb{\pi+b}),
\end{eqnarray}
where we took into account that,
by Lemma~\ref{lemma-lap} (Cf. \eqref{ggg}),
\[
\frac{\mathcal{G}\sb\xxd(\omega+\pi)}{\mathcal{G}\sb 0(\omega+\pi)}
=(-1)^{\Lambda\cdot\xxd}
\frac{\mathcal{G}\sb\xxd(\omega)}{\mathcal{G}\sb 0(\omega)},
\qquad
\omega\in\varOmega\sb 0\cup\varOmega\sb\pi.
\]
It follows that
\[
\zeta\sb\xxd\sp\ttd
=
A\frac{\mathcal{G}\sb\xxd(a)}{\mathcal{G}\sb 0(a)}
(1+(-1)^{\ttd+\Lambda\cdot\xxd})
e^{-i a\ttd}
+B\frac{\mathcal{G}\sb\xxd(b)}{\mathcal{G}\sb 0(b)}
(1-(-1)^{\ttd+\Lambda\cdot\xxd})
e^{-i b\ttd},
\]
finishing the proof.
\end{proof}

\section{Analysis of solitary wave solutions}
\label{sect-solitary-waves}

Here we discuss in more detail
one-, two-, and four-frequency solitary waves,
prove that they have the form
specified in Proposition~\ref{proposition-sw},
and construct particular examples.

\subsection{One-frequency solitary waves}
\label{sect-one-frequency}

\begin{lemma}\label{lemma-one-frequency}
\begin{enumerate}
\item
If $n\le 4$,
there could only be nonzero solitary waves
$\phi e^{-i\omega\ttd}$
with
$\phi\in l^2(\Z^n)$
for $\omega\in\varOmega\sb 0\cup\varOmega\sb\pi$,
where
$\varOmega\sb 0$ and $\varOmega\sb\pi$
are defined in \eqref{def-spectral-gaps}.
If $n\ge 5$,
there could only be
solitary waves
for $\omega\in\overline{\varOmega\sb 0}\cup\overline{\varOmega\sb\pi}$.
\item
For a particular value
$\omega\in\varOmega\sb 0\cup\varOmega\sb\pi$
($\omega\in\overline{\varOmega\sb 0}\cup\overline{\varOmega\sb\pi}$
if $n\ge 5$),
there is a nonzero solitary wave if and only if
\begin{equation}\label{omega-soliton}
\frac{1}{\mathcal{G}\sb{0}(\omega)\cos\omega}
\in{\rm Range}(-\tau^2 W'(\lambda)\at{\lambda>0}).
\end{equation}
\item
The component of the solitary manifold
which corresponds to one-frequency solitary waves
is generically two-dimensional.
\item
In the case $n=1$,
the necessary and sufficient criterion
for the existence of nonzero solitary waves 
is
\begin{equation}\label{1-sigma-sup}
\left(0,\,2\sqrt{\Big(1+\frac{\tau^2 m^2}{2}\Big)^2-1}\right]
\cap
\mathop{\rm Range}(-\tau^2 W'(\lambda)\at{\lambda>0})
\ne\emptyset.
\end{equation}
\end{enumerate}
\end{lemma}

\begin{proof}
Let us substitute the Ansatz
$\psi\sb{\xxd}\sp{\ttd}=\phi\sb{\xxd}e^{-i\omega \ttd}$,
$\omega\in\torus$,
into \eqref{dkg-c}.
Using the relations
\[
\p\sb t^2\phi\sb{\xxd}e^{-i\omega \ttd}
=
(e^{-i\omega(t+1)}+e^{-i\omega(t-1)}-2e^{-i\omega})\phi\sb{\xxd}
=2(\cos\omega-1)\phi\sb{\xxd}e^{-i\omega \ttd},
\]
\[
f\sp{\ttd}
=-W'(\abs{\phi\sb{0}}^2)\phi\sb{0}
e^{-i\omega \ttd}\cos\omega
\]
(Cf. \eqref{def-f}),
we see that $\phi\sb{\xxd}$ satisfies
\[
2(\cos\omega-1)\phi\sb{\xxd}
=\frac{1}{n}\bm{D}\sb{\xxd}^2\phi\sb{\xxd}-\tau^2 m^2\phi\sb{\xxd}\cos\omega
-\delta\sb{\xxd,0}\tau^2 W'(\abs{\phi\sb{0}}^2)\phi\sb{0}\cos\omega.
\]
Equivalently, 
the Fourier transform 
$\hat\phi(\xi)=\sum\sb{\xxd\in\Z^n}\phi\sb{\xxd}e^{-i\xi\cdot \xxd}$
is to satisfy
\begin{equation}\label{Ax}
a(\xi,\omega)
\hat\phi(\xi)
=-\tau^2 W'(\abs{\phi\sb{0}}^2)\phi\sb{0}\cos\omega,
\qquad
\xi\in\torus^n,
\quad
\omega\in\torus.
\end{equation}
Thus, 
\begin{equation}\label{hat-phi}
\hat\phi(\xi)
=\frac{C}
{a(\xi,\omega)},
\qquad
\xi\in\torus^n,
\qquad\mbox{or, equivalently,}
\quad\phi\sb{\xxd}=C\mathcal{G}\sb{\xxd}(\omega),
\qquad
\xxd\in\Z^n,
\end{equation}
with $C\in\C$. 
By Lemma~\ref{lemma-finite-norm},
$\phi$ is of finite $l\sp 2$-norm
if and only if
$\omega\in\varOmega\sb 0\cup\varOmega\sb\pi$
for $n\le 4$,
and
$\omega\in\overline{\varOmega\sb 0}\cup\overline{\varOmega\sb\pi}$
for $n\ge 5$.
This proves the first statement of the lemma.

\medskip

Substituting \eqref{hat-phi} into
\eqref{Ax}, we see that
$C\ne 0$ is to satisfy the equation
\begin{equation}\label{cond-sw}
1=-\tau^2 W'(\abs{C\mathcal{G}\sb{0}(\omega)}^2)
\mathcal{G}\sb{0}(\omega)\cos\omega.
\end{equation}
Equation \eqref{cond-sw} admits a solution if and only if 
the condition
\eqref{omega-soliton} holds,
proving the second statement of the lemma.

For each $\omega\in\varOmega\sb 0\cup\varOmega\sb\pi$,
the set of solutions $C$
to equation \eqref{cond-sw}
(if it is nonempty)
admits the representation
$C=a e^{i s}$
with $a>0$ and $s\in\torus$;
for each particular $\omega$,
the set of values $a$ is discrete
(under the assumption that $W(\lambda)$
is a polynomial of degree larger than $1$).
The solitary manifold
can be locally parametrized by 
two parameters, $a>0$ and $s\in\torus$,
proving the third statement of the lemma.

Finally, in the case $n=1$, the computation yields
\begin{equation}\label{varsigma-omega-0}
\mathcal{G}\sb{0}(\omega)
=
\frac 1 2
\int\sb{\torus^1}
\frac{1}
{\big(1+\frac{\tau^2 m^2}{2}\big)\cos\omega-\cos\xi\sb 1}
\,\frac{d\xi\sb 1}{2\pi}
=
\frac 1 2
\frac{\mathop{\rm sign}{\cos\omega}}
{\sqrt{
\big(1+\frac{\tau^2 m^2}{2}\big)^2\cos^2\omega-1}}.
\end{equation}
It follows that
\[
\frac{1}{\mathcal{G}\sb{0}(\omega)\cos\omega}
=2\sqrt{\Big(1+\frac{\tau^2 m^2}{2}\Big)^2-\frac{1}{\cos^2\omega}},
\qquad
\omega\in\varOmega\sb 0\cup\varOmega\sb\pi,
\]
hence
\[
\mathop{\rm Range}\Big(\frac{1}{\mathcal{G}\sb{0}(\omega)\cos\omega}
\At{\omega\in\varOmega\sb 0\cup\varOmega\sb\pi}
\Big)
=
\left(0,\,\,2\sqrt{\Big(1+\frac{\tau^2 m^2}{2}\Big)^2-1}\right],
\]
showing that
\eqref{omega-soliton}
is equivalent to
\eqref{1-sigma-sup}.
\end{proof}



\subsection{Two-frequency solitary waves}
\label{sect-two-frequency}

Let us study two-frequency solitary wave solutions.
By Lemma~\ref{lemma-ff},
the two frequencies of a two-frequency
solitary wave differ by $\pi$,
hence we need to consider
solitary wave solutions of the form
\[
\psi\sp\ttd
=
p e^{-i\omega_1\ttd}
+
q e^{-i(\omega_1+\pi)\ttd},
\qquad
p,\,q\in l^2(\Z^n),
\]
with $p,\,q\in l^2(\Z^n)$ not identically zero.
We have:
\[
\psi\sb 0\sp\ttd=p\sb 0 e^{-i\omega_1\ttd}+q\sb 0 e^{-i(\omega_1+\pi)\ttd},
\qquad
\abs{\psi\sb 0\sp\ttd}^2
=\alpha+\beta e^{-i\pi\ttd},
\]
where
\begin{equation}\label{def-alpha-beta}
\alpha=\abs{p\sb 0}^2+\abs{q\sb 0}^2,
\qquad
\beta=2\Re(\bar p\sb 0 q\sb 0).
\end{equation}
We can write
\[
W'(\abs{\psi\sb 0\sp\ttd})=M+e^{-i\pi\ttd}N,
\]
with
\begin{equation}\label{mnwp}
M=\frac 1 2(W'(\alpha+\beta)+W'(\alpha-\beta)),
\qquad
N=\frac 1 2(W'(\alpha+\beta)-W'(\alpha-\beta)).
\end{equation}
Taking into account that
\[
\frac{1}{2}(\psi\sb 0\sp{\ttd+1}+\psi\sb 0\sp{\ttd-1})
=
p\sb 0\cos\omega\sb 1 e^{-i\omega\sb 1\ttd}
-
q\sb 0\cos\omega\sb 1 e^{-i(\omega\sb 1+\pi)\ttd},
\]
we see that
the Fourier transform of $\psi\sb\xxd\sp\ttd$
(with respect to both time and space variables)
satisfies
\[
a(\xi,\omega)
\big[
\hat p\delta\sb{\omega\sb 1}
+
\hat q\delta\sb{\omega\sb 1+\pi}
\big]
=-2\pi\tau^2
\big[
M\delta\sb 0+N\delta\sb\pi
\big]
\ast
\big[
(p\sb 0\delta\sb{\omega\sb 1}-q\sb 0\delta\sb{\omega\sb 1+\pi})\cos\omega\sb 1
\big].
\]
Collecting the coefficients at
$\delta\sb{\omega\sb 1}$
and
$\delta\sb{\omega\sb 1+\pi}$,
we get the equations
\[
\left\{
\begin{array}{l}
a(\xi,\omega\sb 1)\hat p(\xi)
=-2\pi\tau^2(M p\sb 0-N q\sb 0)
\cos\omega_1,
\\
a(\xi,\omega\sb 1+\pi)\hat q(\xi)
=-2\pi\tau^2(N p\sb 0-M q\sb 0)
\cos\omega_1.
\end{array}
\right.
\]
Dividing by $a(\xi,\omega)$
(at particular values of $\omega$)
and taking the inverse
Fourier transform with respect to $\xi$,
we have:
\begin{equation}\label{s1-2}
\left\{
\begin{array}{l}
p\sb\xxd=-2\pi\tau^2(M p\sb 0-N q\sb 0)
\mathcal{G}\sb\xxd(\omega_1)\cos\omega_1,
\\
q\sb\xxd
=-2\pi\tau^2(N p\sb 0-M q\sb 0)
\mathcal{G}\sb\xxd(\omega_1+\pi)\cos\omega_1.
\end{array}
\right.
\end{equation}
By Lemma~\ref{lemma-lap},
$\mathcal{G}\sb 0(\omega)\ne 0$
for $\omega\in\varOmega\sb 0\cup\varOmega\sb\pi$
($\omega\in\overline{\varOmega\sb 0}\cup\overline{\varOmega\sb\pi}$
if $n\ge 5$);
therefore,
if either $p\sb 0$ or $q\sb 0$ were zero,
\eqref{s1-2}
would yield that
either $M p\sb 0-N q\sb 0$ or $N p\sb 0-M q\sb 0$ is zero,
hence
either $p\sb\xxd$ or $q\sb\xxd$
would be identically zero.
Thus,
for two-frequency solitary waves,
we can assume that
both $p\sb 0$ and $q\sb 0$ are nonzero.
Then equations \eqref{s1-2}
lead to
\begin{equation}\label{s1-2-3}
1=-2\pi\tau^2(M-N \frac{q\sb 0}{p\sb 0})
\mathcal{G}\sb 0(\omega_1)\cos\omega_1,
\quad
1
=2\pi\tau^2(N \frac{p\sb 0}{q\sb 0}-M)
\mathcal{G}\sb 0(\omega_1)\cos\omega_1.
\end{equation}
We took into account that
$\mathcal{G}\sb 0(\omega_1+\pi)=-\mathcal{G}\sb 0(\omega_1)$
(Cf. Lemma~\ref{lemma-lap}).
Relations \eqref{s1-2-3} are consistent
if
$\sigma:=\frac{q\sb 0}{p\sb 0}=\pm 1$
and
\begin{equation}\label{eqnb-2}
1+2\pi\tau^2 (M-\sigma N)\mathcal{G}\sb 0(\omega_1)\cos\omega_1
=0.
\end{equation}

Now we can prove the following lemma.

\begin{lemma}\label{lemma-two-frequencies}
\begin{enumerate}
\item
The component of the solitary manifold
which corresponds to two-frequency solitary waves
is generically two-dimensional.
\item
Each two-frequency solitary wave
can be represented in the form
\begin{equation}
\psi\sb\xxd\sp\ttd
=
\big(1+(-1)^{\ttd+\Lambda\cdot\xxd}\sigma\big)\phi\sb\xxd
e^{-i\omega\ttd},
\qquad
\xxd\in\Z^n,
\quad
\ttd\in\Z,
\end{equation}
with
$\sigma\in\{\pm 1\}$,
$\Lambda=(1,\dots,1)\in\Z^n$,
and
$\phi\in l^2(\Z^n)$.
\end{enumerate}
\end{lemma}

\begin{proof}
Let us choose $p\sb 0\in\C\setminus 0$
and $\sigma=\pm 1$,
and set $q\sb 0=\sigma p\sb 0$.
Then,
by \eqref{def-alpha-beta},
$\beta=\sigma\alpha$,
and the relations
\eqref{mnwp} give
$M-\sigma N=W'(\alpha-\sigma\beta)=W'(0)$.
The relation \eqref{eqnb-2}
takes the form
\[
1+2\pi\tau^2 W'(0)\mathcal{G}\sb 0(\omega_1)\cos\omega_1
=0,
\]
allowing us to determine $\omega_1$
(if $W'(0)\ge 0$, no such $\omega\sb 1$ exists).
Thus, the corresponding component
of the solitary manifold
is generically of dimension $2$.

To prove the second statement of the lemma,
we notice that, by \eqref{s1-2},
\[
p\sb\xxd=-2\pi\tau^2 p\sb 0(M-\sigma N)
\mathcal{G}\sb\xxd(\omega_1)\cos\omega_1,
\]
\[
q\sb\xxd
=-2\pi\tau^2 \sigma p\sb 0(\sigma N-M)
\mathcal{G}\sb\xxd(\omega_1+\pi)\cos\omega_1
=(-1)\sp{\Lambda\cdot\xxd}\sigma p\sb\xxd.
\]
In the last equality,
we took into account Lemma~\ref{lemma-lap}
(Cf. \eqref{ggg}).
This finishes the proof.
\end{proof}

\subsection{Four-frequency solitary waves}
\label{subsection-ff}

By Lemma~\ref{lemma-ff},
it is enough to consider four-frequency
solitary waves of the form
\begin{equation}\label{ff}
\psi\sp\ttd
=p e^{-i\omega\sb 1\ttd}+q e^{-i(\omega\sb 1+\pi)\ttd}
+r e^{-i\omega\sb 2\ttd}+s e^{-i(\omega\sb 2+\pi)\ttd},
\end{equation}
with $\omega\sb 1\ne\omega\sb 2\mod\pi$
and with
$p,\,q,\,r,\,s\in l^2(\Z^n)$.
By Lemma~\ref{lemma-ff},
we can also assume that
\begin{equation}\label{p0q0r0s0}
q\sb 0=p\sb 0,
\qquad
s\sb 0=-r\sb 0;
\end{equation}
then
\begin{equation}\label{psi-0-t}
\psi\sb 0\sp{\ttd}
=
p\sb 0(e^{-i\omega\sb 1 \ttd}+e^{-i(\omega\sb 1+\pi)\ttd})
+r\sb 0(e^{-i\omega\sb 2 \ttd}-e^{-i(\omega\sb 2+\pi)\ttd}),
\end{equation}

\begin{equation}\label{psi-0-t-2}
\abs{\psi_0\sp{\ttd}}^2
=
2(\abs{p\sb 0}^2+\abs{r\sb 0}^2)
+
2(\abs{p\sb 0}^2-\abs{r\sb 0}^2)
e^{-i\pi \ttd}
=
\alpha
+
\beta e^{-i\pi \ttd},
\end{equation}
where
\begin{equation}\label{def-a-b}
\alpha=2(\abs{p\sb 0}^2+\abs{r\sb 0}^2),
\qquad
\beta=2(\abs{p\sb 0}^2-\abs{r\sb 0}^2).
\end{equation}
Using \eqref{psi-0-t},
we derive
\begin{eqnarray}
&&
\frac 1 2
(\psi\sb 0\sp{\ttd+1}+\psi\sb 0\sp{\ttd-1})
\nonumber
\\
&&
\hskip -14pt
=
p\sb 0\cos\omega\sb 1 e^{-i\omega\sb 1 \ttd}
-p\sb 0\cos\omega\sb 1 e^{-i(\omega\sb 1+\pi)\ttd}
+
r_0\cos\omega\sb 2 e^{-i\omega\sb 2 \ttd}
+r_0\cos\omega\sb 2 e^{-i(\omega\sb 2+\pi)\ttd},
\nonumber
\end{eqnarray}
and for its Fourier transform we have
\begin{eqnarray}
&&
\mathscr{F}
\big[
\frac{1}{2}(\psi\sb 0\sp{\ttd+1}+\psi\sb 0\sp{\ttd-1})
\big](\omega)
\nonumber
\\
&&
\qquad
=
2\pi
p\sb 0
\big(\delta\sb{\omega_1}(\omega)-\delta\sb{\omega_1+\pi}(\omega)\big)
\cos\omega_1
+
2\pi
r\sb 0
\big(\delta\sb{\omega_2}(\omega)+\delta\sb{\omega_2+\pi}(\omega)\big)
\cos\omega_2.
\nonumber
\end{eqnarray}
We have
$W'(\abs{\psi\sp{\ttd}\sb 0}^2)=M+e^{-i\pi \ttd}N$,
$\ttd\in\Z$,
where $M$ and $N$ are given by
\begin{equation}\label{abAB}
M=\frac{1}{2}(W'(\alpha+\beta)+W'(\alpha-\beta)),
\qquad
N=\frac{1}{2}(W'(\alpha+\beta)-W'(\alpha-\beta)),
\end{equation}
hence
$
\mathscr{F}
\big[
W'(\abs{\psi\sp{\ttd}\sb 0}^2)
\big](\omega)
=2\pi(M\delta(\omega)+N\delta\sb\pi(\omega)).
$
Thus, the Fourier transform of $\psi\sb{\xxd}\sp{\ttd}$
(with respect to both time and space variables)
satisfies the following relation:
\begin{eqnarray}
&&
a(\xi,\omega)
\big[
\hat p\delta\sb{\omega\sb 1}+\hat q\delta\sb{\omega\sb 1+\pi}
+
\hat r\delta\sb{\omega\sb 2}+\hat s\delta\sb{\omega\sb 2+\pi}
\big]
\nonumber
\\
&&
\quad
=-2\pi\tau^2
\big[
M\delta_0+N\delta\sb{\pi}
\big]
\ast
\big[
p\sb 0(\delta\sb{\omega_1}-\delta\sb{\omega_1+\pi})\cos\omega_1
+
r\sb 0(\delta\sb{\omega_2}+\delta\sb{\omega_2+\pi})\cos\omega_2
\big].
\nonumber
\end{eqnarray}
In this equation,
$\delta$-functions are functions of $\omega\in\torus$;
the functions
$\hat p$, $\hat q$, $\hat r$, $\hat s$
(Fourier transforms of $p,\,q,\,r,\,s\in l^2(\Z^n)$)
depend on $\xi\in\torus^n$,
and the convolution
in the right-hand side is with respect to $\omega$.
Collecting the coefficients at
$\delta\sb{\omega_1}$, $\delta\sb{\omega_1+\pi}$,
$\delta\sb{\omega_2}$, and $\delta\sb{\omega_2+\pi}$,
we rewrite the above equation
as the following system:
\begin{equation}\label{r1}
\left\{\begin{array}{l}
a(\xi,\omega_1)\hat p(\xi)=-2\pi\tau^2 p\sb 0(M-N)\cos\omega_1,
\\
a(\xi,\omega_1+\pi)\hat q(\xi)=-2\pi\tau^2 p\sb 0(N-M)\cos\omega_1,
\\
a(\xi,\omega_2)\hat r(\xi)=-2\pi\tau^2 r\sb 0(M+N)\cos\omega_2,
\\
a(\xi,\omega_2+\pi)\hat s(\xi)=-2\pi\tau^2 r\sb 0(N+M)\cos\omega_2.
\end{array}
\right.
\end{equation}
Dividing each of these equations by $a(\xi,\omega)$
(taken at the appropriate value of $\omega$),
taking the inverse Fourier transform
with respect to $\xi$
and using
the relation
$\hat{\mathcal{G}}(\xi,\omega)=\frac{1}{a(\xi,\omega)}$
(Cf. \eqref{def-varsigma-c}),
we have:
\begin{equation}\label{rr1}
\left\{
\begin{array}{l}
p\sb\xxd=-2\pi\tau^2 p\sb 0(M-N)
\mathcal{G}\sb\xxd(\omega_1)\cos\omega_1,
\\
q\sb\xxd
=-2\pi\tau^2 p\sb 0(N-M)
\mathcal{G}\sb\xxd(\omega_1+\pi)\cos\omega_1,
\\
r\sb\xxd=-2\pi\tau^2 r\sb 0(M+N)
\mathcal{G}\sb\xxd(\omega_2)\cos\omega_2,
\\
s\sb\xxd
=-2\pi\tau^2 r\sb 0(N+M)
\mathcal{G}\sb\xxd(\omega\sb 2+\pi)\cos\omega_2.
\end{array}
\right.
\end{equation}
Taking into account that,
by Lemma~\ref{lemma-lap} (Cf. \eqref{ggg}),
one has
$\mathcal{G}\sb\xxd(\omega+\pi)
=-(-1)^{\Lambda\cdot\xxd}\mathcal{G}\sb\xxd(\omega)$
for $\omega\in\varOmega_0\cup\varOmega_\pi$,
we rewrite
\eqref{rr1} as
\begin{equation}\label{rr1a}
\left\{
\begin{array}{l}
p\sb\xxd=-2\pi\tau^2 p\sb 0(M-N)
\mathcal{G}\sb\xxd(\omega_1)\cos\omega_1,
\\
q\sb\xxd
=-2\pi\tau^2 p\sb 0(M-N)(-1)^{\Lambda\cdot\xxd}
\mathcal{G}\sb\xxd(\omega_1)\cos\omega_1,
\\
r\sb\xxd=-2\pi\tau^2 r\sb 0(M+N)
\mathcal{G}\sb\xxd(\omega_2)\cos\omega_2,
\\
s\sb\xxd
=2\pi\tau^2 r\sb 0(M+N)(-1)^{\Lambda\cdot\xxd}
\mathcal{G}\sb\xxd(\omega\sb 2)\cos\omega_2.
\end{array}
\right.
\end{equation}
To have $p\sb 0\ne 0$,
the first equation leads to the requirement
\begin{equation}\label{a-b}
1+2\pi\tau^2 (M-N)\mathcal{G}\sb 0(\omega\sb 1)\cos\omega\sb 1=0.
\end{equation}
Similarly,
to have $r\sb 0\ne 0$,
the third equation requires that
\begin{equation}\label{a+b}
1+2\pi\tau^2 (M+N)\mathcal{G}\sb 0(\omega\sb 2)\cos\omega\sb 2=0.
\end{equation}
Note that
the second and the fourth equations from
\eqref{rr1a}
together with \eqref{a-b} and \eqref{a+b}
lead to
$q\sb 0=p\sb 0$
and $s\sb 0=-r\sb 0$,
in consistency with \eqref{p0q0r0s0}.

To construct an example of four-frequency solitary waves,
one fixes $\omega\sb 1,\,\omega\sb 2\in\varOmega_0$
and determines $M$, $N$ from \eqref{a-b} and \eqref{a+b}.
Then one takes arbitrary nonzero
$p\sb 0$ and $r\sb 0$,
defines $\alpha$ and $\beta$ from \eqref{def-a-b},
and chooses a polynomial $W(s)$
such that \eqref{abAB} is satisfied.

\begin{lemma}\label{lemma-four-frequencies}
\begin{enumerate}
\item
The component of the solitary manifold
which corresponds to four-frequency solitary waves
is generically four-dimensional.
\item
Each four-frequency solitary wave
can be represented in the form
\begin{equation}
\psi\sb\xxd\sp\ttd
=
\big(1+(-1)^{\ttd+\Lambda\cdot\xxd}\big)\phi\sb\xxd
e^{-i\omega\ttd}
+
\big(1-(-1)^{\ttd+\Lambda\cdot\xxd}\big)\theta\sb\xxd
e^{-i\omega'\ttd},
\qquad
\xxd\in\Z^n,
\quad
\ttd\in\Z,
\end{equation}
with
$\phi,\,\theta\in l^2(\Z^n)$.
\end{enumerate}
\end{lemma}

\begin{proof}
As follows from the above discussion,
once we have one solitary wave of this type,
we can vary $\omega\sb 1$ and $\omega\sb 2$.
Then the relations
\eqref{a-b}, \eqref{a+b}
determine $M$ and $N$.
The relations
\eqref{abAB}
determine $\alpha$ and $\beta$,
and the relations \eqref{def-a-b}
determine $\abs{p\sb 0}$ and $\abs{r_0}$;
then
$p\sb 0$ and $r\sb 0$ are known up to
(mutually independent) unitary factors.
Thus,
locally the component of the solitary
manifold
corresponding to four-frequency solitary waves
is parametrized by four variables.

The second statement of the lemma
follows from
\eqref{ff}
after noticing that,
by \eqref{rr1a},
one has
\[
q\sb\xxd=(-1)^{\Lambda\cdot\xxd}p\sb\xxd,
\qquad
s\sb\xxd=-(-1)^{\Lambda\cdot\xxd}r\sb\xxd.
\]
\end{proof}

We finished studying the structure of
multifrequency solitary wave solutions.
Now the proof of Proposition~\ref{proposition-sw}
is complete.

\appendix

\section{
Well-posedness
for nonlinear wave equation in discrete space-time}

\label{sect-wp}

\subsection{Continuous case}

Let us first consider the $\mathbf{U}(1)$-invariant nonlinear wave equation
\begin{equation}\label{nlw-xt}
\ddot\psi(x,t)=\Delta\psi(x,t)
-2\p\sb\lambda v(x,\abs{\psi(x,t)}^2)\psi(x,t),
\qquad
x\in\R^n,
\end{equation}
where 
$\psi(x,t)\in\C$
and
$v(x,\lambda)$
is such that $v\in C(\R^n\times\R)$
and $v(x,\cdot)\in C\sp{2}(\R)$ for each $x\in\R^n$.
Equation \eqref{nlw-xt}
can be written in the Hamiltonian form,
with the Hamiltonian
\begin{equation}\label{def-energy}
\mathscr{E}(\psi,\dot\psi)
=\int\sb{\R^n}
\Big[
\frac{\abs{\dot\psi}^2}{2}+\frac{\abs{\nabla\psi}^2}{2}+v(x,\abs{\psi(x,t)}^2)
\Big]\,dx.
\end{equation}
The value of the Hamiltonian
functional $\mathscr{E}$
and the value of the charge functional
\begin{equation}\label{def-charge}
\mathscr{Q}(\psi,\dot\psi)
=\frac{i}{2}\int\sb{\R^n}
\big(
\bar\psi\dot\psi-\dot{\bar\psi}\psi
\big)\,dx
\end{equation}
are formally conserved for solutions to \eqref{nlw-xt}.
A particular case of \eqref{nlw-xt}
is the nonlinear Klein-Gordon equation,
with $v(x,\lambda)=\frac{m^2}{2}\lambda+z(x,\lambda)$,
with $m>0$:
\begin{equation}\label{nlkg}
\ddot\psi=\Delta\psi-m^2\psi-2\p\sb\lambda z(x,\abs{\psi}^2)\psi,
\qquad
x\in\R^n,
\quad
t\in\R.
\end{equation}
If
$z(x,\lambda)\ge 0$ for all $x\in\R^n$, $\lambda\ge 0$,
then the conservation of the energy
\[
\int\sb{\R^n}
\Big[
\frac{\abs{\dot\psi}^2}{2}+\frac{\abs{\nabla\psi}^2}{2}+\frac{m^2\abs{\psi}^2}{2}
+z(x,\abs{\psi}^2)
\Big]\,dx
\]
yields an a priori estimate
on the norm of the solution:
\begin{equation}\label{apec}
\int\sb{\R^n}
\abs{\psi(x,t)}^2\,dx
\le \frac{2}{m^2}\mathscr{E}(\psi\at{t=0},\dot\psi\at{t=0}).
\end{equation}

\subsection{Finite difference approximation}

Let us now describe the discretized equation.
Let $(\xxd,\ttd)\in\Z^n\times\Z$
denote a point of the space-time lattice.
We will always indicate
the temporal dependence by superscripts 
and the spatial dependence by subscripts.
Fix $\varepsilon>0$, and
let $V\sb{\xxd}(\lambda)=v(\varepsilon \xxd,\lambda)$
be a function on $\Z^n\times\R$,
so that $V\sb{\xxd}\in C\sp{2}(\R)$ for each $\xxd\in\Z^n$.
For $\lambda,\,\mu\in\R$
and $\xxd\in\Z^n$,
we introduce
\begin{equation}\label{vmv}
B\sb{\xxd}(\lambda,\mu):=
\left\{
\begin{array}{l}
\frac{V\sb{\xxd}(\lambda)-V\sb{\xxd}(\mu)}{\lambda-\mu},
\qquad
\lambda\ne\mu,
\\
\p\sb\lambda V\sb{\xxd}(\lambda),
\qquad
\lambda=\mu.
\end{array}
\right.
\end{equation}
We consider the Vazquez-Strauss finite-difference scheme for \eqref{nlw-xt}
\cite{MR0503140}:
\begin{eqnarray}\label{dkg-c-xt}
&&
\frac{\psi\sb{\xxd}\sp{\ttd+1}-2\psi\sb{\xxd}\sp{\ttd}+\psi\sb{\xxd}\sp{\ttd-1}}{\tau^2}
=
\sum\sb{j=1}\sp{n}
\frac{\psi\sb{\xxd+\e\sb j}\sp{\ttd}-2\psi\sb{\xxd}\sp{\ttd}+\psi\sb{\xxd-\e\sb j}\sp{\ttd}}{\varepsilon^2}
\nonumber
\\
&&
\qquad\qquad\qquad\qquad
-
B\sb{\xxd}(\abs{\psi\sb{\xxd}\sp{\ttd+1}}^2,\abs{\psi\sb{\xxd}\sp{\ttd-1}}^2)
(\psi\sb{\xxd}\sp{\ttd+1}+\psi\sb{\xxd}\sp{\ttd-1}),
\end{eqnarray}
where
$\psi\sb{\xxd}\sp{\ttd}\in\C$
is
defined on the lattice $(\xxd,\ttd)\in\Z^n\times\Z$.
Above,
\begin{equation}\label{def-ej-xt}
\e\sb 1=(1,0,0,0,\dots)\in\Z^n,
\qquad
\e\sb 2=(0,1,0,0,\dots)\in\Z^n,
\qquad
\mbox{etc.}
\end{equation}

The continuous limit of
\eqref{dkg-c-xt}
is given by \eqref{nlw-xt},
with $\varepsilon \xxd$
corresponding to $x\in\R^n$
and $\tau \ttd$ corresponding to $t\in\R$.
Since
$\p\sb\lambda V\sb{\xxd}(\lambda)
=B\sb{\xxd}(\lambda,\lambda)$,
the continuous limit of the last term
in the right-hand side of \eqref{dkg-c-xt}
coincides with the right-hand side
in \eqref{nlw-xt}.

An advantage of the Strauss-Vazquez finite-difference scheme \eqref{dkg-c-xt}
over other energy-preserving schemes discussed in
\cite{MR1360462,MR1852556}
is that
it is \emph{explicit}:
at the moment $\ttd+1$
the relation
\eqref{dkg-c-xt} only involves the function $\psi$
at the point $\xxd$, allowing for a simple
realization of the solution algorithm
even in higher dimensional case.

\subsection{Well-posedness}

We will denote by $\psi\sp{\ttd}$
the function $\psi$
defined on the lattice $(\xxd,\ttd)\in\Z^n\times\Z$
at the moment $\ttd\in\Z$.

\begin{theorem}[Existence of solutions]
\label{theorem-e}
Assume that
\begin{equation}\label{def-k1}
k\sb 1:=\inf\sb{\xxd\in\Z^n,\lambda\ge 0}
\p\sb\lambda V\sb{\xxd}(\lambda)>-\infty.
\end{equation}
Define
\[
\tau\sb 1
=\left\{
\begin{array}{l}
\sqrt{-1/k\sb 1},\qquad k\sb 1<0;
\\
+\infty,\qquad\quad k\sb 1\ge 0.
\end{array}
\right.
\]
Then for any $\tau\in(0,\tau\sb{1})$
and any $\varepsilon>0$
there exists a global solution
$\psi\sp{\ttd}$,
$\ttd\in\Z$,
to the Cauchy problem for equation \eqref{dkg-c-xt}
with arbitrary initial data
$\psi\sp{0}$,
$\psi\sp{1}$
(which stand for $\psi\sp{\ttd}$ at $\ttd=0$ and $\ttd=1$).

Moreover,
if $(\psi\sp{0},\psi\sp{1})\in l^2(\Z^n)\times l^2(\Z^n)$,
one has
$\psi\sp{\ttd}\in l^2(\Z^n)$
for all $\ttd\in\Z$.
\end{theorem}

Note that we do not claim
in this theorem
that $\norm{\psi\sp{\ttd}}\sb{l^2(\Z^n)}$
is uniformly bounded for all $\ttd\in\Z$.
For the a priori estimates on $\norm{\psi\sp{\ttd}}\sb{l^2(\Z^n)}$,
see Theorem~\ref{theorem-a-priori} below.

One can readily check that
any
$\xxd$-independent
polynomial potential of the form
\begin{equation}\label{poly}
V\sb{\xxd}(\lambda)
=V(\lambda)
=\sum\sb{q=0}\sp{p}C\sb{q}\lambda^{q+1},
\qquad
p\in\N,
\qquad
C\sb{q}\in\R
\quad
\mbox{for}\ 0\le q\le p,
\qquad
C\sb{p}>0
\end{equation}
satisfies
\eqref{def-k1}.
Note that since $\lim\sb{\lambda\to +\infty}V(\lambda)=+\infty$,
this potential is confining.

\begin{theorem}[Uniqueness and continuous dependence
on the initial data]
\label{theorem-u}
Assume that the functions
\[
K\sb{\xxd}\sp\pm(\lambda,\mu)
=
B\sb{\xxd}(\lambda,\mu)
+2\p\sb\lambda B\sb{\xxd}(\lambda,\mu)
(\lambda\pm\sqrt{\lambda\mu})
\]
are bounded from below:
\begin{equation}\label{def-k2}
k\sb 2:=
\inf\sb{\pm,\,\xxd\in\Z^n,\,
\lambda\ge 0,\,\mu\ge 0}
K\sb{\xxd}\sp\pm(\lambda,\mu)>-\infty.
\end{equation}
Define
\[
\tau\sb 2
=\left\{
\begin{array}{l}
\sqrt{-1/k\sb 2},\qquad k\sb 2<0;
\\
+\infty,\qquad\quad k\sb 2\ge 0.
\end{array}
\right.
\]
Let $\tau\in(0,\tau\sb{2})$ and $\varepsilon>0$.
\begin{enumerate}
\item
There exists a solution to the Cauchy problem
for equation \eqref{dkg-c-xt}
with arbitrary initial data $(\psi\sp{0},\psi\sp{1})$,
and this solution is unique.
\item
For any $\ttd>0$,
the map
\[
U(\ttd):\;(\psi\sp\ttd,\psi\sp{\ttd+1})\at{\ttd=0}
\mapsto (\psi\sp\ttd,\psi\sp{\ttd+1})
\]
is continuous
as a map from
$l\sp\infty(\Z^n)\times l\sp\infty(\Z^n)$
to
$l\sp\infty(\Z^n)\times l\sp\infty(\Z^n)$.
\end{enumerate}
\end{theorem}

\begin{remark}\label{remark-k2-k3}
Note that since
\[
K\sb{\xxd}\sp{-}(\lambda,\lambda)=B\sb{\xxd}(\lambda,\lambda)
=\p\sb\lambda W\sb{\xxd}(\lambda),
\]
the values of $k\sb 1$ and $k\sb 2$ from
Theorem~\ref{theorem-e}
and
Theorem~\ref{theorem-u},
whether $k\sb 2>-\infty$,
are related by
$k\sb 2\le k\sb 1$,
and then
the values of $\tau\sb 1$ and $\tau\sb 2$
from these theorems
are related by
$\tau\sb 2\le \tau\sb 1$.
\end{remark}

\begin{theorem}[Existence and uniqueness for polynomial nonlinearities]
\label{theorem-pol}

\ \begin{enumerate}
\item
\label{theorem-pol-i}

The condition
\eqref{def-k2}
holds for any confining polynomial potential
\eqref{poly}.
\item
\label{theorem-pol-ii}
Assume that
\begin{equation}\label{w4-xt}
V\sb{\xxd}(\lambda)=\sum\sb{q=0}\sp{4}C\sb{\xxd,q}\lambda^{q+1},
\qquad
\xxd\in\Z^n,\quad\lambda\ge 0,
\end{equation}
where
$C\sb{\xxd,q}\ge 0$ for $\xxd\in\Z^n$ and $1\le q\le 4$,
and $C\sb{\xxd,0}$ are uniformly bounded from below:
\begin{equation}\label{def-k3}
k\sb 3:=
\inf\sb{\xxd\in\Z^n}
C\sb{\xxd,0}>-\infty.
\end{equation}
\[
\tau\sb 3
=\left\{
\begin{array}{l}
\sqrt{-1/k\sb 3},\qquad k\sb 3<0;
\\
+\infty,\qquad\quad k\sb 3\ge 0.
\end{array}
\right.
\]
Then for any $\tau\in(0,\tau\sb{3})$
and any $\varepsilon>0$
there exists a solution to the Cauchy problem
for equation \eqref{dkg-c-xt}
with arbitrary initial data $(\psi\sp{0},\psi\sp{1})$,
and this solution is unique.
\end{enumerate}
\end{theorem}

Thus, even though the potential \eqref{poly}
satisfies conditions
\eqref{def-k1} and \eqref{def-k2}
in Theorem~\ref{theorem-e} and Theorem~\ref{theorem-u},
the corresponding values $\tau\sb 1$ and $\tau\sb 2$
could be hard to specify explicitly.
Yet,
the second part of Theorem~\ref{theorem-pol}
gives a simple description
of a class of $\xxd$-dependent polynomials
$V\sb{\xxd}(\lambda)$
for which
the range of admissible $\tau>0$
can be readily specified.

\medskip

We will prove
existence and uniqueness results
stated in
Theorems~\ref{theorem-e}, ~\ref{theorem-u}, and~\ref{theorem-pol}
in Appendix~\ref{sect-eu}.

\subsection{Energy conservation}

\begin{theorem}[Energy conservation]
\label{theorem-energy}
Let $\psi$
be a solution to equation \eqref{dkg-c-xt}
such that
$\psi\sp{\ttd}\in l^2(\Z^n)$ for all $\ttd\in\Z$.
Then the discrete energy
\begin{eqnarray}\label{def-energy-t-xt}
&&
E\sp{\ttd}
=
\sum\sb{\xxd\in\Z^n}
\varepsilon^n
\Big[
\big(
\frac{1}{\tau^2}-\frac{n}{\varepsilon^2}
\big)
\frac{\abs{\psi\sb{\xxd}\sp{\ttd+1}-\psi\sb{\xxd}\sp{\ttd}}^2}{2}
\nonumber
\\
&&
\qquad
+
\sum\sb{j=1}\sp{n}\sum\limits\sb{\pm}
\frac{\abs{\psi\sb{\xxd}\sp{\ttd+1}-\psi\sb{\xxd\pm\e\sb j}\sp{\ttd}}^2
}{4\varepsilon^2}
+
\frac{V\sb{\xxd}(\abs{\psi\sb{\xxd}\sp{\ttd+1}}^2)+V\sb{\xxd}(\abs{\psi\sb{\xxd}\sp{\ttd}}^2)}{2}
\Big]
\end{eqnarray}
is conserved.
\end{theorem}

\begin{remark}
The discrete energy is positive-definite
if the grid ratio satisfies
\begin{equation}\label{grid-ratio-condition}
\frac{\tau}{\varepsilon}\le\frac{1}{\sqrt{n}}.
\end{equation}
\end{remark}

\begin{remark}
If $\psi\sp 0$ and $\psi\sp 1\in l^2(\Z^n)$,
then, by Theorem~\ref{theorem-e},
one also has $\psi\sp{\ttd}\in l^2(\Z^n)$
for all $\ttd\in\Z$
as long as
\[
\inf\sb{\xxd\in\Z^n,\,\lambda\ge 0}
\p\sb\lambda V\sb{\xxd}(\lambda)>-\infty.
\]
\end{remark}

\begin{proof}
For any $u$, $v\in\C$,
there is the identity
\begin{equation}\label{kk}
\abs{u}^2-\abs{v}^2
=
\Re
\left[(\bar u-\bar v)\cdot(u+v)\right].
\end{equation}
Applying \eqref{kk}, one has:
\begin{equation}\label{kkk}
\sum\sb{\xxd\in\Z^n}
\!\!
\big(
\abs{\psi\sb{\xxd}\sp{\ttd+1}-\psi\sb{\xxd}\sp{\ttd}}^2
-\abs{\psi\sb{\xxd}\sp{\ttd}-\psi\sb{\xxd}\sp{\ttd-1}}^2
\big)
=
\Re
\!\!
\sum\sb{\xxd\in\Z^n}
\!\!
\big(
\bar\psi\sb{\xxd}\sp{\ttd+1}-\bar\psi\sb{\xxd}\sp{\ttd-1}
\big)
\cdot
\big(
\psi\sb{\xxd}\sp{\ttd+1}-2\psi\sb{\xxd}\sp{\ttd}
+\psi\sb{\xxd}\sp{\ttd-1}
\big).
\end{equation}
Using \eqref{kk},
we also derive the following identity
for any function
$\psi\sb{\xxd}\sp{\ttd}\in\C$:
\begin{eqnarray}\label{pmpm-xt}
&&
\hskip -10pt
\sum\sb{\xxd\in\Z^n}
\sum\sb{j=1}\sp{n}
\Big[
\abs{\psi\sb{\xxd}\sp{\ttd+1}-\psi\sb{\xxd-\e\sb j}\sp{\ttd}}^2
-\abs{\psi\sb{\xxd-\e\sb j}\sp{\ttd}-\psi\sb{\xxd}\sp{\ttd-1}}^2
\nonumber
\\
&&
\qquad\qquad\qquad\qquad
+\abs{\psi\sb{\xxd}\sp{\ttd+1}-\psi\sb{\xxd+\e\sb j}\sp{\ttd}}^2
-\abs{\psi\sb{\xxd+\e\sb j}\sp{\ttd}-\psi\sb{\xxd}\sp{\ttd-1}}^2
\Big]
\nonumber
\\
&&
=\Re\sum\sb{\xxd\in\Z^n}
\sum\sb{j=1}\sp{n}
\Big[
(\bar\psi\sb{\xxd}\sp{\ttd+1}-\bar\psi\sb{\xxd}\sp{\ttd-1})
\cdot
(\psi\sb{\xxd}\sp{\ttd+1}-2\psi\sb{\xxd\pm\e\sb j}\sp{\ttd}+\psi\sb{\xxd}\sp{\ttd-1})
\nonumber
\\
&&
\qquad\qquad\qquad\qquad
+
(\bar\psi\sb{\xxd}\sp{\ttd+1}-\bar\psi\sb{\xxd}\sp{\ttd-1})
\cdot
(\psi\sb{\xxd}\sp{\ttd+1}-2\psi\sb{\xxd+\e\sb j}\sp{\ttd}+\psi\sb{\xxd}\sp{\ttd-1})
\Big]
\nonumber
\\
&&
=\Re\sum\sb{\xxd\in\Z^n}
(\bar\psi\sb{\xxd}\sp{\ttd+1}-\bar\psi\sb{\xxd}\sp{\ttd-1})
\cdot
\Big[
2n
\big(\psi\sb{\xxd}\sp{\ttd+1}
-2\psi\sb{\xxd}\sp{\ttd}
+\psi\sb{\xxd}\sp{\ttd-1}
\big)
\nonumber
\\
&&
\qquad\qquad\qquad\qquad\qquad\qquad\quad
-
2
\sum\sb{j=1}\sp{n}
\big(\psi\sb{\xxd+\e\sb j}\sp{\ttd}
-2\psi\sb{\xxd}\sp{\ttd}
+\psi\sb{\xxd-\e\sb j}\sp{\ttd}
\big)
\Big].
\end{eqnarray}
Further, \eqref{vmv}
together with \eqref{kk}
imply that
\begin{eqnarray}\label{vmvmv}
&&
V\sb{\xxd}(\abs{\psi\sb{\xxd}\sp{\ttd+1}}^2)-V\sb{\xxd}(\abs{\psi\sb{\xxd}\sp{\ttd-1}}^2)
\nonumber
\\
&&
\quad
=
\Re
\big[(\bar\psi\sb{\xxd}\sp{\ttd+1}-\bar\psi\sb{\xxd}\sp{\ttd-1})
\cdot(\psi\sb{\xxd}\sp{\ttd+1}+\psi\sb{\xxd}\sp{\ttd-1})
\big]
B\sb{\xxd}(\abs{\psi\sb{\xxd}\sp{\ttd+1}}^2,\abs{\psi\sb{\xxd}\sp{\ttd-1}}^2).
\end{eqnarray}
Taking into account \eqref{kkk}, \eqref{pmpm-xt}, and \eqref{vmvmv},
we compute:
\begin{eqnarray}
&&
\frac{E\sp{\ttd}-E\sp{\ttd-1}}{\varepsilon^n}
=
\sum\sb{\xxd\in\Z^n}
\Big[
\Big(
\frac{1}{\tau^2}-\frac{n}{\varepsilon^2}
\Big)
\frac{
\abs{\psi\sb{\xxd}\sp{\ttd+1}-\psi\sb{\xxd}\sp{\ttd}}^2
-\abs{\psi\sb{\xxd}\sp{\ttd}-\psi\sb{\xxd}\sp{\ttd-1}}^2
}{2}
\nonumber\\
&&
\qquad\qquad\qquad
+
\sum\sb{j=1}\sp{n}
\sum\limits\sb{\pm}
\frac{
\abs{\psi\sb{\xxd}\sp{\ttd+1}-\psi\sb{\xxd\pm\e\sb j}\sp{\ttd}}^2
-\abs{\psi\sb{\xxd\pm\e\sb j}\sp{\ttd}-\psi\sb{\xxd}\sp{\ttd-1}}^2
}{4\varepsilon^2}
\nonumber\\
&&
\qquad\qquad\qquad
+
\frac{V\sb{\xxd}(\abs{\psi\sb{\xxd}\sp{\ttd+1}}^2)-V\sb{\xxd}(\abs{\psi\sb{\xxd}\sp{\ttd-1}}^2)}{2}
\Big]
\nonumber\\
&&
=
\Re\sum\sb{\xxd\in\Z^n}
(\bar\psi\sb{\xxd}\sp{\ttd+1}-\bar\psi\sb{\xxd}\sp{\ttd-1})
\cdot\Big[
\Big(
\frac{1}{\tau^2}-\frac{n}{\varepsilon^2}
\Big)
\frac{\psi\sb{\xxd}\sp{\ttd+1}-2\psi\sb{\xxd}\sp{\ttd}+\psi\sb{\xxd}\sp{\ttd-1}}{2}
\nonumber\\
&&
\qquad\qquad
+
\frac{
n
\big(\psi\sb{\xxd}\sp{\ttd+1}-2\psi\sb{\xxd}\sp{\ttd}+\psi\sb{\xxd}\sp{\ttd-1}\big)
-
\sum\limits\sb{j=1}\sp{n}
\big(\psi\sb{\xxd+\e\sb j}\sp{\ttd}
-2\psi\sb{\xxd}\sp{\ttd}
+\psi\sb{\xxd-\e\sb j}\sp{\ttd}
\big)
}{2\varepsilon^2}
\nonumber\\
&&
\qquad\qquad
+
\frac{\psi\sb{\xxd}\sp{\ttd+1}+\psi\sb{\xxd}\sp{\ttd-1}}{2}
B\sb{\xxd}(\abs{\psi\sb{\xxd}\sp{\ttd+1}}^2,\abs{\psi\sb{\xxd}\sp{\ttd-1}}^2)
\Big]
.
\nonumber
\end{eqnarray}
The expression in the square brackets
adds up to zero due to \eqref{dkg-c-xt}.
We conclude that $E\sp{\ttd}=E\sp{\ttd-1}$
for all $\ttd\in\Z$.
\end{proof}

\subsection{A priori estimates}

\begin{theorem}[A priori estimates]
\label{theorem-a-priori}
Assume that $\varepsilon>0$ and $\tau>0$ satisfy
\[
\frac{\tau}{\varepsilon}\le\frac{1}{\sqrt{n}}.
\]
Assume that
\begin{equation}\label{w-nlkg}
V\sb{\xxd}(\lambda)=\frac{m^2}{2}\lambda+W\sb{\xxd}(\lambda),
\end{equation}
where $m>0$,
and for each $\xxd\in\Z^n$
the function
$W\sb{\xxd}\in C\sp{2}(\R)$
satisfies
$W\sb{\xxd}(\lambda)\ge 0$
for $\lambda\ge 0$.
Then any solution $\psi\sb{\xxd}\sp{\ttd}$ to the Cauchy problem
\eqref{dkg-c-xt}
with arbitrary initial data
$(\psi\sp{0},\psi\sp{1})\in l^2(\Z^n)\times l^2(\Z^n)$
satisfies the \emph{a priori} estimate
\begin{equation}\label{ape}
\varepsilon^n\norm{\psi\sp{\ttd}}\sb{l^2}^2
\le\frac{4E\sp{0}}{m^2},
\end{equation}
where $E\sp{0}$
is the energy \eqref{def-energy-t-xt}
of the solution $\psi\sb{\xxd}\sp{\ttd}$
at the moment $\ttd=0$.
\end{theorem}

\begin{proof}
This immediately follows from the conservation
of the energy
\eqref{def-energy-t-xt}
with $V\sb{\xxd}(\lambda)$
given by \eqref{w-nlkg},
\[
E\sp{\ttd}
=
\sum\sb{\xxd\in\Z^n}
\varepsilon^n
\Big[
\Big(
\frac{1}{\tau^2}-\frac{n}{\varepsilon^2}
\Big)
\frac{\abs{\psi\sb{\xxd}\sp{\ttd+1}-\psi\sb{\xxd}\sp{\ttd}}^2}{2}
+
\sum\sb{j=1}\sp{n}
\frac{\abs{\psi\sb{\xxd}\sp{\ttd+1}-\psi\sb{\xxd-\e\sb j}\sp{\ttd}}^2
+\abs{\psi\sb{\xxd}\sp{\ttd+1}-\psi\sb{\xxd+\e\sb j}\sp{\ttd}}^2
}{4\varepsilon^2}
\]
\[
+
\frac{m^2(\abs{\psi\sb{\xxd}\sp{\ttd+1}}^2+\abs{\psi\sb{\xxd}\sp{\ttd}}^2)}{4}
+\frac{W\sb{\xxd}(\abs{\psi\sb{\xxd}\sp{\ttd+1}}^2)+W\sb{\xxd}(\abs{\psi\sb{\xxd}\sp{\ttd}}^2)}{2}
\Big]
.
\]
\end{proof}

\begin{remark}
In the continuous limit
$\varepsilon\to 0$,
the relation
\eqref{ape}
is similar to the a priori estimate
\eqref{apec}
for the solutions
to the continuous nonlinear Klein-Gordon equation \eqref{nlkg}.
\end{remark}

\begin{remark}
In
\cite{MR0503140},
in the case $\psi\sb{\xxd}\sp{\ttd}\in\R$, $(\xxd,\ttd)\in\Z\times\Z$
(in the dimension $n=1$),
the following expression
for the discretized energy was introduced:
\begin{eqnarray}\label{def-energy-sv}
&
\displaystyle
E\sb{SV}\sp{\ttd}
=
\frac{1}{2}
\sum\sb{\xxd\in\Z^n}
\Big[
\frac{(\psi\sb{\xxd}\sp{\ttd+1}-\psi\sb{\xxd}\sp{\ttd})^2}
{\tau^2}
+
\frac{(\psi\sb{\xxd+1}\sp{\ttd+1}-\psi\sb{\xxd}\sp{\ttd+1})(\psi\sb{\xxd+1}\sp{\ttd}-\psi\sb{\xxd}\sp{\ttd})}
{\varepsilon^2}
\nonumber
\\
&
+V(\abs{\psi\sb{\xxd}\sp{\ttd+1}}^2)+V(\abs{\psi\sb{\xxd}\sp{\ttd}}^2)
\Big].
\end{eqnarray}
The presence of the second term which is
not positive-definite
deprives one of the a priori $l^2$ bound
on $\psi$,
such as the one
stated in Theorem~\ref{theorem-a-priori}.
In view of this,
the Strauss-Vazquez finite-difference scheme
for the nonlinear Klein-Gordon equation
is not \emph{unconditionally stable}.
Other schemes
(conditionally and unconditionally stable)
were proposed in \cite{MR1360462,MR1852556}.
Now,
due to the a priori bound \eqref{ape},
we deduce that, as the matter of fact,
the Strauss-Vazquez scheme
is stable
in $n$ dimensions
under the condition that the grid ratio is
$\tau/\varepsilon\le 1/\sqrt{n}$.
Note that in the case $\psi\in\R$,
the Strauss-Vazquez energy \eqref{def-energy-sv}
agrees with the energy defined in \eqref{def-energy-t-xt}.
\end{remark}

\subsection{The charge conservation}

Let us consider the charge conservation.
We will define the discrete charge
under the following assumption:

\begin{assumption}\label{ass-alpha-n-xt}
\[
\frac{\tau}{\varepsilon}
=\frac{1}{\sqrt{n}}.
\]
\end{assumption}
Under Assumption~\ref{ass-alpha-n-xt},
$\psi\sb{\xxd}\sp{\ttd}$ drops out of equation \eqref{dkg-c-xt};
the latter can be written as
\begin{equation}\label{dkg-fn}
\big(\psi\sb{\xxd}\sp{\ttd+1}+\psi\sb{\xxd}\sp{\ttd-1}\big)
\big(
1+
\tau^2
B\sb{\xxd}(\abs{\psi\sb{\xxd}\sp{\ttd+1}}^2,\abs{\psi\sb{\xxd}\sp{\ttd-1}}^2)
\big)
=
\frac{1}{n}
\sum\sb{j=1}\sp{n}
(\psi\sb{\xxd+\e\sb j}\sp{\ttd}+\psi\sb{\xxd-\e\sb j}\sp{\ttd}).
\end{equation}

\begin{theorem}[Charge conservation]
\label{theorem-charge}
Let Assumption~\ref{ass-alpha-n-xt} be satisfied.
Let $\psi$ be a solution to equation \eqref{dkg-fn}
such that
$\psi\sp{\ttd}\in l^2(\Z^n)$ for all $\ttd\in\Z$
(see Theorem~\ref{theorem-e}).
Then the discrete charge
\begin{equation}\label{def-charge-t}
Q\sp{\ttd}
=
\frac{i}{4\tau}
\sum\sb{\xxd\in\Z^n}
\varepsilon^n
\big[
\bar\psi\sb{\xxd+\e\sb j}\sp{\ttd}\cdot\psi\sb{\xxd}\sp{\ttd+1}
+\bar\psi\sb{\xxd-\e\sb j}\sp{\ttd}\cdot\psi\sb{\xxd}\sp{\ttd+1}
-\bar\psi\sb{\xxd}\sp{\ttd+1}\cdot\psi\sb{\xxd+\e\sb j}\sp{\ttd}
-\bar\psi\sb{\xxd}\sp{\ttd+1}\cdot\psi\sb{\xxd-\e\sb j}\sp{\ttd}
\big]
\end{equation}
is conserved.
\end{theorem}

\begin{remark}
The continuous limit
of the discrete charge $Q$
defined in
\eqref{def-charge-t}
coincides with the charge functional
\eqref{def-charge}
of the continuous nonlinear wave equation \eqref{nlw-xt}.
\end{remark}

\begin{proof}
Let us prove the charge conservation.
One has:
\[
\frac{4\,\tau}{i\varepsilon^n}Q\sp{\ttd}
=\sum\sb{\xxd\in\Z^n}
\sum\sb{j=1}\sp{n}
\sum\sb{\pm}
\Big[
\bar\psi\sb{\xxd\pm\e\sb j}\sp{\ttd}\cdot\psi\sb{\xxd}\sp{\ttd+1}
-\mathrm{c.\,c.\,}
\Big],
\]
\[
\frac{4\,\tau}{i\varepsilon^n}Q\sp{\ttd-1}
=
\sum\sb{\xxd\in\Z^n}
\sum\sb{j=1}\sp{n}
\sum\sb{\pm}
\Big[
\bar\psi\sb{\xxd\pm\e\sb j}\sp{\ttd-1}\cdot\psi\sb{\xxd}\sp{\ttd}
-\mathrm{c.\,c.\,}
\Big]
=
-
\sum\sb{\xxd\in\Z^n}
\sum\sb{j=1}\sp{n}
\sum\sb{\pm}
\Big[
\bar\psi\sb{\xxd\pm\e\sb j}\sp{\ttd}\cdot\psi\sb{\xxd}\sp{\ttd-1}
-\mathrm{c.\,c.\,}
\Big],
\]
where $\mathrm{c.\,c.\,}$
denotes the complex conjugation of the preceding expression.
Therefore,
\begin{eqnarray}
&&\frac{4\,\tau\big(Q\sp{\ttd}-Q\sp{\ttd-1}\big)}{i\varepsilon^n}
=
\sum\sb{\xxd\in\Z^n}
\sum\sb{j=1}\sp{n}
\sum\sb{\pm}
\bar\psi\sb{\xxd\pm\e\sb j}\sp{\ttd}
\cdot(\psi\sb{\xxd}\sp{\ttd+1}+\psi\sb{\xxd}\sp{\ttd-1})
-\mathrm{c.\,c.\,}
\nonumber
\\
&&
\qquad
=
n\sum\sb{\xxd\in\Z^n}
\big(1+\tau^2 B\sb{\xxd}(\abs{\psi\sb{\xxd}\sp{\ttd+1}}^2,\abs{\psi\sb{\xxd}\sp{\ttd-1}}^2)\big)
\abs{\psi\sb{\xxd}\sp{\ttd+1}+\psi\sb{\xxd}\sp{\ttd-1}}^2
-\mathrm{c.\,c.\,}
=0.
\nonumber
\end{eqnarray}
To get to the second line,
we used the complex conjugate of \eqref{dkg-fn}.
This finishes the proof of Theorem~\ref{theorem-charge}.
\end{proof}

\subsection{Proof of well-posedness}
\label{sect-eu}

First, we prove the existence.
\begin{proof}[Proof of Theorem~\ref{theorem-e}]
We rewrite equation \eqref{dkg-c-xt}
in the following form:
\begin{eqnarray}\label{dkg-f}
&&
\big(\psi\sb{\xxd}\sp{\ttd+1}+\psi\sb{\xxd}\sp{\ttd-1}\big)
\big(1
+\tau^2
B\sb{\xxd}(\abs{\psi\sb{\xxd}\sp{\ttd+1}}^2,\abs{\psi\sb{\xxd}\sp{\ttd-1}}^2)
\big)
\nonumber
\\
&&
\quad
=
\frac{\tau^2}{\varepsilon^2}
\sum\sb{j=1}\sp{n}
\big(
\psi\sb{\xxd+\e\sb j}\sp{\ttd}-2\psi\sb{\xxd}\sp{\ttd}+\psi\sb{\xxd-\e\sb j}\sp{\ttd}
\big)
+2\psi\sb{\xxd}\sp{\ttd},
\quad
\xxd\in\Z^n,
\quad
\ttd\in\Z.
\end{eqnarray}
By \eqref{def-k1}
and the choice of $\tau\sb{1}$ in Theorem~\ref{theorem-e},
for $\tau\in(0,\tau\sb{1})$
one has
\begin{equation}\label{w-pos}
\inf\sb{\xxd\in\Z^n,\lambda\ge 0}
\big(1+\tau^2\p\sb\lambda V\sb{\xxd}(\lambda)\big)>0.
\end{equation}
Since
\begin{equation}\label{inf-b-w}
\inf\sb{
\begin{array}{c}
\scriptstyle
\xxd\in\Z^n
\\
\scriptstyle
\lambda\ge 0,\,\mu\ge 0
\end{array}
}\!\!
B\sb{\xxd}(\lambda,\mu)
=
\inf\sb{
\begin{array}{c}
\scriptstyle\xxd\in\Z^n
\\
\scriptstyle\lambda\ge 0,\,\mu\ge 0,\,\lambda\ne\mu
\end{array}
}
\!\!\!\!\!\!
\frac{V\sb{\xxd}(\lambda)-V\sb{\xxd}(\mu)}{\lambda-\mu}
=
\inf\sb{\xxd\in\Z^n,\,\lambda\ge 0}\p\sb\lambda V\sb{\xxd}(\lambda),
\end{equation}
inequality \eqref{w-pos}
yields
\begin{equation}\label{1e2b}
c:=\inf\sb{\xxd\in\Z^n,\lambda\ge 0,\mu\ge 0}
\big(1+\tau^2 B\sb{\xxd}(\lambda,\mu)\big)
>0.
\end{equation}

Let us show that
equation \eqref{dkg-f} allows us to find $\psi\sb{\xxd}\sp{\ttd+1}$,
for any given $\xxd\in \Z^n$ and $\ttd\in\Z$,
once one knows $\psi\sp{\ttd}$ and $\psi\sp{\ttd-1}$.
Equation \eqref{dkg-f} implies that
\begin{equation}\label{eqn-on-xi}
\big(1+\tau^2 B\sb{\xxd}(\abs{\psi\sb{\xxd}\sp{\ttd+1}}^2,\abs{\psi\sb{\xxd}\sp{\ttd-1}}^2)\big)
(\psi\sb{\xxd}\sp{\ttd+1}+\psi\sb{\xxd}\sp{\ttd-1})
=\xi\sb{\xxd}\sp{\ttd},
\end{equation}
\begin{equation}\label{def-xi}
\xi\sb{\xxd}\sp{\ttd}:=
\frac{\tau^2}{\varepsilon^2}
\sum\sb{j=1}\sp{n}
\big(\psi\sb{\xxd+\e\sb j}\sp{\ttd}-2\psi\sb{\xxd}\sp{\ttd}+\psi\sb{\xxd-\e\sb j}\sp{\ttd}\big)
+2\psi\sb{\xxd}\sp{\ttd}\in\C.
\end{equation}
If $\xi\sb{\xxd}\sp{\ttd}=0$, then there is a solution to \eqref{eqn-on-xi}
given by
$\psi\sb{\xxd}\sp{\ttd+1}=-\psi\sb{\xxd}\sp{\ttd-1}$.
Due to \eqref{1e2b},
this solution is unique.
Now let us assume that $\xi\sb{\xxd}\sp{\ttd}\ne 0$.
We see from \eqref{eqn-on-xi}
that we are to have
\begin{equation}\label{xi-eta-zeta}
\psi\sb{\xxd}\sp{\ttd+1}+\psi\sb{\xxd}\sp{\ttd-1}=s \xi\sb{\xxd}\sp{\ttd},
\qquad
\mbox{ with some $s\in\R$}.
\end{equation}
Let us introduce  the function
\begin{equation}\label{def-fs}
f(s)
:=\big(1+\tau^2 B\sb{\xxd}(\abs{s \xi\sb{\xxd}\sp{\ttd}-\psi\sb{\xxd}\sp{\ttd-1}}^2,\abs{\psi\sb{\xxd}\sp{\ttd-1}}^2)
\big)s.
\end{equation}
We do not indicate dependence of
$f$ on $\psi\sb{\xxd}\sp{\ttd-1}$, $\xi\sb{\xxd}\sp{\ttd}$, and $\xxd$,
treating them as parameters.
For $\xi\sb{\xxd}\sp{\ttd}\ne 0$,
we can solve \eqref{eqn-on-xi}
if we can find $s\in\R$ such that
\begin{equation}\label{eqn-on-s}
f(s)=1.
\end{equation}
Since $f(0)=0$,
while
$
\lim\sb{s\to\infty}
f(s)=+\infty
$
by \eqref{1e2b},
one concludes that there
is at least one solution
$s>0$
to \eqref{eqn-on-s}.

Let us prove that
once
$(\psi\sp{0},\psi\sp{1})\in l^2(\Z^n)\times l^2(\Z^n)$,
then one also knows that
$\norm{\psi\sp{\ttd}}\sb{l^2(\Z^n)}$
remains finite
(but not necessarily uniformly bounded)
for all $\ttd\in\Z$.
As it follows from \eqref{1e2b} and \eqref{eqn-on-xi},
\begin{equation}\label{psi-recurrent}
\abs{\psi\sb{\xxd}\sp{\ttd+1}}
\le
\frac{1}{c}
\abs{\xi\sb{\xxd}\sp{\ttd}}
+\abs{\psi\sb{\xxd}\sp{\ttd-1}}.
\end{equation}
Since
$\norm{\xi\sp{\ttd}}\sb{l^2(\Z^n)}
\le
\left(\frac{4\tau^2}{\varepsilon^2}+2\right)
\norm{\psi\sp{\ttd}}\sb{l^2(\Z^n)}
$
by \eqref{def-xi},
the relation
\eqref{psi-recurrent}
implies the estimate
\begin{equation}
\norm{\psi\sp{\ttd+1}}\sb{l^2(\Z^n)}
\le
\frac{1}{c}
\left(\frac{4\tau^2}{\varepsilon^2}+2\right)
\norm{\psi\sp{\ttd}}\sb{l^2(\Z^n)}
+\norm{\psi\sp{\ttd-1}}\sb{l^2(\Z^n)},
\end{equation}
and, by recursion,
the finiteness of $\norm{\psi\sp{\ttd}}\sb{l^2(\Z^n)}$
for all $\ttd\ge 0$.
The case $\ttd\le 0$ is finished in the same way.
\end{proof}

Now we prove the uniqueness
of solutions
to the Cauchy problem for equation \eqref{dkg-c-xt}
and the continuous dependence on the initial data.

\begin{proof}[Proof of Theorem~\ref{theorem-u}]
First, note that,
by Remark~\ref{remark-k2-k3},
$\tau\sb{1}$ from Theorem~\ref{theorem-e}
and $\tau\sb{2}$ from Theorem~\ref{theorem-u}
are related by
$\tau\sb{2}\le\tau\sb{1}$.
Therefore, the existence of a solution
$\psi\sb{\xxd}\sp{\ttd}$ to
the Cauchy problem for equation \eqref{dkg-c-xt}
follows from Theorem~\ref{theorem-e}.

Let us prove that this solution
$\psi\sb{\xxd}\sp{\ttd}$ is unique.
When in \eqref{def-xi}
one has
\[
\xi\sb{\xxd}\sp{\ttd}:=\frac{\tau^2}{\varepsilon^2}\sum\sb{j=1}\sp{n}
(\psi\sb{\xxd+\e\sb{j}}\sp{\ttd}-2\psi\sb{\xxd}\sp{\ttd}+\psi\sb{\xxd-\e\sb{j}}\sp{\ttd})
+2\psi\sb{\xxd}\sp{\ttd}=0,
\]
then,
by \eqref{1e2b},
the only solution $\psi\sb{\xxd}\sp{\ttd+1}$
to \eqref{eqn-on-xi}
is given by $\psi\sb{\xxd}\sp{\ttd+1}=-\psi\sb{\xxd}\sp{\ttd-1}$.
We now consider the case $\xi\sb{\xxd}\sp{\ttd}\ne 0$.
By \eqref{eqn-on-xi}, \eqref{xi-eta-zeta},
and \eqref{def-fs},
it suffices to prove the uniqueness of the solution
to \eqref{eqn-on-s}.
This will follow if we show that
$f(s)$ satisfies
\begin{equation}\label{fp}
f'(s)>0,
\qquad
s\in\R.
\end{equation}
The explicit expression for $f'(s)$ is
\begin{eqnarray}\label{def-fp-0}
&&
f'(s)
=
1
+\tau^2 B\sb{\xxd}(\abs{s \xi\sb{\xxd}\sp{\ttd}-\psi\sb{\xxd}\sp{\ttd-1}}^2,\abs{\psi\sb{\xxd}\sp{\ttd-1}}^2)
\\
&&
\qquad
+\tau^2 \p\sb\lambda B\sb{\xxd}(\abs{s \xi\sb{\xxd}\sp{\ttd}-\psi\sb{\xxd}\sp{\ttd-1}}^2,\abs{\psi\sb{\xxd}\sp{\ttd-1}}^2)
(-2\Re(\bar\psi\sb{\xxd}\sp{\ttd-1}\cdot\xi\sb{\xxd}\sp{\ttd})+2\abs{\xi\sb{\xxd}\sp{\ttd}}^2 s)s.
\nonumber
\end{eqnarray}
Using the relation \eqref{xi-eta-zeta},
we derive the identity
\begin{eqnarray}
(-2\Re(\bar\psi\sb{\xxd}\sp{\ttd-1}\cdot\xi\sb{\xxd}\sp{\ttd})+2\abs{\xi\sb{\xxd}\sp{\ttd}}^2 s)s
=2\Abs{s \xi\sb{\xxd}\sp{\ttd}-\frac{\psi\sb{\xxd}\sp{\ttd-1}}{2}}^2-\frac{\abs{\psi\sb{\xxd}\sp{\ttd-1}}^2}{2}
\nonumber
\\
=2\Abs{\psi\sb{\xxd}\sp{\ttd+1}+\frac{\psi\sb{\xxd}\sp{\ttd-1}}{2}}^2-\frac{\abs{\psi\sb{\xxd}\sp{\ttd-1}}^2}{2}
\nonumber
\end{eqnarray}
and rewrite the expression \eqref{def-fp-0} for $f'(s)$ as
\begin{eqnarray}
&&
f'(s)
=
1+\tau^2
\Big[
B\sb{\xxd}(\abs{\psi\sb{\xxd}\sp{\ttd+1}}^2,\abs{\psi\sb{\xxd}\sp{\ttd-1}}^2)
\nonumber
\\
&&
\qquad
+2\p\sb\lambda B\sb{\xxd}(\abs{\psi\sb{\xxd}\sp{\ttd+1}}^2,\abs{\psi\sb{\xxd}\sp{\ttd-1}}^2)
\Big(
\abs{\psi\sb{\xxd}\sp{\ttd+1}+\frac{\psi\sb{\xxd}\sp{\ttd-1}}{2}}^2
-\frac{\abs{\psi\sb{\xxd}\sp{\ttd-1}}^2}{4}
\Big)
\Big].
\label{def-fp-1}
\end{eqnarray}
We denote $\lambda=\abs{\psi\sb{\xxd}\sp{\ttd+1}}^2$,
$\mu=\abs{\psi\sb{\xxd}\sp{\ttd-1}}^2$.
Since
\[
\lambda-\sqrt{\lambda\mu}+\frac{\mu}{4}
\le
\abs{\psi\sb{\xxd}\sp{\ttd+1}+\frac{\psi\sb{\xxd}\sp{\ttd-1}}{2}}^2
\le\lambda+\sqrt{\lambda\mu}+\frac{\mu}{4},
\]
we see that
\begin{equation}\label{def-fp-2}
f'(s)\ge 1+\tau^2\min\sb{\pm}
\inf\sb{\xxd\in\Z^n,\,\lambda\ge 0,\,\mu\ge 0}K\sb{\xxd}\sp\pm(\lambda,\mu),
\end{equation}
where
\[
K\sb{\xxd}\sp\pm(\lambda,\mu)
=
B\sb{\xxd}(\lambda,\mu)
+2\p\sb\lambda B\sb{\xxd}(\lambda,\mu)
(\lambda\pm\sqrt{\lambda\mu}).
\]

By
\eqref{def-k2}
and 
by our choice of $\tau\sb{2}$
in Theorem~\ref{theorem-u},
for any $\tau\in(0,\tau\sb{2})$
we have
\[
\kappa:=\inf\sb{\xxd\in\Z^n,\,\lambda\ge 0,\,\mu\ge 0}
\left\{1+\tau^2 K\sb{\xxd}\sp\pm(\lambda,\mu)\right\}>0;
\]
then,
by \eqref{def-fp-2},
$f'(s)\ge \kappa$,
where $\kappa>0$.
It follows that
for $\xi\sb{\xxd}\sp{\ttd}\ne 0$
there is a unique solution $s$
to \eqref{eqn-on-s};
moreover, this solution $s$ continuously depends
on $\psi\sp{\ttd-1}\sb\xxd$
and on
$\xi\sp\ttd\sb\xxd$.
The latter
in turn continuously depends on
$\psi\sp\ttd\sb{\xxd\pm\bm{e}\sb j}$, $1\le j\le n$
(Cf. \eqref{def-xi}).
Thus,
by \eqref{xi-eta-zeta},
the solution $\psi\sb{\xxd}\sp{\ttd+1}$
to equation \eqref{eqn-on-xi}
is uniquely defined
and continuously depends on
$\psi\sb{\xxd}\sp{\ttd-1}$
and
$\psi\sb{\xxd\pm\bm{e}\sb j}\sp{\ttd}$,
$1\le j\le n$.

This finishes the proof of Theorem~\ref{theorem-u}.
\end{proof}

\begin{proof}[Proof of Theorem~\ref{theorem-pol}]
Let us prove that the condition
\eqref{def-k2} in Theorem~\ref{theorem-u}
is satisfied by any polynomial potential
of the form \eqref{poly}.
The inequality \eqref{def-k2}
will be satisfied if the highest order
term from $V(\lambda)$
contributes a strictly positive expression.
More precisely, we need to prove the following result.

\begin{lemma}\label{lemma-31}
Let $V(\lambda)=\lambda^{p+1}$,
so that
$B(\lambda,\mu)=\frac{\lambda^{p+1}-\mu^{p+1}}{\lambda-\mu}$,
$p\ge 0$.
Then the following inequality takes place:
\begin{equation}\label{wn}
\inf\sb{\lambda\ge 0,\,\mu\ge 0,\,\lambda^2+\mu^2=1}
\Big[
B(\lambda,\mu)
+2\p\sb\lambda B(\lambda,\mu)(\lambda\pm\sqrt{\lambda\mu})
\Big]>0.
\end{equation}
\end{lemma}

\begin{proof}
Since $B$ and $\p\sb\lambda B$
are strictly positive
for $\lambda^2+\mu^2>0$,
the inequality \eqref{wn}
is nontrivial
only for the negative sign in \eqref{wn}
and only when $\mu>\lambda$.
First we  note that
\[
B(\lambda,\mu)=\frac{\mu^{p+1}-\lambda^{p+1}}{\mu-\lambda},
\]
\[
\p\sb\lambda B(\lambda,\mu)
=
\frac{-(p+1)\lambda^p(\mu-\lambda)-\lambda^{p+1}+\mu^{p+1}}
{(\mu-\lambda)^2}
=
\frac{\mu^{p+1}-(p+1)\lambda^p\mu+p\lambda^{p+1}}
{(\mu-\lambda)^2}.
\]
Let $z\ge 0$ be such that
$z^2=\lambda/\mu$.
To prove the lemma, we need to check that 
\begin{equation}\label{nee}
\frac{1-z^{2p+2}}{1-z^2}
+2\frac{1-(p+1)z^{2p}+pz^{2p+2}}{(1-z^2)^2}
(z^2-z)>0,
\qquad
0\le z<1,
\end{equation}
or equivalently,
\[
(1+z)(1-z^{2p+2})-2z\big(1-(p+1)z^{2p}+pz^{2p+2}\big)>0.
\]
The left-hand side takes the form
\begin{eqnarray}
&(1+z)(1-z^{2p+2})-2z\big(1-z^{2p+2}-(p+1)(z^{2p}-z^{2p+2})\big)
\nonumber
\\
&=(1-z)(1-z^{2p+2})
+2z(p+1)(z^{2p}-z^{2p+2}),
\nonumber
\end{eqnarray}
which is clearly strictly positive
for all $0\le z<1$ and $p\ge 0$,
proving \eqref{nee}.
\end{proof}

This finishes the proof of
the first part of Theorem~\ref{theorem-pol};
now we turn to the second part.

\begin{lemma}[Uniqueness criterion]
\label{lemma-u-0}
Assume that for
a particular $\tau>0$
and for all $\lambda\ge 0$, $\mu\ge 0$, $\xxd\in\Z^n$,
the following inequalities hold:
\begin{equation}\label{cond-on-b}
1+
\tau^2
\inf\sb{\xxd\in\Z^n,\,\lambda\ge 0,\,\mu\ge 0}
\Big(B\sb{\xxd}(\lambda,\mu)-\p\sb\lambda B\sb{\xxd}(\lambda,\mu)\frac{\mu}{2}\Big)
>0;
\end{equation}
\begin{equation}\label{pb-positive}
\inf\sb{\xxd\in\Z^n,\,\lambda\ge 0,\,\mu\ge 0}
\p\sb\lambda B\sb{\xxd}(\lambda,\mu)\ge 0.
\end{equation}
Then the is a solution $\psi\sb{\xxd}\sp{\ttd}$
to the Cauchy problem for equation \eqref{dkg-c-xt}
with arbitrary initial data $(\psi\sp{0},\psi\sp{1})$, 
and this solution is unique.
\end{lemma}

\begin{proof}[Proof of Lemma~\ref{lemma-u-0}]
The inequalities
\eqref{cond-on-b} and \eqref{pb-positive}
lead to
\[
1+\tau^2\inf\sb{\xxd\in\Z^n,\,\lambda\ge 0}
B\sb{\xxd}(\lambda,\lambda)>0,
\]
hence,
by the same argument as in Theorem~\ref{theorem-e},
there is a solution $\psi\sb{\xxd}\sp{\ttd}$.
The relation
\eqref{def-fp-1}
shows that $f'(s)\ge c$ for some $c>0$.
The rest of the proof is the same as for Theorem~\ref{theorem-u}.
\end{proof}

In the second part of Theorem~\ref{theorem-pol},
we assume that
\begin{equation}
V\sb{\xxd}(\lambda)=\sum\sb{q=0}\sp{4}C\sb{\xxd,q}\lambda^{q+1},
\qquad
\xxd\in\Z^n,\quad\lambda\ge 0,
\end{equation}
where
$C\sb{\xxd,q}\ge 0$ for $\xxd\in\Z^n$ and $1\le q\le 4$,
and
\begin{equation}\label{def-k3-1}
k\sb 3=
\inf\sb{\xxd\in\Z^n}
C\sb{\xxd,0}>-\infty.
\end{equation}
Thus,
the term $C\sb{\xxd,0}\lambda$ in $V\sb{\xxd}(\lambda)$
contributes to $B\sb{\xxd}(\lambda,\mu)$
the expression
$b\sb{\xxd,0}(\lambda,\mu)=C\sb{\xxd,0}$,
while each term
in $V\sb{\xxd}(\lambda)$
of the form
$
C\sb{\xxd,q}\lambda^{q+1},
$
with
$1\le q\le 4$
and $C\sb{\xxd,q}\ge 0$,
contributes to $B\sb{\xxd}(\lambda,\mu)$
the expression
$C\sb{\xxd,q} b\sb{q}(\lambda,\mu)$,
with
$b\sb{q}(\lambda,\mu)=\sum\sb{k=0}\sp{q}\lambda^{q-k}\mu^k$.
For $\tau\in(0,\tau\sb{3})$,
with $\tau\sb{3}=\sqrt{-1/k\sb 3}$
for $k\sb 3<0$
and $\tau\sb{3}=+\infty$
for $k\sb 3\ge 0$,
one has
\begin{equation}\label{1ec}
1+\tau^2\inf\sb{\xxd\in\Z^n}C\sb{\xxd,0}>0.
\end{equation}

\begin{lemma}
\label{lemma-four}
For $1\le q\le 4$,
$b\sb{q}(\lambda,\mu)=\sum\sb{k=0}\sp{q}\lambda^{q-k}\mu^k$
satisfies the inequality
\[
b\sb{q}(\lambda,\mu)
\ge\p\sb\lambda b\sb{q}(\lambda,\mu)\frac{\mu}{2}
\qquad
\mbox{for all \ $\lambda,\,\mu\ge 0$.}
\]
\end{lemma}

By \eqref{1ec} and Lemma~\ref{lemma-four},
condition
\eqref{cond-on-b}
is satisfied.
Since $C\sb{\xxd,q}\ge 0$ for $1\le q\le 4$,
each term
$C\sb{\xxd,q}b\sb q(\lambda,\mu)$
satisfies
condition
\eqref{pb-positive}.
Therefore, by Lemma~\ref{lemma-u-0},
there is a unique solution
$\psi\sb{\xxd}\sp{\ttd}$ to the Cauchy problem for equation \eqref{dkg-c-xt}.
This finishes the proof of Theorem~\ref{theorem-pol}.
\end{proof}


\section{Titchmarsh convolution theorem
for distributions on the circle}
\label{sect-titchmarsh-circle}

The Titchmarsh convolution theorem \cite{titchmarsh}
states that
for any two compactly supported distributions
$f,\,g\in\mathscr{E}'(\R)$,
\begin{equation}\label{titchmarsh-theorem}
\inf\supp f\ast g
=\inf\supp f+\inf\supp g,
\qquad
\sup\supp f\ast g
=\sup\supp f+\sup\supp g.
\end{equation}
The higher-dimensional reformulation
by Lions
\cite{MR0043254}
states
that
for $f,\,g\in\mathscr{E}'(\R^n)$,
the convex hull of the support
of $f\ast g$
is equal to the sum of convex hulls of
supports of $f$ and $g$.
Different proofs of the Titchmarsh convolution theorem
are contained in
\cite[Chapter VI]{MR617913} (Real Analysis style),
\cite[Theorem 4.3.3]{MR1065136} (Harmonic Analysis style),
and \cite[Lecture 16, Theorem 5]{MR1400006} (Complex Analysis style).
Here we give a version of the Titchmarsh Theorem
which is valid for distributions
supported in $n>1$ small intervals
of the circle $\torus=\R\mod 2\pi$.
For brevity,
we only give the result for distributions
supported in two small intervals,
which suffices for the applications
in this paper;
a more general version is proved in \cite{titchmarsh-circle}.

First,
we note that there are zero divisors with respect
to the convolution on the circle.
Indeed, for any two distributions $f$, $g\in\mathscr{E}'(\torus)$
one has
\begin{equation}\label{nt}
(f+S\sb{\pi}f)\ast(g-S\sb{\pi}g)
=f\ast g+S\sb{\pi}(f\ast g)-S\sb{\pi}(f\ast g)
-f\ast g=0.
\end{equation}
Above, $S\sb{y}$, $y\in\torus$,
is the shift operator, defined on $\mathscr{E}'(\torus)$ by
\begin{equation}\label{def-shift-1}
\big(S\sb{y}f\big)(\omega)=f(\omega-y),
\end{equation}
where the above relation
is understood
in the sense of distributions.
Yet, the cases when the Titchmarsh convolution theorem
``does not hold'' (in a certain na\"ive form)
could be specified.
This leads to a version of the Titchmarsh convolution theorem
for distributions on the circle
(Theorem~\ref{theorem-titchmarsh-circle-1} below).

\medskip

Let us start with the following problem
which illustrates our methods.

\medskip

\noindent
{\it Problem.}
Let $f$, $g\in\mathscr{E}'(\R)$ be such that
$\inf\supp f=0$, $\inf\supp g=0$.
If
$(f\ast f)\at{(-\infty,A)}=(g\ast g)\at{(-\infty,A)}$,
for some $A>0$,
then on $(-\infty,A)$ either $f=g$, or $f=-g$.

\medskip
\noindent
{\it Solution.}
Since
$
\inf\supp(f-g)\ast(f+g)
=\inf\supp(f\ast f-g\ast g)
\ge A,
$
the Titchmarsh convolution theorem \eqref{titchmarsh-theorem}
states that
\begin{equation}\label{fga}
\inf\supp(f-g)+\inf\supp(f+g)\ge A.
\end{equation}
One concludes that 
either
$\inf\supp(f-g)\ge A/2$, or $\inf\supp(f+g)\ge A/2$.
In the former case,
we have $f=g$ on $(-\infty,A/2)$,
hence $\inf\supp(f+g)=\inf\supp f=0$.
Therefore \eqref{fga} shows that $\inf\supp(f-g)\ge A$,
hence $f=g$ on $(-\infty,A)$.
Similarly, if $\inf\supp(f+g)\ge A/2$,
one concludes that
$\inf\supp(f+g)\ge A$,
thus $f=-g$ on $(-\infty,A)$.

\medskip

For $I\subset\torus$,
denote
\[
\mathscr{R}\sb 2(I)
=\bigcup\limits\sb{k\in\Z\sb 2}S\sb{\pi k}I,
\qquad
\mbox{where}
\quad
\Z\sb 2=\Z\mod 2.
\]

\begin{theorem}[Titchmarsh theorem for distributions on the circle]
\label{theorem-titchmarsh-circle-1}
Let $f,\,g\in\mathscr{E}'(\torus)$.
Let
$I,\,J\subset\torus$ be two closed intervals 
such that
$
\supp f\subset \mathscr{R}\sb 2(I),
$
$
\supp g\subset \mathscr{R}\sb 2(J),
$
and assume that
there is no closed interval
$I'\subsetneq I$
such that
$\supp f\subset \mathscr{R}\sb 2(I')$
and no closed interval $J'\subsetneq J$
such that
$\supp g\subset \mathscr{R}\sb 2(J')$.

Assume that
\begin{equation}\label{i-j-small}
\abs{I}+\abs{J}<\pi.
\end{equation}
Let $K\subset I+J\subset\torus$
be a closed interval such that
$
\supp f\ast g\subset \mathscr{R}\sb 2(K).
$
Then $\lambda:=\inf K-\inf I-\inf J>0$
if and only if there is
$\sigma\in\{\pm 1\}$
such that
\begin{equation}\label{asbs}
\big(
f+\sigma S\sb{\pi}f
\big)
\At{(\sup I-\pi,\inf I+\lambda)}=0,
\qquad
\big(
g-\sigma S\sb{\pi}g
\big)
\At{(\sup J-\pi,\inf J+\lambda)}=0.
\end{equation}
Similarly,
one has $\rho:=\sup I+\sup J-\sup K>0$
if and only if there is
$\sigma\in\{\pm 1\}$
such that
\begin{equation}\label{asbss}
\big(f+\sigma S\sb{\pi}f\big)
\At{(\sup I-\rho,\inf I+\pi)}=0,
\qquad
\big(g-\sigma S\sb{\pi}g\big)
\At{(\sup J-\rho,\inf J+\pi)}=0.
\end{equation}
\end{theorem}

\begin{remark}
While $\mathscr{E}'(\torus)=\mathscr{D}'(\torus)$,
we use the notation $\mathscr{E}'(\torus)$
for the consistency with the
requirements of the standard Titchmarsh convolution
theorem \eqref{titchmarsh-theorem}.
\end{remark}

Informally, we could say that
the intervals $I$, $J$, and $K$
play the role similar to
``convex hulls''
of supports
of $f,\,g,\,f\ast g\in\mathscr{E}'(\torus)$.
If $K\subsetneq I+J$
(certain ``na\"ive form of the Titchmarsh convolution theorem''
is not satisfied),
then both $f$ and $g$ satisfy certain
symmetry properties
on $\mathscr{R}\sb 2(U)$ and on $\mathscr{R}\sb 2(V)$,
where open non-intersecting
intervals $U$ and $V$
can be chosen so that
$U\cup K\cup V\supset I+J$.

\begin{proof}[Proof of Theorem~\ref{theorem-titchmarsh-circle-1}]
We will only prove \eqref{asbs};
the relations \eqref{asbss}
follow by applying the reflection to $\torus$
and using the first part of the theorem.

The ``if'' part
of the theorem
is checked by direct computation.
Let
$f\in\mathscr{E}'(I\cup S\sb{\pi}I)$,
where $I\subset\torus$, $\abs{I}<\pi/2$,
$g\in\mathscr{E}'(J\cup S\sb{\pi}J)$,
where $J\subset\torus$, $\abs{J}<\pi/2$,
and assume that
$f=\pm S\sb\pi f$ on $(\sup I-\pi,\inf I+\lambda)$,
$g=\mp S\sb\pi g$ on $(\sup J-\pi,\inf J+\lambda)$.
Then, as in \eqref{nt},
\begin{eqnarray}
&&
(f\ast g)\at{(\sup I+\sup J-2\pi,\inf I+\inf J+\lambda)}
\nonumber
\\
&&
=f\at{(\sup I-\pi,\inf I+\lambda)}
\ast g\at{(\sup J-\pi,\inf J+\lambda)}
+(S\sb\pi f)\at{(\sup I-\pi,\inf I+\lambda)}
\ast(S\sb\pi g)\at{(\sup J-\pi,\inf J+\lambda)}
\nonumber
\\
&&
=f\at{(\sup I-\pi,\inf I+\lambda)}
\ast g\at{(\sup J-\pi,\inf J+\lambda)}
-f\at{(\sup I-\pi,\inf I+\lambda)}
\ast g\at{(\sup J-\pi,\inf J+\lambda)}=0.
\nonumber
\end{eqnarray}

Let us now prove the ``only if'' part.
One has
$\supp f\subset \mathscr{R}\sb 2(I)$,
$\supp g\subset \mathscr{R}\sb 2(J)$,
$\supp f\ast g\subset\mathscr{R}\sb 2(K)\subset\mathscr{R}\sb 2(I+J)$.
Due to the restriction
\eqref{i-j-small},
the sets
$\mathscr{R}\sb 2(I)$, $\mathscr{R}\sb 2(J)$,
and $\mathscr{R}\sb 2(I+J)$
each consist of $n$ non-intersecting intervals.
For $j\in\Z\sb 2$,
let us set
$f\sb j=(S\sb{\pi j}f)\at{I}\in\mathscr{E}'(I)$,
$g\sb j=(S\sb{\pi j}g)\at{J}\in\mathscr{E}'(J)$,
$h\sb j=\big(S\sb{\pi j}(f\ast g)\big)\at{K}\in\mathscr{E}'(I+J)$;
then,
for $j\in\Z\sb 2$,
\begin{eqnarray}\label{hfg}
h\sb j
=
\big(S\sb{\pi j}(f\ast g)\big)\At{I+J}
=
\!\!
\sum\sb{\stackrel{\scriptstyle k+l=j\!\!\!\mod 2}{\scriptstyle k,\,l\in\Z\sb 2}}
(S\sb{\pi k}f)\at{I}
\ast(S\sb{\pi l}g)\at{J}
=
\!\!
\sum\sb{\stackrel{\scriptstyle k+l=j\!\!\!\mod 2}{\scriptstyle k,\,l\in\Z\sb 2}}
f\sb k\ast g\sb l.
\end{eqnarray}
Using the relation \eqref{hfg},
for $\sigma=\pm 1$
we have:
\begin{equation}\label{wh}
(f\sb 0+\sigma f\sb 1)
\ast
(g\sb 0+\sigma g\sb 1)
=
(f\sb 0\ast g\sb 0+f\sb 1\ast g\sb 1)
+
\sigma
(f\sb 0\ast g\sb 1+f\sb 1\ast g\sb 0)
=
h\sb 0+\sigma h\sb 1.
\end{equation}
Applying the Titchmarsh convolution theorem \eqref{titchmarsh-theorem}
to this relation,
we obtain:
\begin{equation}\label{isf-isg}
\inf\supp
(f\sb 0+\sigma f\sb 1)
+
\inf\supp
(g\sb 0+\sigma g\sb 1)
=
\inf\supp
(h\sb 0+\sigma h\sb 1)
\ge\inf K,
\end{equation}
where we took into account that
$\min\sb{j\in\Z\sb 2}\inf\supp h\sb j\ge\inf K$.
Let us pick $\sigma\in\{\pm 1\}$
such that
\begin{equation}\label{sasa}
\inf\supp
(g\sb 0+\sigma g\sb 1)
=
\min\limits\sb{j\in\Z\sb 2}
\inf\supp g\sb j
=\inf J.
\end{equation}
For this value of $\sigma$,
\eqref{isf-isg} yields:
\[
\inf\supp
(f\sb 0+\sigma f\sb 1)
\ge\inf K-\inf J=\inf I+\lambda,
\]
proving
the first relation in \eqref{asbs}.
The second relation in \eqref{asbs} follows
due to the symmetric roles of $f$ and $g$.
The opposite sign (negative sign at $\sigma$)
follows from
\eqref{sasa}.
\end{proof}

Here is the
convolution theorem for powers of a distribution.

\begin{theorem}[Titchmarsh theorem for powers
of a distribution on the circle]
\label{theorem-titchmarsh-powers}
Let $f\in\mathscr{E}'(\torus)$.
Let $I\subset\torus$ be a closed interval
such that $\supp f\subset \mathscr{R}\sb 2(I)$,
and assume that there is no $I'\subsetneq I$
such that
$\supp f\subset \mathscr{R}\sb 2(I')$.

Assume that $\abs{I}<\frac{\pi}{p}$, for some $p\in\N$.
Then the smallest closed interval
$K\subset p I$ such that
$\supp f\sp{\ast p}\subset \mathscr{R}\sb 2(K)$
is $K=p I$.
\end{theorem}

Above, we used the notations
$p I=\underbrace{I+\dots+I}\sb{p}$
and
$f\sp{\ast p}=\underbrace{f\ast\dots\ast f}\sb{p}$.

Let us notice that the proof
of Theorem~\ref{theorem-titchmarsh-powers}
for the case $p=2$
immediately follows from Theorem~\ref{theorem-titchmarsh-circle-1}.
(For example, the relations \eqref{asbs} with $f=g$
are mutually contradictory unless $\lambda=0$.)
By induction,
this also gives the proof for $p=2^N$, with any $N\in\N$,
and then one can deduce the statement of
Theorem~\ref{theorem-titchmarsh-powers}
for any $p\le 2^N$, but under the condition
$\abs{I}<\frac{\pi}{2^N}$,
which is stronger than
$\abs{I}<\frac{\pi}{p}$.

\begin{proof}[Proof of Theorem~\ref{theorem-titchmarsh-powers}]
One has
$\supp f\sp{\ast p}\subset\mathscr{R}\sb 2(p I)$.
Due to the smallness of $I$,
both $\mathscr{R}\sb 2(I)$ and $\mathscr{R}\sb 2(p I)$
are collections of $n$ non-intersecting
intervals.
Define
\[
f\sb j:=(S\sb{\pi j}f)\at{I}\in\mathscr{E}'(I),
\qquad
h\sb j:=
\big(S\sb{\pi j}(f\sp{\ast p})\big)
\at{I}\in\mathscr{E}'(I).
\]
Then
\begin{eqnarray}\label{hff}
&h\sb j
=\big(S\sb{\pi j}(f\sp{\ast p})\big)\At{p I}
=
\sum\limits\sb{
\stackrel
{\scriptstyle j\sb 1+\dots+j\sb p=j\!\!\mod 2}
{\scriptstyle j\sb 1,\,\dots\,j\sb p\in\Z\sb 2}}
(S\sb{\pi j\sb 1}f)\at{I}\ast
\dots\ast(S\sb{\pi j\sb p}f)\at{I}
\nonumber
\\
&
=
\sum\limits\sb{
\stackrel
{\scriptstyle j\sb 1+\dots+j\sb p=j\!\!\mod 2}
{\scriptstyle j\sb 1,\,\dots,\,j\sb p\in\Z\sb 2}}
f\sb{j\sb 1}\ast\dots\ast f\sb{j\sb p},
\qquad
j\in\Z\sb 2.
\end{eqnarray}
Taking into account \eqref{hff},
for $\sigma=\pm 1$
one has:
\begin{equation}\label{ffhj}
(f\sb 0+\sigma f\sb 1)\sp{\ast p}
=
\sum\sb{j\in\Z\sb 2}
\sigma^j
\Big[
\sum\sb{j\sb 1+\dots+j\sb p=j\!\!\mod 2}
f\sb{j\sb 1}\ast\dots\ast f\sb{j\sb n}
\Big]
=h\sb 0+\sigma h\sb 1.
\end{equation}
Now we apply the Titchmarsh convolution theorem
to \eqref{ffhj}, getting
\[
p\inf\supp
(f\sb 0+\sigma f\sb 1)
=\inf\supp
(h\sb 0+\sigma h\sb 1).
\]
There is $\sigma=\pm 1$ such that
$
\inf\supp(f\sb 0+\sigma f\sb 1)
=\min\limits\sb{j\in\Z\sb 2}
\inf\supp f\sb j$;
for this value of $\sigma$,
\[
p\min\limits\sb{j\in\Z\sb 2}\inf\supp f\sb j
=\inf\supp
(h\sb 0+\sigma h\sb 1)
\ge\min\limits\sb{j\in\Z\sb 2}\inf\supp h\sb j.
\]
On the other hand,
\eqref{hff} yields the inequalities
$\inf\supp h\sb j\ge p\min\limits\sb{k\in\Z\sb 2}\inf\supp f\sb k$,
for any $j\in\Z\sb 2$.
It follows that
\[
\min\limits\sb{j\in\Z\sb 2}\inf\supp h\sb j=p\min\limits\sb{j\in\Z\sb 2}\inf\supp f\sb j
\]
and similarly
\[
\max\limits\sb{j\in\Z\sb 2}\sup\supp h\sb j=p\max\limits\sb{j\in\Z\sb 2}\sup\supp f\sb j.
\]
\end{proof}

\medskip

Denote
$f\sp\sharp(\omega)=\overline{f(-\omega)}$.
Let $f\in\mathscr{E}'(\torus)$
and let
$I\subset\torus$
be a closed interval such that
$\supp f\subset \mathscr{R}_2(I)$.
Assume that
there is no
closed interval $I'\subsetneq I$ such that
$\supp f\subset \mathscr{R}_2(I')$.

\begin{theorem}
\label{theorem-titchmarsh-weird}
If
$I\subset(-\pi/2,\pi/2)$
and
$\abs{I}<\pi/2$,
then
the inclusion
$
\supp f\ast f\sp\sharp\subset\{0;\pi\}
$
implies that
$
\supp f\subset\{\inf I;\sup I;\pi+\inf I;\pi+\sup I\}.
$
Moreover, there are
distributions $\mu$, $\nu\in\mathscr{E}'(\torus)$,
each supported at a point,
such that
\begin{equation}\label{fab1}
f=\mu+S\sb{\pi}\mu+\nu-S\sb{\pi}\nu.
\end{equation}
\end{theorem}

\begin{proof}
If $I$ consists of one point,
$I=\{p\}\subset(-\pi/2,\pi/2)$,
then
$\supp f=\mathscr{R}\sb 2(p)=\{p;\pi+p\}$,
and \eqref{fab1} holds with
\[
\mu=\frac{f+S\sb\pi f}{2}\At{I},
\quad
\nu=\frac{f-S\sb\pi f}{2}\At{I}.
\]
Now we assume that
$\abs{I}>0$.
Define $J=-I$ and $K=\{0\}\subset I+J$.
Then
$\supp f\sp\sharp\subset\mathscr{R}\sb 2(J)$
and there is no $J'\subsetneq J$
such that
$\supp f\sp\sharp\subset\mathscr{R}\sb 2(J')$.
According to the conditions of the theorem,
$\supp f\ast f\sp\sharp\subset\mathscr{R}\sb 2(K)$;
hence, one has:
\begin{equation}
\lambda:=\inf K-\inf I-\inf J=\sup I-\inf I=\abs{I}>0.
\label{ff1}
\end{equation}
Applying Theorem~\ref{theorem-titchmarsh-circle-1}
to \eqref{ff1}, we conclude that
there is $\sigma\in\{\pm 1\}$
such that
\begin{equation}\label{asdf1}
(f+\sigma S\sb{\pi}f)\at{(\sup I-\pi,\sup I)}=0
\end{equation}
and also
$\inf\supp(f\sp\sharp+\sigma S\sb{\pi}f\sp\sharp)\at{(-\frac{\pi}{2},\frac{\pi}{2})}
=
-\sup I$;
this last relation implies that
\begin{equation}\label{asdf1a}
\sup\supp(f+\sigma S\sb{\pi}f)\at{(-\frac{\pi}{2},\frac{\pi}{2})}
=
\sup I.
\end{equation}
Also, by Theorem~\ref{theorem-titchmarsh-circle-1},
there is $\sigma'\in\{\pm 1\}$
such that
$
(f\sp\sharp+\sigma' S\sb{\pi}f\sp\sharp)\at{(-\inf I-\pi,-\inf I)}=0,
$
hence
\begin{equation}\label{asdf3a}
(f+\sigma' S\sb{\pi}f)\at{(\inf I,\inf I+\pi)}=0.
\end{equation}
Comparing \eqref{asdf1a} with \eqref{asdf3a},
we conclude that $\sigma'=-\sigma$;
then \eqref{asdf1} and \eqref{asdf3a}
allow us to conclude that both $f$ and $S\sb\pi f$
vanish on $(\inf I,\sup I)$,
hence
\[
\supp f\subset\{\inf I;\sup I;\pi+\inf I;\pi+\sup I\}.
\]
By \eqref{asdf1} and \eqref{asdf3a},
if $\sigma=1$,
the relation \eqref{fab1} holds
with $\mu=f\at{(\inf I,\pi/2)}$
and $\nu=f\at{(-\pi/2,\sup I)}$.
If instead $\sigma=-1$,
the relation \eqref{fab1} holds
with $\mu=f\at{(-\pi/2,\sup I)}$
and $\nu=f\at{(\inf I,\pi/2)}$.
\end{proof}



\def\cprime{$'$} \def\cprime{$'$} \def\cprime{$'$} \def\cprime{$'$}
  \def\cprime{$'$} \def\cprime{$'$} \def\cprime{$'$} \def\cprime{$'$}
  \def\cprime{$'$} \def\cprime{$'$} \def\cprime{$'$}

\end{document}